\documentclass{amsart}

\usepackage[margin=4.2cm]{geometry}
\usepackage{amsmath,amssymb,amsfonts,amsthm}
\usepackage[usenames,dvipsnames]{color}
\usepackage{mathtools}
\usepackage[all,cmtip]{xy}
\usepackage{enumerate}
\usepackage[colorlinks=true,linkcolor=red,citecolor=blue]{hyperref}
\usepackage{color,epsfig}
\usepackage{axodraw4j}

\newtheorem{theorem}{Theorem}[section]
\newtheorem{proposition}[theorem]{Proposition}
\newtheorem{lemma}[theorem]{Lemma}
\newtheorem{corollary}[theorem]{Corollary}

\theoremstyle{definition}
\newtheorem{definition}[theorem]{Definition}

\theoremstyle{remark}
\newtheorem{remark}[theorem]{Remark}
\newtheorem{example}[theorem]{Example}

\newcommand{\PP}{\mathbb{P}}
\newcommand{\MT}{\mathcal{MT}}
\newcommand{\open}{\mathrm{open}}
\newcommand{\closed}{\mathrm{closed}}

\newcommand{\QQ}{\mathbb{Q}}

\newcommand{\RR}{\mathbb{R}}
\newcommand{\CC}{\mathbb{C}}
\newcommand{\ZZ}{\mathbb{Z}}

\newcommand{\Pe}{\mathcal{P}}
\newcommand{\mm}{\mathfrak{m}}

\newcommand{\Or}{\mathcal{O}}

\newcommand{\s}{\mathsf{s}}
\newcommand{\To}{\longrightarrow}

\newcommand{\A}{\mathcal{A}}

\newcommand{\B}{\mathrm{B}}
\newcommand{\dR}{\mathrm{dR}}
\renewcommand{\min}{\backslash}
\renewcommand{\c}{\mathrm{c}}

\DeclareMathOperator{\sv}{\mathsf{sv}}
\DeclareMathOperator{\ev}{ev}
\DeclareMathOperator{\id}{id}
\DeclareMathOperator{\ren}{ren}
\DeclareMathOperator{\lf}{lf}

\title{Single-valued integration and superstring amplitudes in genus zero}

\author{Francis Brown}
\address{All Souls College, Oxford, Oxford OX1 4AL, UK}
\email{francis.brown@all-souls.ox.ac.uk}

\author{Cl\'{e}ment Dupont}
\address{Institut Montpelli\'{e}rain Alexander Grothendieck, Université de Montpellier, CNRS, Montpellier, France}
\email{clement.dupont@umontpellier.fr}

\date{}

\begin{document}

\maketitle

\begin{abstract}   We study open and closed string amplitudes at tree-level in string perturbation theory  using the methods of  single-valued integration which were developed in the prequel to this paper \cite{BD1}. 
Using dihedral coordinates on the moduli spaces of curves of genus zero with marked points, we define a canonical regularisation of both open and closed string perturbation amplitudes at tree level, and deduce that they admit a Laurent expansion in Mandelstam variables whose coefficients are multiple zeta values (resp. single-valued multiple zeta values). Furthermore, we prove the existence of a motivic Laurent expansion whose image under the period map is the open string expansion, and whose image under the single-valued period map is the closed string expansion. This proves the recent conjecture of Stieberger that closed string amplitudes are the single-valued projections of (motivic lifts of) open string amplitudes.

Finally, applying a variant of the single-valued formalism for cohomology with coefficients yields the KLT formula expressing closed string amplitudes as quadratic expressions in open string amplitudes.

\end{abstract}

\section{Introduction}

\subsection{The beta function}
As motivation for our results, it is instructive to consider  the special  case of the Euler beta function (Veneziano amplitude \cite{Veneziano1968}) 
  \begin{equation} \label{introbetareal}
\beta(s,t) = \int_{0}^1  x^s (1-x)^t  \frac{dx}{x(1-x)} = \frac{\Gamma(s)\Gamma(t)}{\Gamma(s+t)}  \ \cdot 
\end{equation}
The integral converges for $\mathrm{Re}( s)>0$, $\mathrm{Re} ( t)  >  0 $.
Less familiar  is the  complex beta function   (Virasoro--Shapiro  amplitude \cite{virasoro, shapiro}), given by 
\begin{equation}  \label{introbetaC}
\beta_{\CC} (s,t) = -\frac{1}{2\pi i}\int_{\PP^1(\CC)}   |z|^{2s} |1-z|^{2t}  \frac{dz \wedge d \overline{z}}{|z|^2 |1-z|^2} = \frac{\Gamma(s)\Gamma(t)\Gamma(1-s-t)}{\Gamma(s+t) \Gamma(1-s)\Gamma(1-t)}\ \cdot 
\end{equation} 
The integral  converges  in the region $\mathrm{Re}( s) >0$, $\mathrm{Re} ( t)  >  0$ , $\mathrm{Re}( s +t )<1$.

The beta function  admits the following Laurent expansion
\begin{equation} \label{introbetastexpansion} \beta(s,t) =  \Big( \frac{1}{s}+\frac{1}{t} \Big)  \exp  \Big( \sum_{n \geq 2}  \frac{(-1)^{n-1} \zeta(n)}{n}  \big((s+t)^n -s^n -t^n\big)\Big)\ , 
\end{equation}
and  the complex beta function  has a very similar expansion
\begin{equation} \label{introcomplexbetaexpansion} \beta_{\CC}(s,t) =  \Big( \frac{1}{s}+\frac{1}{t} \Big)  \exp  \Big( \sum_{ \substack{ n \geq 2 \\ n \textrm{ odd}}}  \frac{(-1)^{n-1} 2 \, \zeta(n)}{n}  \big((s+t)^n -s^n -t^n\big)\Big)\ .
\end{equation}
It is important to note that these Laurent expansions are taken at the point $(s,t)=(0,0)$ which lies outside the domain of convergence of the respective integrals. 

The coefficients in \eqref{introcomplexbetaexpansion} can be expressed as  `single-valued' zeta values which satisfy: $$\zeta^{\sv}(2n) =0  \quad \hbox{  and } \quad  \zeta^{\sv}(2n+1) = 2 \, \zeta(2n+1) $$
for $n \geq 1.$ The Laurent expansion \eqref{introcomplexbetaexpansion} can thus be viewed as a `single-valued'  version of \eqref{introbetastexpansion}.  To make this precise, we
  define  a \emph{motivic beta function}
 $\beta^{\mm}(s,t)$
 which is a formal  Laurent expansion  in motivic zeta values:
  \begin{equation} \label{introbetamot} \beta^{\mm} (s,t) =  \Big( \frac{1}{s}+\frac{1}{t} \Big)  \exp  \Big( \sum_{n \geq 2}  \frac{(-1)^{n-1} \zeta^{\mm}(n)}{n}  \big((s+t)^n -s^n -t^n\big)\Big)\ ,
  \end{equation}
  whose coefficients $\zeta^{\mm}(n)$ are motivic periods of the  cohomology of  the  moduli spaces of curves $\overline{\mathcal{M}}_{0,n+3}$ relative to certain boundary divisors. It has a  de Rham version $\beta^{\mathfrak{m},\dR}(s,t)$, obtained from it by applying the de Rham projection term by term. 
 One has
$$\beta(s,t) = \mathrm{per} \, ( \beta^{\mm}(s,t) ) \qquad\hbox{ and } \qquad \beta_{\CC}(s,t) = \s \, (\beta^{\mathfrak{m},\dR}(s,t))$$
where $\s$ is the single-valued period map which is defined on de Rham motivic periods.
We can therefore conclude  that the Laurent expansions of $\beta(s,t)$ and $\beta_{\CC}(s,t)$ are deduced from a single object, namely, the motivic beta function \eqref{introbetamot}.  

The first objective of this paper is to generalise all of the above for general string perturbation amplitudes at tree-level.

\subsubsection*{Cohomology with coefficients} 
There is another sense in which \eqref{introbetareal}  is a single-valued version of \eqref{introbetaC} that does not involve expanding in $s, t$ and uses  cohomology with coefficients.

  For  generic  values of $s,t$ (i.e.,   $s ,t,s+t \notin \ZZ$),  it is known how to interpret  $\beta(s,t)$  as  a period of a canonical pairing between algebraic de Rham cohomology and locally finite Betti (singular) homology:
\begin{equation}\label{eq: cohomology groups coefficients intro}
H^1(X,   \nabla_{s,t} ) \qquad \hbox{ and } \qquad H_1^{\mathrm{lf}} ( X(\CC), \mathcal{L}_{-s,-t})\ ,
\end{equation}
where $X=\PP^1 \backslash \{0,1,\infty\}$, $\nabla_{s,t}$ is the  integrable connection 
$$\nabla_{s,t} = d + s \, d\log x  +  t \, d\log(1-x) $$ 
on the rank one algebraic vector bundle $\Or_X$, and $\mathcal{L}_{-s,-t}$ is the rank one local system generated by $x^s(1-x)^t$, which is a  flat section of $\nabla_{-s,-t}=\nabla_{s,t}^{\vee}$ (see Example \ref{example: M04withcoeffs}). An important feature of this situation is Poincar\'{e} duality which gives rise to de Rham and Betti pairings between \eqref{eq: cohomology groups coefficients intro} for $(s,t)$ and for $(-s,-t)$. Compatibility between these pairings amounts to the following functional equation for the beta function:
\begin{equation}\label{eq: functional equation beta intro}
2\pi i\left(\frac{1}{s}+\frac{1}{t}\right) = \beta(s,t)\beta(-s,-t)\left(\frac{2}{i}\frac{\sin(\pi s)\sin(\pi t)}{\sin(\pi(s+t))}\right)\ ,
\end{equation}
where the factor in brackets on the left-hand side is the de Rham pairing of $\frac{dx}{x(1-x)}$ with itself and the factor in brackets on the right-hand side is the inverse of the Betti pairing of $(0,1)\otimes x^{s}(1-x)^{t}$ with $(0,1)\otimes x^{-s}(1-x)^{-t}$.

As in the case of relative cohomology with constant coefficients studied in \cite{BD1}, there exists a single-valued formalism for cohomology with coefficients in this setting for which we give an integral formula (Theorem \ref{theoremsvforlogwithcoeffs}). This formula implies that $\beta_{\CC}(s,t)$ is a  single-valued period of \eqref{eq: cohomology groups coefficients intro}, which amounts to the equality
\begin{equation} \label{introbetaDC1} \beta_{\CC}(s,t) = - \Big(  \frac{1}{s} + \frac{1}{t} \Big) \beta(s,t) \beta(-s,-t)^{-1}
\end{equation}
and proves the second equality in \eqref{introbetaC}.

Applying the functional equation \eqref{eq: functional equation beta intro} we then get the following `double copy formula' expressing a single-valued period as a quadratic expression in ordinary periods:
\begin{equation}  \label{introbetaDC2} 
\beta_{\CC}(s,t) = -\frac{1}{2\pi i} \Big( \frac{2}{i} \frac{ \sin (\pi s) \sin (\pi t)}{\sin (\pi (s+t))} \Big) \beta(s,t)^2 \ . 
\end{equation}
This formula is an  instance of the  Kawai--Lewellen--Tye (KLT) relations \cite{klt}.\medskip

 In conclusion, there are \emph{three} different ways to deduce the complex beta function from the classical beta function: via \eqref{introbetaDC1} or the double copy formula \eqref{introbetaDC2}, or by applying the single valued period map term by term in its Laurent expansion.
 
 \subsection{General string amplitudes at tree level}
 The general $N$-point genus zero open string amplitude is formally  written as an integral  which generalises  \eqref{introbetareal}:
$$I^{\open}(\omega, \underline{s}) = \int_{0<t_1<\cdots< t_{N-3}<1}   \prod_{1\leq i<j\leq N-3} (t_j-t_i)^{s_{ij}}\,   \omega$$
where $\omega$ is a meromorphic form with certain logarithmic singularities (see \S\ref{sect: StringAmpSimplicial}), and 
$ \underline{s} =\{s_{ij}\}$ are Mandelstam variables satisfying momentum conservation equations \eqref{MC}.

It turns out that one can write the closed string amplitudes in the form 
$$I^{\closed}(\omega, \underline{s}) = (2\pi i)^{3-N} \int_{\CC^{N-3}}   \prod_{1\leq i<j\leq N-3} |z_j-z_i|^{2s_{ij}}   \,   \nu_S\wedge\overline{\omega}\ .$$
Later we shall rewrite the domain of integration as the complex points of the compactified moduli space of curves of genus $0$ with $N$ ordered marked points. Then,   the  form
$$\nu_{S} =  (-1)^{\frac{N(N-1)}{2}} \prod_{i=0}^{N-3} (t_{i+1} - t_{i})^{-1}  dt_1\wedge \cdots \wedge dt_{N-3} \hspace{1cm} (t_0=0,t_{N-2}=1) $$
is logarithmic and  has poles along the boundary of the domain of integration of the open string amplitude. It is in fact the image of the homology class of this domain  under the map $c_0^{\vee}$ defined in \cite{BD1}.

   The first task is to interpret the open and closed string amplitudes rigourously as integrals over the moduli space of curves $\mathcal{M}_{0,N}$. An immediate problem is that the poles of the integrand lie along divisors which do not cross normally. 
 Using a  cohomological interpretation    of  the   momentum conservation equations in \S\ref{sect: MC}, we show how to resolve the singularities of the integral by rewriting it  in terms of dihedral coordinates. These are certain cross-ratios  $u_c$ in the $t_i$,  indexed by chords $c$ in an $N$-gon, whose zero loci form a normal crossing divisor.
  Thus, for example, we write in \S\ref{sect: StringInDihedral}:
  $$I^{\open}(\omega, \underline{s}) = \int_{X^{\delta}} \left( \prod_c u_c^{s_c}  \right) \,  \omega$$
  where $X^{\delta}$ is the locus where all $0< u_c <1$ and the $s_c$ are linear combinations of the $s_{ij}$.  This rewriting of the amplitude evinces the divergences of the integrand and the potential poles in the Mandelstam variables. A similar expression holds for the closed string amplitude, in which $u_c^{s_c}$ is replaced by $|u_c|^{2s_c}$  and in which the domain of integration is replaced by the complex points of the Deligne--Mumford compactification  $\overline{\mathcal{M}}_{0,N}$.

  By an inclusion-exclusion procedure  close in spirit 
  to  renormalisation\footnote{Strictly speaking, this is not renormalisation in the physical sense since there are no ultraviolet divergences in the perturbative superstring amplitudes.} of algebraic integrals in perturbative quantum field theory \cite{brownkreimer}, we can explicitly remove all poles using  properties of dihedral coordinates and the combinatorics of chords. The renormalisation fundamentally hinges on  special properties of morphisms between moduli spaces which play the role of counter-terms and are described in  \S\ref{sect: Renorm}. 
  \begin{theorem}  There is a canonical `renormalisation' 
   $$I^{\mathrm{open}}(\omega,\underline{s})= \sum_{J}  \frac{1}{s_J}   \int_{X_J}   \Omega_J^{\mathrm{ren}} \qquad \hbox{ where }\quad s_J= \prod_{c\in J} s_c$$
 indexed by sets $J$  of non-crossing chords in an $N$-gon, where $\Omega_J^{\mathrm{ren}}$ is explicitly defined. The integrals on the right-hand side are convergent around $s_{ij}=0$. They are by definition products of convergent integrals over domains $X^\delta$ of various dimensions. 
     \end{theorem} 
  This theorem provides an interpretation of the poles in the Mandelstam variables $s_{ij}$ in terms of the poles of $\omega$ (see for example \eqref{residueinMandel}). 
  A similar statement holds for the closed string amplitude (Theorem \ref{thmrenormclosed}). 
 Having extended the range of convergence of the integrals using the previous theorem, we are then in a position to take a Laurent expansion around $s_{ij}=0$. The   coefficients  in this  expansion, which are canonical,  are products of convergent integrals of the form:
 $$\int_{X^{\delta}}  \left(\prod_c \log^{n_c}(u_c)  \right)  \eta \ $$
 where the product ranges over  chords  $c$ in a  polygon and $n_c\in\mathbb{N}$. 
  We then  show how to interpret these integrals as periods of moduli spaces $\mathcal{M}_{0,N'}$ for   larger $N'$ by replacing the logarithms with integrals (non-canonically).  A key, and subtle point, is that  they are integrals over a domain $X^{\delta'}$  of  a   global regular form with   logarithmic singularities.  We can therefore interpret the previous integrals as motivic periods of universal moduli space motives, and hence define a motivic version of the string amplitude. 
  
  \begin{theorem} There is a motivic string amplitude:
  $$I^{\mm}(\omega, \underline{s}) \quad \in \quad \Pe^{\mm,+}_{\mathcal{MT}(\ZZ)}  ((s_{c}))$$
  which is  a Laurent expansion with coefficients in the ring of motivic multiple zeta values of homogeneous weight. Its period is the open string amplitude
  $$  \mathrm{per} \, I^{\mm}(\omega, \underline{s}) =  I^{\mathrm{open}}(\omega, \underline{s})\   .  $$
It follows that  the  coefficients of $I^{\mathrm{open}}(\omega, \underline{s})$ are multiple zeta values.
    \end{theorem} 
  
  The first statement has been used implicitly in \cite{schlottererstiebergermotivic, schlottererschnetz} by assuming the period conjecture for multiple zeta values. The fact that the Laurent coefficients are multiple zeta values  is folklore. 
 A  subtlety  in the previous theorem is that the motivic lift $I^{\mm}(\omega, \underline{s})$ is a priori not unique, as there are many possible ways to express the logarithms $\log(u_c)$ as integrals. We believe that  one could fix these choices if one wished. In any case, the period conjecture suggests that the motivic amplitude $I^{\mm}(\omega, \underline{s})$ is independent of these choices.
  
   By applying  the general theorems on single-valued integration proved in the prequel to this paper \cite{BD1} we deduce that the closed string amplitude is the single-valued version of the motivic amplitude. 
  
  \begin{theorem} \label{thm: introsvandclosed} Let $\pi^{\mathfrak{m},\dR}$ denote the de Rham projection map from effective mixed Tate motivic periods to de Rham motivic periods (which maps $\zeta^{\mm}$ to $\zeta^{\mm,\dR}$), and $\s$ the single-valued period map (which maps $\zeta^{\mm,\dR}$ to $\zeta^{\sv}$).  Then 
  $$I^{\closed} (\omega, \underline{s})  =  \s\, \pi^{\mathfrak{m},\dR} \,  I^{\mm}(\omega, \underline{s}).$$
  It follows that the coefficients in the canonical Laurent expansion of the closed string amplitudes are single-valued multiple zeta values. 
  \end{theorem} 
  This theorem, for periods (i.e., assuming the period conjecture) was conjectured in \cite{stiebergersvMZV, StiebergerTaylor} and proved independently by a very different method from our own in \cite{schlottererschnetz}. Since the first draft of this paper was written, yet another approach to computing the closed string amplitudes appeared in \cite{vanhovezerbini}. 
  An interesting consequence of Theorem \ref{thm: introsvandclosed} is that it suggests that  the space generated by closed string amplitudes might be  closed under the action of the de Rham motivic Galois group. It is important to note that the proof of the previous theorem, in contrast to the approach sketched in \cite{schlottererschnetz},  uses no prior knowledge of multiple zeta values or polylogarithms, and  merely involves  an application of our general results on single-valued integrals.

\subsubsection*{String amplitudes from the point of view of cohomology with coefficients and double copy formulae} 
In the final parts of this paper \S\S\ref{sect: CoeffBackground},\ref{sect: SectCohomcoeffs},  we  consider  the open string amplitude as  a period of the canonical pairing between algebraic de Rham cohomology with coefficients in a certain universal (Koba--Nielsen) algebraic vector bundle with connection, and locally finite homology with coefficients in its dual local system:
$$H^{N-3} ( \mathcal{M}_{0,N} , \nabla_{\underline{s}} ) \otimes  H^{\mathrm{lf}}_{N-3} ( \mathcal{M}_{0,N} , \mathcal{L}_{-\underline{s}} ) \To \CC \ . $$
As in the case of the beta function, Poincar\'{e} duality exchanges $\underline{s}$ and $-\underline{s}$ and leads to quadratic functional equations for open string amplitudes generalising \eqref{eq: functional equation beta intro}.

It is important to note that this interpretation of the open string amplitude, as a function of generic Mandelstam variables, is quite different from its interpretation as a Laurent series. After defining the single-valued period map, our main theorem 
 (Theorem \ref{theoremsvforlogwithcoeffs}) provides an interpretation of the closed string amplitudes  
 $I^{\closed}(\omega, \underline{s})$ as its
single-valued periods.   Theorem  \ref{theoremsvforlogwithcoeffs}   is in no way logically equivalent to the previous results since it is not obvious that the two notions of `single-valuedness', namely as a function of the $s_{ij}$, or term-by-term in their Laurent expansion, coincide.  The paper \cite{BDLauricella} provides yet  another connection between these two different cohomological points of view.

 As a consequence of Theorem  \ref{theoremsvforlogwithcoeffs}, we immediately  deduce an identity relating closed and open string amplitudes which involves the period matrix, its inverse, and the de Rham pairing. By the compatibility between the de Rham and Betti pairings, it in turn implies  a `double copy formula' which generalizes \eqref{introbetaDC2}. It expresses closed string amplitudes as quadratic expressions in open string amplitudes  but this time using the Betti intersection pairing (Corollary \ref{theorem2ndDC}).  Since Mizera has recently shown \cite{mizera} that the inverse transpose matrix of Betti intersection numbers coincides with the matrix of KLT coefficients, our formula  implies the  KLT relations. 

  Because our results for genus zero string amplitudes are in fact instances of a more general mathematical theory \cite{BD1}, valid for all algebraic varieties,  we expect that many of these results may carry through in some form to higher genera. It remains to be seen, in the light of  \cite{Witten}, if this has a chance of leading to a 
possible  double copy formalism for higher genus string amplitudes.

\subsection{Contents}  In \S \ref{sect: Dihed} we review the geometry of the moduli spaces $\mathcal{M}_{0,N}$, dihedral coordinates, and the forgetful maps which play a key role in the regularisation of singularities. In \S \ref{sect: StringAmp0} we recall the definitions of tree-level string amplitudes, their interpretation as moduli space integrals, and discuss their convergence. Section \S\ref{sect: Renorm} defines the `renormalisation' of string amplitudes via the subtraction of counter-terms  which uses the natural maps between moduli spaces. It  uses in an essential way the fact that the zeros of dihedral coordinates are normal crossing. Lastly, in \S\ref{sect: MotAmp} we construct the motivic amplitude and prove the main theorems using \cite{BD1}. 
The final sections \S\S\ref{sect: CoeffBackground},\ref{sect: SectCohomcoeffs} treat cohomology with coefficients as discussed above. In an appendix, we prove a folklore result that  the Parke--Taylor forms are a basis of cohomology with coefficients.

\subsection{Acknowledgements} 
This project has received funding from the European Research Council (ERC) under the European Union’s Horizon 2020 research and innovation programme (grant agreement no. 724638). Both authors  thank the IHES for hospitality. The second author was partially supported by ANR grant ANR-18-CE40-0017. This paper was initiated during the trimester ``Periods in number theory, algebraic geometry and physics'' which took place at the HIM Bonn in 2018,  to which both authors offer their thanks.  Many thanks to Andrey Levin, whose talk during this programme on the dilogarithm inspired this project, and also to Federico Zerbini for discussions. The second author also thanks Mike Falk for discussions on cohomology with coefficients and hyperplane arrangements.

\section{Dihedral coordinates and geometry of \texorpdfstring{$\mathcal{M}_{0,S}$}{M0,S}} \label{sect: Dihed}

Let $n \geq 0$  and let 
 $S$ be a set with $n+3$ elements, which we frequently identify with $\{1,\ldots,n+3\}$.
  Let $\mathcal{M}_{0,S}$ denote the moduli space of curves of genus zero with marked points labelled by $S$.  It is a smooth scheme over $\ZZ$ whose points correspond to  sets of $n+3$ distinct points $p_s \in \PP^1$,  for $s\in S$, modulo the action of $\mathrm{PGL}_2$.  Since this action is simply triply transitive, we can place $p_{n+1}=1, p_{n+2}=\infty, p_{n+3}=0$ and define the \emph{simplicial coordinates} $(t_1,\ldots,t_n)$ to be the remaining $n$ points.
In other words, they are defined for $1\leq i \leq n$ as the cross-ratios 
\begin{equation}\label{simplicial}
t_i=\frac{(p_i-p_{n+3})(p_{n+1}-p_{n+2})}{(p_{i}-p_{n+2})(p_{n+1}-p_{n+3})} \ \cdot
\end{equation}
Note that the indexing differs slightly from that in \cite{BrENS}. 
These coordinates identify $\mathcal{M}_{0,S}$ as the hyperplane complement $(\PP^1 \backslash \{0,1,\infty\})^n$ minus diagonals, and are widespread in the physics literature. We also use \emph{cubical coordinates}:
\begin{equation} \label{cubical} x_1 = t_1/t_2  \ , \ \ldots \ , \  x_{n-1} = t_{n-1}/t_n  \ , \  x_n = t_n \ . \end{equation}

\subsection{Dihedral extensions of moduli spaces} 
 A \emph{dihedral structure} $\delta$ for $S$ is an identification of $S$ with the edges of an $(n+3)$-gon (which we call $(S,\delta)$, or simply $S$ when $\delta$ is fixed) modulo dihedral symmetries. When we identify $S$ with $\{1,\ldots,n+3\}$ we take $\delta$ to be the `standard' dihedral structure that is compatible with the linear order on $S$. Let $\chi_{S,\delta}$ denote the set of chords of $(S,\delta)$. The dihedral extension  $\mathcal{M}_{0,S}^{\delta}$ of $\mathcal{M}_{0,S}$   is a smooth affine scheme over $\ZZ$ of dimension $n$ defined in \cite{BrENS}.   Its affine ring $\Or(\mathcal{M}_{0,S}^{\delta})$ is the ring over $\ZZ$ generated by  `dihedral coordinates' $u_c$, for each chord $c\in \chi_{S,\delta}$, modulo the ideal generated by the relations 
\begin{equation} \label{ucrelations}
\prod_{c\in A}  u_c  + \prod_{c\in B} u_c =1
\end{equation} 
for all sets of chords $A, B \subset \chi_{S,\delta}$ which cross completely (defined in \cite[\S 2.2]{BrENS}). We  frequently use the following special case: if $c$, $c'$ are crossing chords, then 
\begin{equation}  \label{ucprimecrossingc}
u_{c'} = 1 - x u_c  
\end{equation}
where $x$ is a product of dihedral coordinates which depends on $c,c'.$

 \begin{figure}[h]
{\includegraphics[height=6cm]{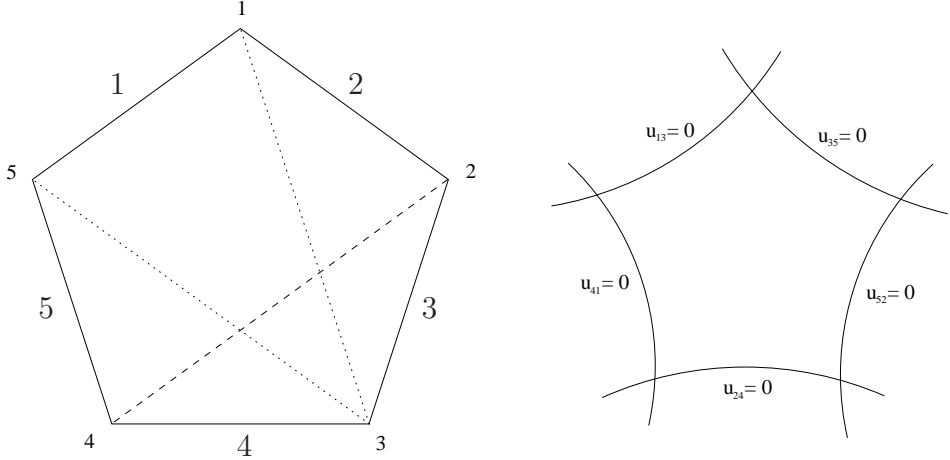}} 
\put(-318,135){\Large{$1$}}
\put(-345,50){\Large{$5$}}
\put(-270,-2){\Large{$4$}}
\put(-200,50){\Large{$3$}}
\put(-228,135){\Large{$2$}}
\caption{On the left: three out of the five chords in a pentagon, corresponding to the dihedral coordinates $u_{24}$ (dashed), $u_{35}$, $u_{13}$ (dotted). The figure illustrates the relation $u_{24}=1-u_{13}u_{35}$. On the right: the five divisors on $\mathcal{M}_{0,5}^{\delta}$ defined by $u_{ij}=0$ form a pentagon. Two divisors  intersect if and only if the corresponding chords do not cross.  }\label{figurepentagons}
\end{figure}

  The zero locus of $u_c$ is denoted $D_c \subset \mathcal{M}_{0,S}^{\delta}$.  We have 
$$\mathcal{M}_{0,S} = \mathcal{M}_{0,S}^{\delta} \backslash D \qquad \hbox{ where } \qquad D =  \bigcup_{c \in \chi_{S,\delta}}  D_c$$
and  $D$ (also denoted  by $\partial \mathcal{M}^{\delta}_{0,S}$) is a  simple normal crossing divisor.   Two components $D_c$, $D_{c'}$ intersect if and only if $c$, $c'$ do not cross.
In the case  $|S|=4$, $\mathcal{M}_{0,S}^{\delta} = \mathbb{A}^1$, and the divisor $D$ has two components, $0$ and $1$, corresponding to  the two chords in a square. The case $|S|=5$ is pictured in Figure \ref{figurepentagons}.

Let us write $\overline{\mathcal{M}}_{0,S}$ for the Deligne--Mumford compactification of $\mathcal{M}_{0,S}.$   The open subspace $\mathcal{M}_{0,S}^{\delta} \subset \overline{\mathcal{M}}_{0,S}$ can be obtained by removing all boundary divisors which are not compatible with the dihedral structure $\delta$. The set of $\mathcal{M}_{0,S}^{\delta}$ as $\delta$ ranges over all dihedral structures form an open affine cover of $\overline{\mathcal{M}}_{0,S}$. 

\subsection{Morphisms} \label{sectmorphismMon}
Given a subset $T \subset S$ with $|T|\geq 3$, let $\delta|_{T}$ denote the dihedral structure on $T$ induced by $\delta$. There is a partially defined map $f_T:\chi_{S,\delta}\rightarrow \chi_{T,\delta_{|T}}$ induced by contracting all edges in $(S,\delta)$  not in $T$. Since some chords map to the outer edges of the polygon $(T,\delta|_T)$ under this operation, it is only defined on the complementary set of such chords in $\chi_{S,\delta}$.   It gives rise to a `forgetful map'
 $$f_T : \mathcal{M}^{\delta}_{0,S} \To \mathcal{M}^{\delta|_T}_{0,T} $$
 whose associated  morphism of affine rings   $f^*_T: \Or( \mathcal{M}^{\delta_{|T}}_{0,T}) \rightarrow \Or(\mathcal{M}^\delta_{0,S})$ is
\begin{equation} \label{fTstar} f_T^* ( u_c ) = \prod_{f_T(c')=c} u_{c'}
\end{equation}
where $c\in \chi_{T,\delta|_T}$, and $c'$ ranges over its preimages in $\chi_{S,\delta}$. 
 The forgetful map restricts to a morphism $f_T : \mathcal{M}_{0,S} \rightarrow \mathcal{M}_{0,T}$ between the open moduli spaces.

A dihedral coordinate $u_c$ is such a morphism $f_{T_c}: \mathcal{M}_{0,S}^{\delta} \rightarrow \mathcal{M}_{0,T_c}^{\delta|_{T_c}} \cong \mathbb{A\!}^1$, where $T_c$ is the set of four edges which meet the endpoints of  $c$.

\subsection{Strata} Cutting along $c\in \chi_{S,\delta}$ breaks the polygon $S$ into two  smaller polygons,  $(S', \delta')$ and $(S'', \delta'')$, with $S=(S'\min\{c\})\sqcup (S''\min\{c\})$ (see e.g. \cite[Figure 3]{BrENS}). There is a  canonical isomorphism 
\begin{equation}  \label{Disproduct}
D_c  \cong \mathcal{M}_{0,S'}^{\delta'} \times \mathcal{M}_{0,S''}^{\delta''}\ .
\end{equation} 
In particular,  the restriction of a dihedral coordinate $u_{c'}$ to the divisor $D_{c}$, where $c'$ and $c$ do not cross, is  the dihedral coordinate $u_{c'}$ on either  $(S', \delta')$ or $(S'', \delta'')$, depending on  which component $c'$ lies in.

\begin{definition} \label{defNotationS/J}  Let  $J\subset \chi_{S,\delta}$ be a  set of $k$ non-crossing chords. Cutting $(S,\delta)$ along $J$ decomposes it into  polygons $(S_i,\delta_i)$, $0 \leq i \leq k$. Write
$$\mathcal{M}^{\delta/J}_{0,S/J} =  \mathcal{M}^{\delta_0}_{0,S_0} \times \cdots \times \mathcal{M}^{\delta_k}_{0,S_{k}}$$
and 
similarly, 
$\mathcal{M}_{0,S/J} =  \mathcal{M}_{0,S_0} \times \cdots \times \mathcal{M}_{0,S_{k}}.$
If $J= \{j_1,\ldots, j_k\}$ then set
$$ D_J =  D_{j_1} \cap \cdots \cap D_{j_k} \ . $$
There is a canonical isomorphism $D_J \cong \mathcal{M}^{\delta/J}_{0,S/J}$.
\end{definition}

\subsection{Trivialisation maps}
A crucial ingredient in our `renormalisation' of differential forms is to use  dihedral coordinates to define a canonical trivialisation of the normal bundles of the divisors $D_c$ in a compatible manner.
In order to define this, we  shall fix a cyclic order $\gamma$ on $S$, which is compatible with $\delta$. 
 Such a cyclic structure is simply   a choice of orientation of the polygon $(S,\delta)$.

\begin{definition}  Let $c$ be a chord as above. Let $T', T''$ be the subsets of $S$ consisting of the edges in  $S'\backslash \{c\}, S'' \backslash \{c\}$,  respectively, together with the next edge in $S$ with respect to the cyclic ordering $\gamma$.  Let $\delta', \delta''$ be the induced dihedral structures on $T', T''$. 
There are  natural bijections $S'\simeq T'$ and $S''\simeq T''$ where in each case we identify the chord $c$ with the next edge after $S'$ or $S''$ in the cyclic ordering. 
 Consider the map 
\begin{equation}  \label{fDc}
 f^{\gamma}_{c} : \mathcal{M}_{0,S}^{\delta}  \To   \mathbb{A}^1\times  \mathcal{M}_{0,S'}^{\delta'} \times  \mathcal{M}_{0,S''}^{\delta''} 
  \end{equation}
 induced by  $f^{\gamma}_{c} = f_{T_c}\times f_{T'} \times f_{T''} $, where  $\mathcal{M}^{\delta|_{T_c}}_{T_c}$ is identified with $\mathbb{A}^1$. An illustration of this map is given in Figure \ref{figuretrivialisation}.
 \end{definition}

 \begin{figure}[h]
    \begin{center}
    \fcolorbox{white}{white}{
    \begin{picture}(451,143) (101,-65)
    \SetWidth{1.5}
    \SetColor{Black}
    \Line(154.565,77.283)(198.727,77.283)
    \Line(198.727,77.283)(231.848,44.162)
    \Line(231.848,44.162)(231.848,0)
    \Line(154.565,77.283)(121.444,44.162)
    \Line(121.444,44.162)(121.444,0)
    \Line(121.444,0)(154.565,-33.121)
    \Line(154.565,-33.121)(198.727,-33.121)
    \Line(198.727,-33.121)(231.848,0)
    \Text(176,84)[lb]{{\Black{$1$}}}
    \Text(219,64)[lb]{{\Black{$2$}}}
    \Text(236,19)[lb]{{\Black{$3$}}}
      \Text(248,19)[lb]{\Large{\Black{$\overset{f_c^{\gamma}}{\to}$}}}
    \Text(220.808,-22)[lb]{{\Black{$4$}}}
    \Text(176.646,-44.162)[lb]{{\Black{$5$}}}
    \Text(124,-22)[lb]{{\Black{$6$}}}
    \Text(111,19)[lb]{{\Black{$7$}}}
    \Text(127,64)[lb]{{\Black{$8$}}}
      \Text(176,-60)[lb]{\Large{\Black{$S$}}}
    \Line[dash,dashsize=5.244](154.565,77.283)(154.565,-33.121)
    \Line(270.171,44.162)(314.333,44.162)
    \Line(314.333,44.162)(314.333,0)
    \Line(314.333,0)(270.171,0)
    \Line(270.171,0)(270.171,44.162)
      \Line[dash,dashsize=5.244](270.171,44.162)(314.333,01)
    \Line(327.5,44)(371.5,44)
    \Line(327.5,0)(371.5,0)
    \Line(327.5,44.162)(327.5,0)
    \Line(371.5,44.162)(371.5,0)
    \Line(407,55)(440,55)
    \Line(440,55)(462,22)
    \Line(407,55)(385,22)
    \Line(385,22)(407,-11)
    \Line(440,-11)(407,-11)
    \Line(440,-11)(462,22)
    \Text(290,48)[lb]{{\Black{$1$}}}
      \Text(290,-10)[lb]{{\Black{$6$}}}
       \Text(288,-40)[lb]{\Large{\Black{$T_c$}}}
       \Text(273,19)[lb]{{\Black{$8$}}}
          \Text(307,19)[lb]{{\Black{$5$}}}
            \Text(347,48)[lb]{{\Black{$8$}}}
      \Text(347,-10)[lb]{{\Black{$6$}}}
       \Text(345,-40)[lb]{\Large{\Black{$T'$}}}
       \Text(332,19)[lb]{{\Black{$7$}}}
          \Text(365,19)[lb]{{\Black{$1$}}}
            \Text(422,60)[lb]{{\Black{$6$}}}
               \Text(390,40)[lb]{{\Black{$5$}}}
                  \Text(390,-2)[lb]{{\Black{$4$}}}
                    \Text(422,-20)[lb]{{\Black{$3$}}}
                      \Text(420,-40)[lb]{\Large{\Black{$T''$}}}
                      \Text(455,-2)[lb]{{\Black{$2$}}}
                         \Text(455,40)[lb]{{\Black{$1$}}}
       \Text(145,0)[lb]{\Large{\Black{$c$}}}
        \SetWidth{0.3}
  \end{picture}
}
\end{center}
\vspace{0.1in}
\caption{An illustration of the trivialisation map $f_c^\gamma$ (the cyclic orientation $\gamma$ is clockwise, induced by the numbering).}\label{figuretrivialisation}
\end{figure}
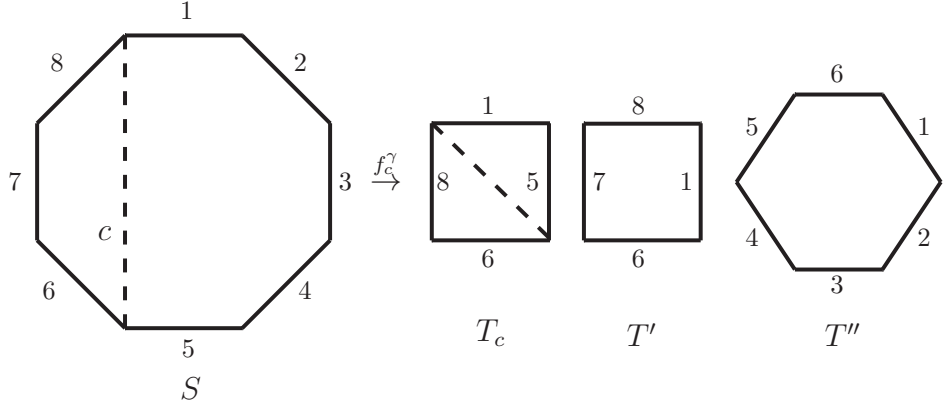
 
  The first component $f_{T_c}$ is simply the dihedral coordinate $u_c$, and hence the restriction of (\ref{fDc}) to $D_c$ 
    induces the isomorphism (\ref{Disproduct}). Note that (\ref{Disproduct}) is canonical, but $f^{\gamma}_{c}$ depends on the choice of cyclic structure $\gamma$.

  \begin{lemma} \label{lemfccommute}
If $c_1,c_2 \in \chi_{S,\delta}$ do not cross,  then $f^{\gamma}_{c_1} \circ f^{\gamma}_{{c_2}}= f^{\gamma}_{{c_2}} \circ f^{\gamma}_{{c_1}}$.
\end{lemma}
\begin{proof} Cutting along $c_1,c_2$ decomposes the oriented polygon $(S,\gamma)$ into  three smaller polygons $(S_1,\gamma_1)$, $(S_{12}, \gamma_{12})$, $(S_2, \gamma_2)$, where $S_1$ has one edge labelled by $c_1$, $S_2$ has one edge labelled by $c_2$, and $S_{12}$ two edges labelled $c_1, c_2$.  The graphs $S_1\cap S$ and $S_2\cap S$  each have one connected component and $S_{12} \cap S$ has exactly two components (one of which may reduce to a single vertex). Extending each such component by the next edge in the cyclic order defines sets $T_1, T_{12}, T_2 \subset S$ where  $|T_1|=|S_1|+1$, $|T_2|=|S_2|+1$, and $|T_{12}|= |S_{12}|+2$. One checks from the definitions that
$$f^{\gamma}_{c_1} f^{\gamma}_{c_2} = f_{T_{c_1}} \times f_{T_{c_2}}\times f_{T_1} \times f_{T_{12}} \times f_{T_2} \ ,$$
which is symmetric in $c_1, c_2$.  The point is that the operation of `extending by adding the next edge in the cyclic order' does not depend on the order
in which one cuts along the chords $c_1,c_2$. 
\end{proof}

\begin{definition} 
 Let $J\subset\chi_{S,\delta}$ be a set of non-crossing chords, and define
$$f^{\gamma}_J:  \mathcal{M}_{0,S}^{\delta} \To \mathbb{A}^J\times \mathcal{M}_{0,S/J}^{\delta/J} $$
for the composite of the maps $f_c^\gamma$, for $c\in J$, in any order, where
 $\mathbb{A}^J = (\mathbb{A}^1)^J$.
Its restriction to $D_J$ gives the canonical isomorphism  $D_J \cong \mathcal{M}^{\delta/J}_{0,S/J}$.
\end{definition} 
 When the cyclic ordering is fixed, we shall drop the $\gamma$ from the notation. 

\subsection{Domains}
Let  $X^{\delta} \subset \mathcal{M}_{0,S}^{\delta}(\RR)$ be the subset defined by the positivity  of the dihedral coordinates $u_c > 0$, for all $c \in \chi_{S,\delta}$. In simplicial coordinates (\ref{simplicial}) it is the open simplex $\{0<t_1<\cdots <t_n<1\}$. In cubical coordinates (\ref{cubical}) it is the open hypercube $(0,1)^n$.
It serves as a domain of integration. On the domain $X^{\delta}$, every dihedral coordinate $u_c$ takes values in $(0,1)$ by \eqref{ucprimecrossingc}. Given a cyclic ordering  $\gamma$ on $S$, (\ref{fDc}) defines a homeomorphism
$$f^{\gamma}_{c}:   X^{\delta} \cong  (0,1)\times X^{\delta'} \times X^{\delta''}\ .$$
More generally, for any set $J$ of $k$ non-crossing chords  (Definition \ref{defNotationS/J})  set
\begin{equation} \label{XJisproduct} X_{J} =     X^{\delta_0} \times \cdots \times X^{\delta_k}\ .\end{equation} 
It follows by iterating the above that 
\begin{equation} \label{fDXdelta} 
f^{\gamma}_{J}  :   X^{\delta} \cong (0,1)^{k} \times X_{J}  \ .
\end{equation} 
The closure $\overline{X^{\delta}} \subset \mathcal{M}_{0,S}^{\delta}(\RR)$ for the analytic topology is a compact manifold with corners which has the structure of an associahedron. 
Note that the maps $f^{\gamma}_J$ do not extend to homeomorphisms of the closed polytopes $\overline{X^{\delta}} $.

\subsection{Logarithmic differential forms} 
We define $\Omega_S^{\bullet}$ to be the graded $\QQ$-subalgebra of regular forms on $\mathcal{M}_{0,S}$
generated by the $d \log u_c$, $c\in \chi_{S,\delta}$. 
These are functorial with respect to forgetful maps, i.e. 
\begin{equation}\label{eq:fTstar log forms}
f_T^*  : \Omega_T^{\bullet} \To \Omega_S^{\bullet}
\end{equation}
which follows from (\ref{fTstar}). 
One knows that all algebraic relations between the forms $d \log u_c$  are generated by quadratic relations and furthermore, by
 Arnol'd--Brieskorn,  that $\mathcal{M}_{0,S}$ is formal, i.e., the natural map
\begin{equation} \label{formality}  \Omega_S^\bullet \overset{\sim}{\To} H_{\dR}^\bullet(\mathcal{M}_{0,S}/\QQ)
\end{equation}
is an isomorphism of $\QQ$-algebras. Consequently,  one has 
\cite[\S 6.1]{BrENS} 
\begin{equation} 
\label{H1description} 
H^1_{\dR} (\mathcal{M}_{0,S}/\QQ) = \bigoplus_{c\in \chi_{S,\delta}} \QQ\Big[ \frac{du_c}{u_c} \Big]\ .
\end{equation} 
Finally, it follows from mixed Hodge theory \cite{delignehodge2} (see, e.g., \cite[\S4]{BD1}), that 
$$\Omega_S^r  =\Gamma \big(\overline{\mathcal{M}}_{0,S},   \Omega^r_{\overline{\mathcal{M}}_{0,S}/\QQ}(\log \partial \overline{\mathcal{M}}_{0,S}) \big)  $$
are the global sections of the sheaf of regular $r$-forms over $\QQ$, with logarithmic singularities along $\partial  \overline{\mathcal{M}}_{0,S}=  
\overline{\mathcal{M}}_{0,S}\backslash  \mathcal{M}_{0,S}$.

\subsection{Residues} Taking the residue of logarithmic differential forms defines a map
$$\mathrm{Res}_{D_c} :    \Omega^{|S|-3}_S  \To  \Omega^{|S'|-3}_{S'} \otimes  \Omega^{|S''|-3}_{S''} \ .$$
It can be represented graphically by cutting $S$ along the chord $c$  (see e.g. \cite[Proposition 4.4 and Remark 4.5]{dupontvallette} where $\mathrm{Res}_{D_c}$ is denoted by  $\Delta_{\{c\}}$ up to a sign). Residues are functorial with respect to forgetful maps: 

\begin{lemma} \label{lemResfunctorial} Let $T\subset S$ as in \S\ref{sectmorphismMon} and $c' \in \chi_{T,\delta|_T}$. Let $c$ be a chord in $\chi_{S,\delta}$ in the preimage of $c'$ with respect to $f_T:\chi_{S,\delta}\rightarrow \chi_{T,\delta|_T}$. Suppose that cutting along $c'$ breaks $(T,\delta|_T)$
into polygons $T', T''$,  and cutting along $c$ breaks $(S,\delta)$ into polygons $S'$, $S''$. Then  the following diagram commutes: 
$$\begin{array}{ccc}
\Omega^{|T|-3}_T    &  \overset{f_T^*}{\To}   &   \Omega^{|S|-3}_S  \\
 \downarrow_{\mathrm{Res}_{D_{c'}}}  &   &   
\downarrow_{\mathrm{Res}_{D_c}}  \\
 \Omega^{|T'|-3}_{T'} \otimes \Omega^{|T''|-3}_{T''}   &  \overset{f^*_{T'}\otimes f^*_{T''}}{\To}   &  \Omega^{|S'|-3}_{S'}\otimes \Omega^{|S''|-3}_{S''}  \end{array}
$$ 
\end{lemma}
\begin{proof} This is simply the functoriality of the residue.
It can also be checked explicitly using (\ref{fTstar}) and (\ref{H1description}) which implies that
$$\mathrm{Res}_{D_c} f_T^* ( d\!\log u_{c'}\wedge\omega)  =\mathrm{Res}_{D_c}   \Big(\big( d\!\log u_c + \sum_x d\!\log u_x\big)\wedge f_T^*(\omega)\Big)  =  f_T^*(\omega)\big|_{u_c=0}\ ,$$
where $x$ ranges over chords in the preimage of $c'$ not equal to $c$.   Thus the statement reduces to the equation 
$(f_{T'}^*\otimes f_{T''}^*) (\omega|_{u_{c'}=0}) = f_T^*(\omega)|_{u_c=0}$, which is clear. 
For a form $\omega$ which does not have a pole along $D_{c'}$, we have $\mathrm{Res}_{D_{c'}} \omega=0$. One checks using (\ref{fTstar}) and the fact that  $f_T(c) = c'$
that $\mathrm{Res}_{D_{c}}f_T^*(\omega)=0$.  \end{proof}
 
In the opposite direction, a cyclic structure  $\gamma$ defines maps
\begin{equation}\label{eq:fcstar log forms}
(f^{\gamma}_c)^* : \QQ\textstyle\frac{dx}{x}\otimes\Omega^{|S'|-3}_{S'} \otimes   \Omega^{|S''|-3}_{S''}  \To \Omega^{|S|-3}_S 
\end{equation}
which send  $\frac{dx}{x}\otimes\omega' \otimes \omega''$ to $d\log(u_c)\wedge f_{T'}^*(\omega')\wedge f^*_{T''}(\omega'')$.

\begin{lemma} We have
\begin{equation} \label{Resonfcstar} \mathrm{Res}_{D_c} \left( (f^{\gamma}_c)^* (\textstyle\frac{dx}{x}\otimes\omega' \otimes \omega'' ) 
\right) = \omega' \otimes \omega''
\end{equation}
and 
\begin{equation} \label{Resonfcstarzero} \mathrm{Res}_{D_{c'}} \left( (f^{\gamma}_c)^* (\textstyle\frac{dx}{x}\otimes\omega' \otimes \omega'' ) 
\right) = 0
\end{equation}
if $c$ and $c'$ cross.
\end{lemma}
\begin{proof} The first equality follows from the definition. Suppose that cutting $(S,\delta)$ along   $c$ decomposes it into  $S', S''$. Since $c'$ crosses $c$, (\ref{ucprimecrossingc}) implies that $d \log (u_c) = d \log (1-x u_{c'})$ vanishes along $u_{c'}=0$. Hence forms in
 $(f^{\gamma}_{c})^* (\QQ \frac{dx}{x}  \otimes\Omega^{|S'|-3}_{S'} \otimes  \Omega^{|S''|-3}_{S''} )$ have no poles  along $D_{c'}$, which proves the second equality.
\end{proof}

\subsection{Summary of structures} With a view to generalisations we briefly list the geometric ingredients in our renormalisation procedure. We have a simple normal crossing divisor $D \subset \mathcal{M}_{0,S}^{\delta}$ whose induced stratification defines an operad structure  (\ref{Disproduct}). More precisely, this is a dihedral operad in the sense of \cite{dupontvallette}. We have  spaces of global regular logarithmic forms equipped with 
\begin{itemize}
\item (Residues)  $$\mathrm{Res}_{D_c} : \Omega^{|S|-3}_S   \To     \Omega_{S'}^{|S'|-3} \otimes  \Omega_{S''}^{|S''|-3}\ .$$ 
\item (Trivialisations, depending on a choice of cyclic structure on $S$)   $$(f_c^{\gamma})^* :  \QQ\textstyle\frac{dx}{x}\otimes \Omega_{S'}^{|S'|-3} \otimes  \Omega_{S''}^{|S''|-3} \To  \Omega^{|S|-3}_S $$
\end{itemize}
satisfying a certain number of compatibilities. In this paper, we also use the property 
$u_c|_{D_{c'}}=1$ if $D_c$ and $D_{c'}$ do not intersect, but 
we plan to return to the  renormalisation of integrals in a more general context with a leaner set of axioms.

\subsection{Examples} \label{Examples4,5} 
Let $|S|=4$, and set $S=\{s_1,s_2,s_3,s_4\}$ with the natural dihedral structure $\delta$. The square $(S,\delta)$ has two chords, and  two dihedral coordinates
$$v_{24}= x \qquad \hbox{ and } \qquad v_{13} = 1-x \ .$$
Here, and in the next example, a subscript  $ij$  denotes the chord meeting the edges labelled
$\{s_i,s_{i+1}\}$ and $\{s_j, s_{j+1}\}$ \cite[Figure 2]{BrENS}. 
The scheme $\mathcal{M}_{0,4}$ is isomorphic, via the coordinate $x$,  to $\PP^1 \backslash \{0,1,\infty\}$ and its dihedral extension is 
$\mathcal{M}_{0,4}^{\delta} = \mathrm{Spec}\, \ZZ[v_{24}, v_{13}]/(v_{24}+v_{13}=1) \cong \mathbb{A}^1$. The domain
${X^{\delta}}  \subset \RR\backslash \{0,1\}$ is the open interval  $(0,1)$. Its closure
$\overline{X^{\delta}} \subset \mathbb{A}^1(\RR) = \RR$ is $[0,1]$.

Let $|S|=5$, and set $S=\{s_1,s_2,s_3,s_4,s_5\}$ with the natural dihedral structure $\delta$.  The five chords in the pentagon $(S, \delta)$ give rise to five dihedral coordinates which satisfy equations given in \cite[\S 2.2]{BrENS}.  These equations define the affine scheme $\mathcal{M}_{0,5}^{\delta}$. The pair $x=u_{24}, y=u_{25}$ are cubical coordinates (\ref{cubical}), and embed
$$(x,y) : \mathcal{M}_{0,5}  \To \mathcal{M}_{0,4} \times \mathcal{M}_{0,4}  \cong   \PP^1 \backslash \{0,1,\infty\} \times \PP^1 \backslash \{0,1,\infty\}$$
Its image is  the complement of the  hyperbola $xy=1$.
We can write all other dihedral coordinates using (\ref{ucrelations}) in terms of these two to give:
$$u_{13}= 1-xy \quad , \quad u_{24}= x \quad , \quad u_{35} = \frac{1-x}{1-xy} \quad , \quad u_{14} = \frac{1-y}{1-xy}   \quad , \quad u_{25} = y\ .$$
The domain 
$X^{\delta}$ maps to the open unit square $\{(x,y): 0<x,y<1\}$.
The first coordinate, $x$, is the  forgetful map which forgets the edge $s_5$: 
$$   (x,y) \mapsto x:     \mathcal{M}^{\delta}_{0,5} \overset{\pi}{ \To}  \mathcal{M}^{\delta}_{0,4} $$
The induced map on affine rings satisfies $\pi^*(v_{24}) = u_{24}$, $\pi^*(v_{13})= u_{13} u_{35}$.

\subsection{de Rham projection} 
We now fix a dihedral structure $\delta$ on $S$ and write $S$ for $(S,\delta)$. There is  a volume form $\alpha_{S,\delta}$ on  $\mathcal{M}_{0,S}^{\delta}$ which is canonical up to a sign \cite[\S 7.1]{BrENS}. A cyclic structure on $S$  defines an orientation on the cell $X^{\delta}$ and fixes the sign of $\alpha_{S,\delta}$  if we demand that its integral over  $X^{\delta}$ be positive.    In simplicial coordinates it is given by \cite[(7.1)]{BrENS}: 
\begin{equation} \label{omegavolformula} \alpha_{S,\delta} =    \prod_{i=0}^{n-1} (t_{i+2} - t_{i})^{-1} dt_1\wedge\cdots\wedge dt_n  \ ,
\end{equation} 
with the convention $t_0=0$, $t_{n+1}=1$.
 
\begin{definition} \label{definitionnuS} Writing $u_{S,\delta}= \prod_{c \in \chi_{S,\delta}} u_c$ we define
\begin{equation}  \nu_{S,\delta} =  (-1)^{\frac{n(n+1)}{2}}  u_{S,\delta}^{-1}\,  \alpha_{S,\delta} 
\end{equation}
\end{definition} 
Note that the sign is the same as in \cite[\S4.5]{BD1}. 
When the dihedral structure is clear from the context, we write $\nu_S$ for $\nu_{S,\delta}.$

\begin{lemma}  \label{lemnuSpoles} The form $\nu_{S,\delta}$   defines a meromorphic form  on $\overline{\mathcal{M}}_{0,S}$  with 
  logarithmic singularities, and has simple poles only along   those divisors which bound the cell $X^{\delta}$.  
  In simplicial coordinates (\ref{simplicial}), and using  the convention $t_0=0$, $t_{n+1}=1$,
 \begin{equation} \label{nusimplicial} \nu_{S,\delta} = (-1)^{\frac{n(n+1)}{2}} \prod_{i=0}^n (t_{i+1} - t_{i})^{-1} dt_1\wedge \cdots \wedge dt_n  \ .
 \end{equation} 
\end{lemma}

\begin{proof} Using the equation $1-u_c= \prod_{c' \in A} u_{c'}$,   where $A$ is the set of chords which cross $c'$, which as an instance of  (\ref{ucrelations}), 
we deduce that
$$u_{S,\delta}^2 = \prod_{i \;\mathrm{mod}\; n+3} (1 - u_{\{i,i+1,i+2,i+3\}})\ . $$
Using the definition of dihedral coordinates as  cross-ratios \cite[(2.6) and \S 2.1]{BrENS}, 
$$ u_{S,\delta}^2=  \prod_{i \;\mathrm{mod}\; n+3} \frac{ (z_i-z_{i+1})(z_{i+2}-z_{i+3})    }{(z_i-z_{i+2})(z_{i+1}-z_{i+3}) }
 = \left( \prod_{i \;\mathrm{mod}\; n+3} \frac{ z_i-z_{i+1}   }{z_i-z_{i+2} }    \right)^2  \ .$$
Using the fact that $u_{S,\delta}$ is positive on $X^\delta$ one obtains
$$u_{S,\delta}=\prod_{i=0}^{n}(t_{i+1}-t_i)\, \prod_{i=0}^{n-1}(t_{i+2}-t_i)^{-1}\ $$
after passing to simplicial coordinates. Equation 
 \eqref{nusimplicial} then follows from \eqref{omegavolformula}.   From  \eqref{nusimplicial},   we see that $\nu_{S,\delta}$ is, up to a sign, the cellular differential form \cite[\S2]
 {browncarrschneps}  corresponding to $(S,\delta)$.  The rest follows from \cite[Proposition 2.7]{browncarrschneps}.
\end{proof}

After passing to cubical coordinates one obtains the more symmetric expression
\begin{equation} \label{nuscubical} \nu_{S,\delta}   = (-1)^{\frac{n(n+1)}{2}}\bigwedge_{i=1}^{n}   \frac{ dx_i }{x_i (1-x_i)} \ \cdot \end{equation} 
The following proposition follows from the computations in \cite[\S4]{BD1}.

\begin{proposition} \label{cornuSdual}
 Let $[\overline{X^{\delta}}]  \in  H_0^\B(\mathcal{M}^{\delta}_{0,S}, \partial \mathcal{M}_{0,S}^{\delta})$ denote the class of the closure of the domain $X^\delta$. Then, with $c_0^{\vee}$ as defined in \cite[\S4]{BD1},
  we have
  $$c_0^{\vee}\left( [\overline{X^{\delta}}]\right) = (2\pi i)^{-n}\nu_{S,\delta}\ .$$
\end{proposition}

Working in cubical coordinates and using \eqref{nuscubical} we get the following compatibility between the $\nu_{S,\delta}$ and the  maps  $f_c^\gamma:\mathcal{M}_{0,S}\rightarrow \mathcal{M}_{0,T_c}\times\mathcal{M}_{0,S'}\times\mathcal{M}_{0,S''}$. We set 
$$\nu_c=-\frac{du_c}{u_c(1-u_c)}=u_c^*(\nu_{T_c,\delta_{|T_c}})\ .$$

\begin{lemma}   \label{lemfcnus} We have $(f_c^\gamma)^* (\nu_c\otimes \nu_{S',\delta'} \otimes \nu_{S'',\delta''})  = (-1)^{(|S'|-1)(|S''|-1)} \nu_{S,\delta}.$
\end{lemma}

The sign is compatible with the single-valued Fubini theorem discussed in \cite[\S5]{BD1}: for $\omega\in \Omega_{T_c}^1$, $\omega'\in \Omega_{S'}^{|S'|-3}$, $\omega''\in \Omega_{S''}^{|S''|-3}$ we have
	\begin{equation*}
		\begin{split}
		\int_{\overline{\mathcal{M}}_{0,S}(\CC)}&\nu_{S}\wedge \overline{(f_c^\gamma)^*(\omega\otimes\omega'\otimes\omega'')} \\
		& = \left(\int_{\overline{\mathcal{M}}_{0,4}(\CC)}\nu_c\wedge\overline{\omega}\right) \left(\int_{\overline{\mathcal{M}}_{0,S'}(\CC)}\nu_{S'}\wedge \overline{\omega'} \right)\left(\int_{\overline{\mathcal{M}}_{0,S''}(\CC)}\nu_{S''}\wedge \overline{\omega''}\right)\ .
		\end{split}
		\end{equation*}

 Let $J = \{j_1,\ldots, j_k\}$ be a set of $k$ non-crossing chords. With the notation of Definition \ref{defNotationS/J} we may define
 $$\nu_J = \nu_{j_1}\otimes \cdots \otimes \nu_{j_k} \quad \in\quad \Omega^1_{T_{j_1}}\otimes\cdots\otimes \Omega^1_{T_{j_k}}$$
 $$\nu_{S/J} = \nu_{S_0} \otimes \cdots \otimes \nu_{S_k} \quad \in \quad \Omega^{|S_0|-3}_{S_0} \otimes \cdots \otimes \Omega^{|S_k|-3}_{S_k} \ .$$
We have the compatibility
\begin{equation}\label{eq:fstar nu}
(f_J^\gamma)^*(\nu_J\otimes \nu_{S/J}) = \pm\nu_{S}\ ,
\end{equation} 
with a sign that is compatible with the single-valued Fubini theorem.

\section{String amplitudes in genus \texorpdfstring{$0$}{0}} \label{sect: StringAmp0}
We give a self-contained account of open and closed string amplitudes in genus $0$, recast them  in terms of dihedral coordinates,  and discuss their convergence. The results in this section are standard in the  physics literature, which is very extensive. The seminal references are \cite{GreenSchwarzWitten}, \cite{klt}.  More recent work,  including \cite{schlottererstiebergermotivic}, \cite{stiebergersvMZV}, \cite{StiebergerTaylor}, \cite{BroedelSchlottererStiebergerTerasoma}, served as the main inspiration for the work below.

\subsection{Momentum conservation} \label{sect: MC}

Let $N=n+3\geq 3$ and let $s_{ij}\in\CC$ for $1\leq i,j\leq N$ satisfying $s_{ij}=s_{ji}$ and $s_{ii}=0$.
Let $(x_i:y_i)$ denote homogeneous coordinates on $\PP^1$ for $1\leq i \leq N.$
Consider the  functions 
$$f_{\underline{s}}=\prod_{1\leq i<j\leq N} (x_jy_i-x_iy_j)^{s_{ij}} \qquad \hbox{and}  \qquad g_{\underline{s}}=\prod_{1\leq i<j\leq N} |x_jy_i-x_iy_j|^{2s_{ij}}$$
on $((\CC\times \CC)\backslash \{0,0\})^N$.  The former is multi-valued, the latter is single-valued.

\begin{lemma} The functions $f_{\underline{s}}$,  $g_{\underline{s}}$ define (multi-valued, in the  case of $f_{\underline{s}}$) functions on the configuration  space of distinct points $p_1,\ldots,p_N\in\PP^1(\CC)$ if and only if 
\begin{equation} \label{MC} \sum_{1\leq j \leq N} s_{ij}=0  \qquad \hbox{ for all }  \quad  1\leq i \leq N\ .
\end{equation}
In this case, they are automatically $\mathrm{PGL}_2$-invariant and define (multi-valued, in the case of $f_{\underline{s}}$) functions on the moduli space
$\mathcal{M}_{0,N}(\CC)$. 
\end{lemma}

\begin{proof} The functions $f_{\underline{s}}$, $g_{\underline{s}}$ are invariant under scalar transformations $(x_i, y_i) \mapsto (\lambda_i x_i, \lambda_i y_i)$ if and only if (\ref{MC}) holds.  The first part of the statement follows. For the second, observe that $\mathrm{GL}_2$  acts by left matrix multiplication on 
$$\begin{pmatrix}
x_1 & x_2 & \ldots & x_N \\ 
y_1 & y_2 & \ldots & y_N
\end{pmatrix}
$$
Since each term $x_jy_i-x_iy_j$ is minus the determinant of the matrix formed from the columns $i,j$, $\mathrm{GL}_2$ acts via scalar multiplication. We have already established that scalar invariance is  equivalent to (\ref{MC}), and hence proves the second part.   The last part follows since  the moduli space $\mathcal{M}_{0,N}$ is the quotient of  the configuration space of $N$ distinct points in $\PP^1$ modulo the action of $\mathrm{PGL}_2$. 
\end{proof}

We call (\ref{MC}), together with $s_{ij}=s_{ji}$ and $s_{ii}=0$ the \emph{momentum conservation} equations. The solutions to these equations form a vector space (scheme) $V_N$. 

When they hold, denote the above functions simply     by 
$$f_{\underline{s}}=\prod_{1\leq i<j\leq N} (p_j-p_i)^{s_{ij}} \qquad  \hbox{and}  \qquad 
g_{\underline{s}}=\prod_{1\leq i<j\leq N} |p_j-p_i|^{2s_{ij}} \ ,$$ where $p_i=x_i/y_i$.
 The former has a canonical branch on the locus where the points $p_i$ are located on the circle $\PP^1(\RR)$ in the natural order, which corresponds to the domain $X^\delta\subset \mathcal{M}_{0,N}(\RR)$.
When $(\ref{MC})$ holds,  the  differential $1$-form 
\begin{equation} \label{omegasdef} \omega_{\underline{s}} = \frac{df_{\underline{s}}}{f_{\underline{s}}} = \sum_{1\leq i< j \leq N} s_{ij} \frac{d p_j - dp_i}{p_j-p_i}
\end{equation}
defines a logarithmic 1-form on $(\overline{\mathcal{M}}_{0,N},\partial\overline{\mathcal{M}}_{0,N})$. 
We therefore obtain a linear map 
\begin{equation} \label{VNKtoH1} V_N (\mathbb{K}) \To  \Gamma(\overline{\mathcal{M}}_{0,N},\Omega^1_{\overline{\mathcal{M}}_{0,N}/\mathbb{K}}(\log\partial \overline{\mathcal{M}}_{0,N})) \cong H^1_{\dR}(\mathcal{M}_{0,N}/\mathbb{K})
\end{equation}
for any field $\mathbb{K}$ of characteristic zero.

\begin{lemma} \label{lem: VNtoH1isom} The  map \eqref{VNKtoH1} is an isomorphism.
\end{lemma} 
\begin{proof} It is  injective: if  $\omega_{\underline{s}}$ were to vanish then its residue  along $p_i=p_j$,  viewed as a divisor in the configuration space of $N$ distinct points on the projective line, would vanish. Hence all $s_{ij}=0$. 
Next observe that $V_{N-1} \subset V_N$, and that $V_N/V_{N-1}$ is  generated by $s_{iN}=s_{Ni}$, for $1\leq i \leq N-1$, subject to the single relation 
$$s_{1N} + \cdots + s_{N-1N}=0\ .$$
Therefore $\dim V_N = \dim V_{N-1} + N-2$.  By injectivity, $V_3=0$ since $\mathcal{M}_{0,3}$ is a point.  Hence $\dim V_N = N(N-3)/2$, which equals  $\dim H_{\dR}^1(\mathcal{M}_{0,N})$ by (\ref{H1description}), and so \eqref{VNKtoH1} is an isomorphism. 
\end{proof}

\subsection{String amplitudes in simplicial coordinates} \label{sect: StringAmpSimplicial}
 It is customary in the physics literature to write the open and closed string amplitudes in simplicial coordinates (\ref{simplicial}). We use the coordinate system on $V_N$ consisting of the $s_{ij}$ for $1\leq i<j\leq n$ along with the $s_{i,n+1}$ and $s_{i,n+3}$ for $1\leq i\leq n$. We use the notation $s_{0,i}=s_{i,n+3}$.
Let $|S|=N=n+3$ and let  $\omega \in \Omega^{n}_S$ be a global logarithmic form. Let $s_{ij}\in \CC$ be a solution to the momentum conservation equations (\ref{MC}).  The associated open string amplitude is formally written as  the integral  
\begin{equation} \label{Iopensimplicial} \int_{0<t_1<\ldots< t_{n}<1}  \left( \prod_{0\leq i<j \leq n+1} (t_j-t_i)^{s_{ij}} \right)\,   \omega
\end{equation}
with the convention $t_0=0$, $t_{n+1}=1$. In the literature (see Theorem \ref{thm: Parke Taylor} below), $\omega$ is typically of the form
 \begin{equation}\label{eq: shape of omega}
 \frac{dt_1\wedge\cdots \wedge dt_n}{\prod_{i=0}^{n}(t_{\sigma(i+1)}-t_{\sigma(i)})}
 \end{equation}
 for some permutation $\sigma$ of $\{0,\ldots,n+1\}$.

Closed string amplitudes are written in the  form  
\begin{equation} \label{Iclosedsimplicial} (2\pi i)^{-n} \int_{\CC^{n}}  \left( \prod_{0\leq i<j \leq n+1} |z_j-z_i|^{2s_{ij}}\right) \,  \nu_{S}\wedge \overline{\omega} \end{equation}
where  $\nu_{S}$ was given in Definition \ref{definitionnuS}.
 For $\omega$ of the form \eqref{eq: shape of omega}, we can rewrite \eqref{Iclosedsimplicial} as
 $$\pi^{-n}\,\int_{\CC^n}  \frac{\prod_{0\leq i<j\leq n+1}|z_j-z_i|^{2s_{ij}}}{\prod_{i=0}^n(z_{i+1}-z_i)\prod_{i=0}^{n}(\overline{z}_{\sigma(i+1)}-\overline{z}_{\sigma(i)}) }\, d^2z_1\wedge\cdots\wedge d^2z_n\ ,$$
 with the notation $d^2z=d\mathrm{Re}(z)\wedge d\mathrm{Im}(z)$. Note that the  apparently complicated sign in the definition of $\nu_S$ is such that all signs cancel in the previous formula, in agreement with the conventions in the physics literature.
 
Convergence of these integrals is discussed below. 
 As we shall see, a huge amount is gained  by first rewriting them  in  dihedral coordinates.

\subsection{String amplitudes in dihedral coordinates} \label{sect: StringInDihedral} Let $S=(S,\delta)$ be a set of cardinality $N=n+3\geq 3$ and fix a dihedral structure. Suppose that $s_{ij}$ are solutions to the momentum-conservation equations. It follows from Lemma \ref{lem: VNtoH1isom} and (\ref{H1description}) that we can uniquely write 
$$\omega_{\underline{s}}  = \sum_{c\in \chi_{S} } s_c \frac{d u_c}{u_c}\ , $$ 
where the $s_c$ are linear combinations of the $s_{ij}$ indexed by each chord in $(S,\delta)$. Thus
the $s_c$ form a natural system of coordinates for the space $V_N$. More precisely:

\begin{lemma} Denoting a chord  by a set of edges $c=\{a,a+1,b,b+1\}$, we have
\begin{eqnarray} s_{ij}  &= & s_{\{i,i+1,j-1,j\}} + s_{\{i-1,i,j,j+1\}} - s_{\{i-1,i,j-1,j\}}- s_{\{i,i+1,j,j+1\}} \label{sijassc} \\
 s_{\{i,i+1, j , j+1\}} &  = &  \sum_{i <a < b \leq j} s_{ab} \ .\nonumber
 \end{eqnarray} 
\end{lemma}

\begin{proof}  See \cite[(6.14) and (6.17)]{BrENS}.
\end{proof}
The coordinates $s_c$ are better suited than the $s_{ij}$ for studying (\ref{Iopensimplicial})
and (\ref{Iclosedsimplicial}). By \eqref{sijassc}, we have on appropriate branches (e.g., on $X^{\delta}$) the equation:
$$\prod_{1\leq i<j\leq N}(p_j-p_i)^{s_{ij}} = \prod_{c\in\chi_{S}} u_c^{s_c}\ .$$

\begin{definition} For a tuple of complex numbers $\underline{s}=(s_c)_{c\in\chi_{S}}$ and a logarithmic form $\omega\in\Omega^{|S|-3}_S$, define the open string amplitude, when it converges, to be:
\begin{equation} \label{openamp}  I^{\open}(\omega,\underline{s}) = \int_{X^{\delta}} \left( \prod_{c\in \chi_{S}} u_c^{s_c}  \right) \, \omega\ .
\end{equation} 
Define the closed string amplitude, when it converges,  to be:
\begin{equation} \label{closedamp} I^{\closed}(\omega,\underline{s}) = (2\pi i)^{-n} \int_{\overline{\mathcal{M}}_{0,S}(\CC)}     \left( \prod_{c\in \chi_{S}} |u_c|^{2s_c} \right)  \, \nu_{S}\wedge\overline{\omega} \ .
\end{equation}
\end{definition} 
 These definitions  are  equivalent to     (\ref{Iopensimplicial})
and (\ref{Iclosedsimplicial}), respectively,  after passing to simplicial coordinates. For the closed string case, one can change its domain using the fact that  $  \overline{\mathcal{M}}_{0,S}(\CC) \supset \mathcal{M}_{0,S}(\CC) \subset \CC^{n}$  differ by sets of Lebesgue measure zero. \medskip

In the physics literature, one usually wants to expand string amplitudes in the Mandelstam variables $\underline{s}$. However, the integrals \eqref{openamp} and \eqref{closedamp} generally do not converge if $\underline{s}$ is close to zero, as the following propositions show.

 \subsection{Convergence of the open and closed string amplitudes}
\begin{proposition} \label{lemconv} The integral 
$I^{\open}(\omega, \underline{s})$ of (\ref{openamp})
converges absolutely for $s_c\in\CC$ satisfying
\begin{equation} \label{OpenConvergence} \mathrm{Re}(s_c) > \begin{cases}  0 \qquad \hbox{ if } \qquad \mathrm{Res}_{D_c} \omega \neq 0\ ;\nonumber \\
-1 \quad \,\, \hbox{ if } \qquad \mathrm{Res}_{D_c} \omega = 0 \ .\nonumber \\
\end{cases}
\end{equation}
\end{proposition} 
\begin{proof} Let $J\subset \chi_{S}$ be a set of non-crossing chords. The set $J$ can be extended to a maximal set $J\subset J' \subset \chi_{S}$ of non-crossing chords.
The $u_{j}$ for $j \in J'$ form a system of local  coordinates on $\mathcal{M}^{\delta}_{0,S}$ \cite[\S2.4]{BrENS}.   For any $\varepsilon>0$, consider the   set 
$$ S^J_{\varepsilon}=  \{  0\leq u_{j} < \varepsilon \quad   \hbox{ for}  \quad j \in J \ , \  u_{j'}\geq  \varepsilon \quad \hbox{ for } \quad  j'\in J' \backslash J\} \quad \subset \quad  \overline{X^{\delta}}\ .$$
The sets $S^J_{\varepsilon}$, for varying $J$, cover   $\overline{X^{\delta}}$ for sufficiently small $\varepsilon$. This follows because the latter is defined by the domain $u_c\geq 0$ for all $c\in \chi_{S}$. Since $u_c$ and $u_{c'}$ can only  vanish simultaneously if $c$ and $c'$ do not cross by (\ref{ucprimecrossingc}), it follows that 
$$\overline{X^{\delta}} = \bigcup_{\varepsilon>0}  \bigcup_{J}  S^J_{\varepsilon}\ .$$
This implies the covering property by   compactness of $\overline{X^{\delta}}$. 
It  suffices to show that the integrand is absolutely convergent on each $S^J_{\varepsilon}$. In the local coordinates $u_j$,  the normal crossing property means that we can write the integrand of (\ref{openamp}) as
$$    \left( \prod_{c\in \chi_{S}} u_c^{s_c}  \right) \omega = \left(\prod_{c\notin J} u_c^{s_c} \right)  \omega_0 \wedge   \left( \bigwedge_{c\in J} u_c^{s_c}  \frac{d u_c}{u_c^{p_c}}   \right) $$
where $p_c = - \mathrm{ord}_{D_c} \omega$ is the order of the pole  of $\omega$ along $u_c=0$, and $\omega_0$ has no poles on $S^J_{\varepsilon}$. Since $x^{\alpha} dx$ is integrable on $[0,\varepsilon)$ for $\mathrm{Re}\, \alpha >-1$,   the condition 
$\mathrm{Re}\, (s_c - p_c) >-1$ for all $c\in J$
guarantees absolute convergence over $S^J_{\varepsilon}$. 
\end{proof}
Note that the region of convergence does not permit a Taylor expansion at $s_c=0$.

\begin{proposition}  \label{propIclosedconverges} Let $N=|S|$. The integral $I^{\closed}(\omega,\underline{s})$ of (\ref{closedamp})  
converges absolutely for $s_c\in \CC$ satisfying 
\begin{equation} \label{ClosedConvergence} \frac{1}{N^2}>  \mathrm{Re}(s_c) > \begin{cases}  0 & \hbox{ if } \qquad \mathrm{Res}_{D_c} \omega \neq 0\ ; \\
- \frac{1}{2} & \hbox{ if } \qquad \mathrm{Res}_{D_c} \omega = 0 \ . \\
\end{cases}
\end{equation}
\end{proposition} 

\begin{proof} Let $\Omega$ denote the integrand of (\ref{closedamp}).    Let $D \subset \overline{\mathcal{M}}_{0,S}$ be an irreducible boundary divisor.  Supose first that $D$ is a component of $\partial \mathcal{M}_{0,S}^{\delta}$  and is therefore defined by $u_c=0$ for some  $c\in \chi_{S}$.  By Lemma \ref{lemnuSpoles}, $\nu_S$ has a simple pole along $D$. In the local coordinate $z=u_c$,  $\Omega$ has  at worst poles  of the form:
$$|z|^{2s_c}  \frac{dz \wedge d\overline{z}}{z \overline z} \; \hbox{ if } \;  \mathrm{Res}_{D_c} \omega \neq 0  \qquad , \qquad  |z|^{2s_c}  \frac{dz \wedge d\overline{z}}{z } \; \hbox{ if } \;  \mathrm{Res}_{D_c} \omega = 0  \ .$$
In polar coordinates $z=\rho\, e^{i \theta}$,  the left-hand term is  proportional to $ \rho^{2s_c-1} d\rho\, d\theta$ and hence integrable for $\mathrm{Re}(s_c)>0$, the right-hand term to $\rho^{2s_c} d\rho\, d\theta$ and hence integrable for $\mathrm{Re}(s_c)>-1/2$.
Now consider a boundary divisor $D$ which is a component of $\overline{\mathcal{M}}_{0,S} \backslash \mathcal{M}_{0,S}$ but which is not a component of $\partial \mathcal{M}_{0,S}^{\delta}$ (at `infinite distance').  It is  defined by a local coordinate $z=0$ (which is a dihedral coordinate with respect to some other dihedral structure on $S$).  By Lemma \ref{lemnuSpoles}, $\nu_S$ has no pole along $z=0$. Since $\omega$ has logarithmic singularities,  $\Omega$ is locally at worst of the form
$$  |z|^{2 p}  \frac{dz \wedge d\overline{z}}{z }$$
where $p$ is a linear form in the $s_c$. Since any cross-ratio $u_c$ has at most a simple zero or pole along $D$, it follows that $p= \sum_{c\in\chi_{S}} a_c s_c$ where $a_c\in \{0,\pm 1\}$ (an explicit formula for   $p$ in terms of $s_c$ is given in \cite[\S7.3]{BrENS}). By passing to polar coordinates one sees that the integrability condition reads $2\,\mathrm{Re}(p)>-1$. Assuming \eqref{ClosedConvergence} one gets the inequality
$$2\,\mathrm{Re}(p)>-2\frac{|\chi_{S}|}{N^2}=-\frac{N(N-3)}{N^2}>-1\ ,$$
and we are done.
\end{proof}

Put differently, for any $s_c$ satisfying  the assumptions \eqref{ClosedConvergence},  the integrand of (\ref{closedamp}) is polar-smooth on $(\overline{\mathcal{M}}_{0,S}, \partial \overline{\mathcal{M}}_{0,S})$ in the sense of Definition 3.7 of \cite{BD1}.

\subsection{Example}

 Let $|S|=4$,  and set $\omega = \frac{dx}{x(1-x)}$. We have for $s,t\in\CC$,
 $$I^{\open}(\omega, (s,t))=  \int_{0}^1  x^{s}(1-x)^t  \frac{dx}{x(1-x)}=\int_0^1x^{s-1}(1-x)^{t-1}dx\ .$$
 This is the classical beta function $\beta(s,t)$, which converges for $\mathrm{Re}(s)>0$, $\mathrm{Re}(t)>0$. 
 For the closed string amplitude we get
 \begin{align*}
 I^{\closed}(\omega, (s,t))&= \frac{1}{2\pi i} \int_{\PP^1(\CC)}  |z|^{2s}|1-z|^{2t} \frac{dz}{z(1-z)} \wedge \left(-\frac{d \overline{z}}{\overline{z}(1-\overline{z})}\right)\\
 & = \frac{1}{\pi}\, \int_{\CC}|z|^{2(s-1)}|1-z|^{2(t-1)}\,d^2z \ ,
 \end{align*}
 where  $d^2z=d\mathrm{Re}(z)\wedge d\mathrm{Im}(z)$. This is the complex beta function $\beta_\CC(s,t)$, which converges for $\mathrm{Re}(s)>0$, $\mathrm{Re}(t)>0$, $\mathrm{Re}(s+t)<1$.

\section{`Renormalisation' of string amplitudes} \label{sect: Renorm}

\subsection{Formal moduli space integrands}  Let us fix $S=(S,\delta)$ as above. 
We shall interpret the integrands of string amplitudes as formal symbols in dihedral coordinates, with a view to either  taking a Taylor expansion 
in the  variables $s_c$, or specialising to complex numbers in the case when the integrals are convergent. This will furthermore enable us to treat the open and closed string integrands simultaneously. To this end, consider  a fixed  commutative monoid $(M, +)$ which is free with finitely many generators. The main example will be
$M_S = \bigoplus_{c\in\chi_{S}}\mathbb{N}\,s_c$, the monoid of non-negative integer linear combinations of  the symbols $s_c$.

\begin{definition}
Denote by $F_S(M)$ the $\QQ$-algebra generated by  formal symbols $u_c^m$, for $c\in \chi_{S}$ and $m \in M$, modulo the relations  $u^0_c=1$ and
$$u_c^{m+m'} = u_c^m u_c^{m'} \qquad \hbox{ for all}  \quad m,m' \in M \ . $$ 
Similarly, if $c$ is a chord, let $F_c(M)$ be the $\QQ$-algebra generated by $u_c^m$ for $m\in M$ modulo the above relations. Let us write
$$\mathcal{A}_S (M)=   F_S(M)\otimes   \Omega^{|S|-3}_S $$
and set 
$$\mathcal{A}_c(M) = F_c(M) \otimes  \QQ \,d\!\log(u_c) \ .$$
\end{definition} 

We write the elements of $\A_S(M)$ without the tensor product, as linear combinations of $\left(\prod_{c\in \chi_{S}} u_c^{m_c}\right) \omega$. Similarly, an element of $\A_c(M)$ is denoted $u_c^{m_c}\,d\!\log(u_c)$.

\begin{definition}\label{defi:ASJ} Let  $J\subset \chi_{S}$ be any subset  of non-crossing chords as in Definition \ref{defNotationS/J}.  Write $J= \{j_1,\ldots, j_k\}$.  Let us define
\begin{eqnarray*} 
\mathcal{A}_J(M) & = &   \mathcal{A}_{j_1}(M) \otimes \cdots \otimes \mathcal{A}_{j_k}(M) \\
\mathcal{A}_{S/J}(M)  &=  &  \mathcal{A}_{S_0}(M) \otimes \cdots \otimes \mathcal{A}_{S_k}(M) \ .
\end{eqnarray*} 
\end{definition}

\begin{remark}\label{rem:koszul sign rule}
There is no preferred linear order on $J$ or on the set of polygons that are cut out by $J$. The  tensor products in Definition \ref{defi:ASJ} are therefore to be understood in the tensor category of graded vector spaces with the Koszul sign rule, where $\A_S(M)$ has degree $|S|-3$, and  $\A_c(M)$ has degree $4-3=1$. 
\end{remark}

\begin{remark} We can think of  the formal function $u_c^{m_c}$ as a  horizontal section of the formal connection $\nabla = d - {m_c}\, d\!\log (u_c)$  on the 
trivial rank $1$ bundle on the punctured (total space of the) normal bundle
to $D_c$. 
\end{remark} 

A forgetful map $f_T: \mathcal{M}_{0,S} \rightarrow \mathcal{M}_{0,T}$ defines a morphism $ f_T^*  :  F_T(M) \rightarrow F_S(M)$
via formula (\ref{fTstar}). By combining it with \eqref{eq:fTstar log forms} we get a morphism
\begin{equation}\label{eq:fTstarformal}
f_T^*:\A_T(M)\rightarrow \A_S(M)\ .
\end{equation}

We can realise the formal moduli space integrands as differential forms as follows.

\begin{definition}  \label{defnalphreal}
Given an additive map $\alpha: M \rightarrow \CC$, define a $\QQ$-linear  map
\begin{align*} \rho^{\open}_{\alpha}   & :\A_S(M)  \To  \hbox{differential forms on } X^{\delta}  \\
{}&\left(\prod_c u_c^{m_c}\right) \omega    \mapsto   \left(\prod_c u_c^{\alpha(m_c)}\right) \omega   \ .
\end{align*}
It is single-valued since $u_c^\alpha=\exp(\alpha\log(u_c))$ and $\log(u_c)$ has a canonical branch on $X^{\delta}$, which is the region $0<u_c<1$.  In a similar manner, we can define
\begin{align*} \rho^{\closed}_{\alpha}   & :\A_S(M)  \To  \hbox{differential forms on } \mathcal{M}_{0,S}(\CC)  \\
{}&\left(\prod_c u_c^{m_c}\right) \omega    \mapsto  (2\pi i)^{-n} \left(\prod_c |u_c|^{2\alpha(m_c)}\right) \nu_{S}\wedge\overline{\omega}   \ .
\end{align*}

\end{definition}

\subsection{Infinitesimal behaviour}

We define a kind of  residue of formal differential forms along boundary divisors which  encodes the infinitesimal behaviour of functions in the neighbourhood of the  divisor. We first define the evaluation map
$$\mathrm{ev}_c:F_S(M)\rightarrow F_c(M)\otimes F_{S'}(M)\otimes F_S(M)$$
as the morphism sending a formal symbol $u_{c'}$ to $1$ if $c'$ crosses $c$, and all other symbols to identically named symbols.

\begin{definition} For any $c\in \chi_{S}$ we define the map
$$R_{c}\quad  :  \quad   \mathcal{A}_S(M)  \To \mathcal{A}_{c}(M) \otimes \mathcal{A}_{S'}(M) \otimes \mathcal{A}_{S''}(M)  $$
to be the tensor product of the evaluation map $\mathrm{ev}_c$ and the map of logarithmic differential forms  $\omega\mapsto d\log(u_c)\otimes \mathrm{Res}_{D_c}(\omega)$.

\end{definition} 

\begin{lemma}
If $c_1, c_2\in\chi_{S}$ do not cross, then $R_{c_1}\circ R_{c_2} = R_{c_2}\circ R_{c_1}$.
\end{lemma}

\begin{proof}
The commutativity for the evaluation maps is clear. Since $d\log(u_{c_1})\wedge d\log(u_{c_2})=-d\log(u_{c_2})\wedge d\log(u_{c_1})$, the residues anticommute: $\mathrm{Res}_{D_{c_1}}\circ\mathrm{Res}_{D_{c_2}} = - \mathrm{Res}_{D_{c_2}}\circ\mathrm{Res}_{D_{c_1}}$. This sign is compensated by the Koszul sign rule (see Remark \ref{rem:koszul sign rule}) for the tensor product $\A_{c_1}(M)\otimes \A_{c_2}(M)\simeq \A_{c_2}(M)\otimes \A_{c_1}(M)$.
\end{proof}

Let $J=\{j_1,\ldots,j_k\}\subset \chi_{S}$ be a subset of non-crossing chords as in Definition \ref{defNotationS/J}. By the previous lemma one can compute the iterated residue 
$R_J = R_{j_1} \circ \cdots \circ R_{j_k}$ in any order, which provides   a linear map
$$R_J \quad  : \quad   \mathcal{A}_S(M)  \To  \mathcal{A}_J(M) \otimes \mathcal{A}_{S/J}(M)\ .$$

\subsection{Trivialisation maps}

Fix a cyclic ordering $\gamma$ on $S$ which is compatible with $\delta$. 
Using the morphisms \eqref{eq:fTstarformal}  define  for each chord $c\in\chi_{S}$ a trivialisation map
\begin{equation}\label{eq:fcstar formal F}
f_c^*:F_c(M)\otimes F_{S'}(M)\otimes F_{S''}(M)\rightarrow F_S(M)
\end{equation}
by  $f_c^*(u_c^{m}\otimes U'\otimes U'')=u_c^{m}f_{T'}^*(U')f_{T''}^*(U'')$ (compare  \eqref{eq:fcstar log forms}).   One checks that
\begin{equation}\label{eq:ev and fstar}
\mathrm{ev}_c\circ f_c^*=\mathrm{id} \ .
\end{equation}

\begin{lemma}\label{lemdiff formal F}
For $U\in F_S(M)$, the difference $U-f_c^*(\mathrm{ev}_c(U))$ lies in the ideal of $F_S(M)$ generated by elements
 $u^{m'}_{c'}-1$
for all chords $c'$ crossing $c$, and $m'\in M.$
\end{lemma}

\begin{proof}
Follows from the definitions.
\end{proof}

By tensoring \eqref{eq:fcstar formal F} with the map of logarithmic forms \eqref{eq:fcstar log forms} one gets a map
\begin{equation}\label{eq:fcstar formal}
f_{c}^*  \quad  : \quad     \mathcal{A}_{c}(M) \otimes \mathcal{A}_{S'}(M) \otimes \mathcal{A}_{S''}(M)   \To \mathcal{A}_S(M)\ .
\end{equation}

\begin{lemma}
If $c_1, c_2\in\chi_{S}$ do not cross, then $f_{c_1}^*\circ f_{c_2}^* = f_{c_2}^*\circ f_{c_1}^*$.
\end{lemma}

\begin{proof}
The commutativity for the maps on the components of the tensor products  involving formal symbols is clear. The claim is thus a consequence of Lemma \ref{lemfccommute}, which treats the components which are logarithmic forms. 
\end{proof}

Let $J=\{j_1,\ldots,j_k\}\subset \chi_{S}$ be a subset of non-crossing chords as in Definition \ref{defNotationS/J}. By the previous lemma one can compute the iterated trivialisation map
$f_J^* = f_{j_1}^* \circ \cdots \circ f_{j_k}^*$ in any order, which provides   a map
$$f_J^* \quad  : \quad   \mathcal{A}_J(M) \otimes \mathcal{A}_{S/J}(M) \To \mathcal{A}_S(M) \ .$$

\subsection{Compatibilities}

	The maps $R_c$ and $f_c^*$ satisfy the following compatibilities.

	\begin{lemma} \label{Rcfcproperties}
	\begin{enumerate} 
	\item For every chord $c$ we have $R_{c} \circ f_{c}^*  =   \mathrm{id} $.
	\item For two crossing chords $c, c'$ we have $R_{c'} \circ f_{c}^*  =   0$.
	\end{enumerate}
	\end{lemma}
	
	\begin{proof}
	(1) follows from \eqref{Resonfcstar} and \eqref{eq:ev and fstar}, and (2) follows from \eqref{Resonfcstarzero}.
\end{proof}
	
	\begin{lemma}
Let $c_1, c_2$ be two chords in $\chi_{S}$ which do not cross. Cutting along the chord $c_i$ produces polygons $S'_i, S''_i$, for $i=1,2$ with the induced cyclic or dihedral structures. 
Without loss of generality, suppose that $c_1$ lies in $S'_2$. Then 
\begin{equation}  \label{Rcfcommute}
R_{c_1}  \circ f_{c_2}^* = (\id \otimes  f_{c_2}^* \otimes \id) \circ (\id \otimes R_{c_1} \otimes \id )
\end{equation} 
as a map from $\mathcal{A}_{c_2} \otimes \mathcal{A}_{S'_2} \otimes \mathcal{A}_{S''_2}     \rightarrow \mathcal{A}_{c_1}\otimes  \mathcal{A}_{S'_1} \otimes \mathcal{A}_{S''_1}   $.

\end{lemma}
\begin{proof}
Use the notations of lemma \ref{lemfccommute}. We wish to show the following diagram commutes, where the horizontal maps are induced by  forgetful morphisms:
$$
\begin{array}{ccc}
\A_{c_2} \otimes \A_{T_1 \cup T_{12}}  \otimes \A_{ T_{2}}   &  \To  & \mathcal{A}_S   \\
  \downarrow_{\id \otimes R_{c_1}  \otimes \id} &   &  \downarrow_{R_{c_1}}  \\
 \A_{c_2}  \otimes \A_{c_1}  \otimes \A_{T_1}  \otimes  \A_{T_{12}} \otimes  \A_{T_{2}}  &  \To   &   \A_{c_1} \otimes  \A_{T_1}  \otimes  \A_{ T_{12} \cup T_{2}}  
\end{array}
$$ 
The commutativity of this diagram on the level of formal symbols is clear, and the commutativity on the level of logarithmic forms is a consequence of Lemma \ref{lemResfunctorial}.
\end{proof} 
Note that \eqref{Rcfcommute} has to be understood via the Koszul sign rule. 

We can simply write it in the unambiguous form $R_{c_2}  \circ f_{c_1}^* =  f_{c_1}^*  \circ R_{c_2}$ since the source and target of a map  $f_{c}^*$ or $R_{c}$ is uniquely determined by the data of $c$.\medskip 

The following lemma will not be needed in our renormalisation procedure, but will play a role in the analysis of the convergence of string amplitudes.

\begin{lemma}\label{lemdiff formal}
For $\Omega\in\A_S(M)$ and a chord $c\in\chi_{S}$, the difference $\Omega - f_c^*R_c\Omega$ is a linear combination of elements:
\begin{enumerate}[(i)]
\item $U\omega$ with $U\in F_S(M)$ and $\omega\in \Omega_S^{|S|-3}$ such that $\mathrm{Res}_{D_c}\omega=0$;
\item $(u^{m'}_{c'}-1)U\omega,$ with $U\in F_S(M)$ and $\omega\in \Omega_S^{|S|-3}$, for some chord $c'$ crossing $c$ and some $m'\in M$.

\end{enumerate}
\end{lemma}

\begin{proof}
This is a consequence of
Lemma \ref{lemdiff formal F}.
\end{proof}

\subsection{Integrability and residues}

	\begin{proposition}\label{prop:integrability}
	Let $\Omega\in \A_S(M)$ such that $R_c\Omega=0$ for every chord $c$. There exists an $\varepsilon>0$ such that:
	\begin{enumerate}[(1)]
	\item  $ \rho_{\alpha}^{\mathrm{open}} ( \Omega)$ is absolutely integrable on $\overline{X^{\delta}}$ for any realisation $\alpha:M\rightarrow\CC$ such that $\mathrm{Re}\,\alpha(m)>-\varepsilon$ for every generator $m$ of $M$;
	\item $ \rho^{\mathrm{closed}}_{\alpha} \,( \Omega)$ is absolutely integrable on $\overline{\mathcal{M}}_{0,S}(\CC)$ for any realisation $\alpha:M\rightarrow \CC$ such that $-\varepsilon<\mathrm{Re} \, \alpha(m) < \varepsilon$ for every generator $m$ of $M$.
	\end{enumerate}
	\end{proposition}

	\begin{proof}
	\begin{enumerate}[(1)]
	\item  
It  suffices to show that the integrand is absolutely convergent on each set $S^J_{\bullet}$, defined in the proof of Proposition \ref{lemconv}. The normal crossing property implies that we only need to treat  divergences in every local coordinate $t=u_c$ for $c$ a chord. By Lemma \ref{lemdiff formal} it is enough to consider integrands of the form \emph{(i)} and \emph{(ii)}. Since $t^\alpha dt$ is integrable around $0$ if $\mathrm{Re}\,\alpha>-1$, it suffices to check in each case  that  $\rho^{\mathrm{open}}_\alpha(\Omega)$ is a linear combination of  $t^{\alpha(m)}\omega_0$ for $m\in M$ and $\omega_0$ a smooth form with no poles along $t=0$. The case \emph{(i)} is clear. In the case \emph{(ii)} we use \eqref{ucprimecrossingc} and write $1-u_{c'}=t\psi_0$ where $\psi_0$ has no pole along $t=0$, to deduce that $u_{c'}^{\alpha(m)}-1= (1-t\psi_0)^{\alpha(m)}-1$.   Since  forms in $\Omega_S^{|S|-3}$ have at most logarithmic poles at $t=0$, the claim follows.

	\item We need to prove that $\rho^{\mathrm{closed}}_\alpha(\Omega)$ is integrable in the neighbourhood of every irreducible component $D$ of $\partial\overline{\mathcal{M}}_{0,S}=\overline{\mathcal{M}}_{0,S}\min \mathcal{M}_{0,S}$.  Supose first that $D=D_c$ is a component of $\partial \mathcal{M}_{0,S}^{\delta}$. By Lemma \ref{lemnuSpoles}, $\nu_{S}$ has a logarithmic pole along $D_c$. By Lemma \ref{lemdiff formal} it is enough to treat the case of integrands \emph{(i)} and \emph{(ii)}. We work with a local coordinate $z=u_c$. In case \emph{(i)} we see that the singularities of $\rho^{\mathrm{closed}}_\alpha(\Omega)$ are of the type $|z|^{2\alpha(m)} \frac{dz\wedge d\overline{z}}{z}$ for some $m\in M$. Rewriting in polar coordinates $z=\rho\, e^{i \theta}$, this is proportional to $\rho^{2\alpha(m)}d\rho$, which is integrable for $\mathrm{Re}\,\alpha(m)>-\frac{1}{2}$. In case \emph{(ii)} we use \eqref{ucprimecrossingc} to write
	$$ |u_{c'}|^{2\alpha(m')} -1 =|1-x z |^{2\alpha(m')} - 1 = z\psi_0 + \overline{z}\xi_0$$
	where $\psi_0$ and $\xi_0$ have no pole along $z=0$. The singularities of $\rho^{\mathrm{closed}}_\alpha(\Omega)$ are thus at worst of the type $|z|^{2\alpha(m)}\frac{dz\wedge d\overline{z}}{z}$ or $|z|^{2\alpha(m)}\frac{dz\wedge d\overline{z}}{\overline{z}}$, and the claim follows as in case \emph{(i)}. Now consider an irreducible component $D$ of $\partial\overline{\mathcal{M}}_{0,S}$ which is not a component of $\partial \mathcal{M}_{0,S}^{\delta}$ (at `infinite distance').  It is defined by a local coordinate $z=0$. 
	By Lemma \ref{lemnuSpoles}, $\nu_S$ has no pole along $z=0$ and  the singularities of $\rho^{\mathrm{closed}}_\alpha(\Omega)$ are at worst of the type $|z|^{2\alpha(\widetilde{m})}\frac{dz\wedge d\overline{z}}{\overline{z}}$ for some $\widetilde{m}$ in the abelian group generated by $M$,
	by the same argument as in  the proof of Proposition \ref{propIclosedconverges}.
	This is integrable around $z=0$ for $\mathrm{Re}\,\alpha(\widetilde{m})>-\frac{1}{2}$.
Since there are finitely many such divisors, the latter condition is implied by the hypotheses (2) for sufficiently small $\varepsilon.$
	\end{enumerate}
	\end{proof}
	
	\begin{remark}
	In the case of closed string amplitudes, an integrand $\rho_\alpha^{\mathrm{closed}}(\Omega)$ satisfying the assumptions of Proposition \ref{prop:integrability} (2) is polar-smooth on $(\overline{\mathcal{M}}_{0,S}, \partial \overline{\mathcal{M}}_{0,S})$ in the sense of Definition 3.7 of \cite{BD1}.
	\end{remark}

\subsection{Renormalisation of formal moduli space integrands}

\begin{definition}
Define a renormalisation map
\begin{eqnarray}
\mathcal{A}_S(M) & \To &  \mathcal{A}_S(M)   \\ 
 \Omega & \mapsto & \Omega^{\ren}   =  \sum_{J \subset \chi_{S}}   (-1)^{|J|}    f_{J}^* R_J   \, \Omega \nonumber \ ,
\end{eqnarray} 
where  $J$ ranges over all sets of non-crossing chords in  $\chi_{S}$.  \end{definition}
The reason for calling this map the renormalisation map, even though it does not agree with the notion of renormalisation in the strict physical sense, is that it is mathematically very close to the renormalisation procedure given in \cite{brownkreimer}.

\begin{proposition}   \label{propOmegaRennopoles}  For all $c\in \chi_{S}$,  $ R_c  \, \Omega^{\ren} = 0 \ .$
\end{proposition}
\begin{proof}  By the second part of Lemma \ref{Rcfcproperties}, $R_c f_J^*= R_c f^*_{c'} f_{J\backslash c'}^* =0$ if $J$ contains a chord $c'$ which crosses $c$. Let us denote by $S_c$ the set of subsets $K\subset \chi_{S}$ consisting of non-crossing chords $c'\neq c$ that do not cross $c$. Then, in $R_c\Omega^{\mathrm{ren}}$, only the summands indexed by $J=K$ and $J=K\cup\{c\}$, for $K\in S_c$, contribute. Therefore
$$R_{c} \Omega^{\ren}  = \sum_{K \in S_c}  \Big((-1)^{|K|}   R_c   f_{K}^* R_K   +   (-1)^{|K\cup \{c\}|}   R_c   f_{K\cup \{c\}}^* R_{K \cup \{c\}} \Big)\, \Omega\ .$$
Each  summand is of the form 
$$(-1)^{|K|}    \Big( R_c f_K^* R_K   -    R_c  f_c^*   f_K^*  R_c R_K \Big) \Omega\ .$$
By the first part of Lemma \ref{Rcfcproperties},  $(R_c  f_c^* )  f_K^*  R_c R_K= f_K^* R_c     R_K$. By the commutation relation (\ref{Rcfcommute}), this is $R_c f_K^* R_K$,  and therefore the previous expression vanishes.
\end{proof} 

We extend the renormalisation map  to tensor products of  forms by defining it be the identity on every $\A_c(M)$. For $|J|=k$ it acts upon 
$$    \A_J(M)\otimes \A_{S/J}(M) $$
via 
$\mathrm{id}^{\otimes k}\otimes \ren^{\otimes k+1} $, and is denoted also by $\ren$. 

\begin{proposition}  \label{propOmegarenormalised} Any form $\Omega$ admits a canonical  decomposition (depending only on the choice of cyclic ordering of $S$ involved in the definition of $f_J^*$):
\begin{equation} \label{OmegaintermsofRen} \Omega  =  \sum_{J\subset \chi_{S}}    f_{J}^* (R_J   \, \Omega)^{\ren}
\end{equation}
where the sum is over non-crossing sets of chords in $\chi_{S}$. 
\end{proposition}

\begin{proof}
We prove formula (\ref{OmegaintermsofRen}) by induction on $|S|$. Suppose it is true for all sets $S$ with $<N$ elements, and let $|S|=N$.
Then applying the formula (\ref{OmegaintermsofRen}) to each component of $R_K \Omega$, for $K\neq\emptyset$, we obtain
$$ R_K \Omega   \ =  \  \sum_{K\subset J} f_{J\backslash K}^*  ( R_{J\backslash K} R_K \Omega)^{\ren} 
  \  = \    \sum_{K\subset J} f_{J\backslash K}^*  ( R_{J} \Omega)^{\ren} \ .$$
Now, substituting into the definition of $\Omega^{\ren}$, we obtain
$$
\Omega^{\ren}   \ =  \   \Omega + \sum_{\emptyset \neq K} (-1)^{|K|} f_K^* (R_K \Omega)  
  \ =  \ \Omega +  \sum_{\emptyset \neq K} (-1)^{|K|} \sum_{K\subset J} f_K^*  f_{J\backslash K}^*  ( R_{J} \Omega)^{\ren} $$
  $$ =  \Omega +  \sum_{\emptyset \neq J}  \Big(\sum_{\emptyset\neq K \subset J} (-1)^{|K|}\Big) f_J^*  ( R_{J} \Omega)^{\ren}  $$
Via the binomial formula, 
$$\sum_{\emptyset\neq K \subset J} (-1)^{|K|} = \sum_{k\geq 1} (-1)^k \binom{|J|}{k} = -1$$
and therefore
$$\Omega^{\ren} = \Omega - \sum_{\emptyset \neq J}  f_J^*  ( R_{J} \Omega)^{\ren} \ .$$
Rearranging gives (\ref{OmegaintermsofRen}) and completes the induction step.  The initial case with $|S|=3$ is trivial, since $\Omega = \Omega^{\ren}$. 
\end{proof}

\begin{example} \label{exampleOmegaren4}
Let $|S|=4$. Let  $M=\mathbb{N}\, s \oplus \mathbb{N}\, t$ and consider
$$\Omega = x^{s} (1-x)^t \Big( \frac{dx}{x} + \frac{dx}{{1-x}}\Big)\ .$$
Then $R_0 \Omega = x^s \frac{dx}{x}$ and $R_1 \Omega = (1-x)^t \frac{dx}{1-x}$.
We have 
\begin{align*}
\Omega^{\ren} & =   \big(x^s (1-x)^t - x^s\big) \frac{dx}{x}  +   \big(  x^s(1-x)^t -(1-x)^t \big) \frac{dx}{1-x}\\
& = x^{s-1}((1-x)^t-1)\, dx + (1-x)^{t-1}(x^s-1)\, dx\ ,
\end{align*}
and  formula (\ref{OmegaintermsofRen}) is the statement:
$$\Omega = \Omega^{\ren}  + x^s \frac{dx}{x} + (1-x)^t \frac{dx}{1-x} \ \cdot$$
If we identify $s,t$ and their images by a realisation $\alpha: M \rightarrow \CC$ then the renormalised open string integrand $\rho^{\mathrm{open}}_{\alpha}(\Omega^{\ren})$ is integrable on $[0,1]$ if $\mathrm{Re}(s) , \mathrm{Re}(t)>-1$. 

On the other hand, the renormalised closed string integrand  $\rho^{\mathrm{closed}}_{\alpha}(\Omega^{\ren}) $ is, up to the factor $-(2\pi i)^{-1}$:
\begin{align*} 
 & \big(|z|^{2s} |1-z|^{2t} - |z|^{2s}\big) \frac{dz\wedge d\overline{z}}{z\overline{z}(1-z)} +   \big(  |z|^{2s} |1-z|^{2t} -|1-z|^{2t} \big) \frac{dz\wedge d\overline{z}}{z(1-z)(1-\overline{z})}\\
 & = |z|^{2(s-1)}(|1-z|^{2t}-1)\frac{dz\wedge d\overline{z}}{1-z} + |1-z|^{2(t-1)}(|z|^{2s}-1)\frac{dz\wedge d\overline{z}}{z}\ \cdot
 \end{align*}
It is integrable on $\PP^1(\CC)$ for $\mathrm{Re}(s),\mathrm{Re}(t)>-\frac{1}{2}$ and $\mathrm{Re}(s+t)<1$.

\end{example}

\subsection{Laurent expansion of  open string integrals}
 Let 
 \begin{equation}\label{eq:Omega for Laurent expansion}
 \Omega =  \left(\prod_{c\in \chi_{S}}  u_c^{s_c} \right) \, \omega  
 \end{equation}
be the integrand of (\ref{openamp}), viewed inside $\A_S(M_S),$ where $M_S = \bigoplus_{c \in \chi_{S}} \mathbb{N} s_c$.

\begin{definition} Let  $J\subset \chi_{S}$ be a set of non-crossing chords.  Set
\begin{equation} \label{OmegaJdef} \Omega_J =  \Big(\prod_{c \in \chi_J}  u_c^{s_c} \Big)   \,\mathrm{Res}_{D_J}(\omega) \quad  \in  \quad   \A_{S/J}(M_S)
\end{equation}
where $\mathrm{Res}_{D_J}$ denotes the iterated residue along irreducible components of $D_J$ and $\chi_J$ denotes the set of chords in $\chi_{S} \backslash J$ which do not cross any element of $J$.    Let
\begin{equation} \label{sJdefn} s_J = \prod_{c\in J} s_c\ .
\end{equation} 
\end{definition}

 The integral of $\Omega$ over $X^{\delta}$ can be canonically renormalised as follows.

\begin{theorem} \label{thmrenormOpen}  For   all  $s_c \in \CC$ satisfying the assumptions of Proposition \ref{prop:integrability} (1),
\begin{equation} \label{renormintOmegaopen} 
\int_{X^{\delta}} \Omega    \quad =  \quad \sum_{J \subset \chi_{S}}  \frac{1}{s_J}  \int_{X_J}   \Omega_J^{\ren} \ , 
\end{equation} 
where the sum in the right-hand side is over all subsets of non-crossing chords (including the empty set).  
The integrals on the right-hand side converge for 
$$\mathrm{Re}(s_c) > -\varepsilon \qquad \hbox {for some  } \varepsilon>0\ .$$ 
  \end{theorem}

\begin{proof}
For any subset $J$ of non-crossing chords,  we have
$$R_J\Omega = \Big(\prod_{c\in J}  u_c^{s_c} \frac{du_c}{u_c} \Big)\, \Omega_J \qquad \hbox{ hence } \qquad (R_J  \Omega)^{\ren} = \Big(\prod_{c\in J}  u_c^{s_c} \frac{du_c}{u_c} \Big) \, \Omega^{\ren}_J \  ,$$
where tensors are omitted for simplicity. 
By Proposition \ref{propOmegarenormalised}, we have 
$$\int_{X^{\delta}} \Omega = \sum_J \int_{X^{\delta}}  f_J^* (R_J \Omega)^{\ren} = \sum_J \int_{f_J  (X^{\delta}) } (R_J \Omega)^{\ren}\ .$$
By \eqref{fDXdelta} we have  $f_J (X^{\delta}) = (0,1)^{J} \times X_J$. Each summand in the last term equals
$$
 \int_{(0,1)^{J}\times X_J} \Big(\prod_{c\in J}  u_c^{s_c} \frac{du_c}{u_c} \Big)  \,  \Omega^{\ren}_J = \Big( \prod_{c\in J} \frac{1}{s_c}\Big) \int_{X_J} \Omega^{\ren}_J$$
 which proves (\ref{renormintOmegaopen}). Absolute convergence of every integral for $\mathrm{Re}(s_c) > -\varepsilon$ is  guaranteed by Proposition \ref{prop:integrability} (1) and Proposition \ref{propOmegaRennopoles}.
\end{proof}

The upshot is that   each integral on the right-hand side of (\ref{renormintOmegaopen})  now admits a Taylor expansion around $s_c=0$ which lies in the region of convergence:
$$ \int_{X_J}   \Omega_J^{\ren}   \quad \in \quad \CC[[ s_c \ : \ c\in \chi_{S}]]\ .$$
Note that this integral is a linear combination of a product of integrals over  $X^{\delta'}$, for various $\delta'$, by (\ref{XJisproduct}).
\begin{corollary}The open string amplitude has a canonical Laurent expansion 
$$ I^{\mathrm{open}}(\omega,\underline{s})  \quad \in \quad \CC[\textstyle\frac{1}{s_c} \ :  \ \mathrm{ord}_{D_c} \omega =-1]  [[ s_c \ : \ c\in \chi_{S}]]\ .$$ 
\end{corollary}
By   proposition \ref{lemconv},   it only has simple poles in the $s_c$ corresponding to chords $c$ such that $\mathrm{Res}_{D_c} \, \omega \neq 0$.  
 More precisely, 

\begin{corollary}  The residue at $s_c=0$ of $ I^{\mathrm{open}}(\omega,\underline{s}) $ is
\begin{equation} \label{residueinMandel} \mathrm{Res}_{s_c} \int_{X^{\delta}}  \Omega = \int_{X^{\delta} \cap D_c}   \Omega_c \ .\end{equation}
\end{corollary}
\begin{proof} 
It follows from the formula (\ref{renormintOmegaopen}) that:
$$\frac{1}{s_c} \mathrm{Res}_{s_c} \int_{X^{\delta}} \Omega= \frac{1}{s_c} \mathrm{Res}_{s_c} \sum_{J} \frac{1}{s_J} \int_{X_J} \Omega_J^{\ren}=     \sum_{c\in J} \frac{1}{s_J} \int_{X_J} \Omega_J^{\ren}\ .$$
By similar arguments to those in the proof of theorem \ref{thmrenormOpen}, we have
$$\frac{1}{s_c} \int_{X^{\delta} \cap D_c} \Omega_c = \int_{f_c(X^{\delta})} u_c^{s_c} \frac{du_c}{u_c} \, \Omega_c = \int_{X^{\delta}} f_c^* R_c \Omega\ .$$
Proposition \ref{propOmegarenormalised} is stated for a form $\Omega \in \A_S(M)$ but holds more generally for a tensor products of forms in $\A_{S'}(M)\otimes \A_{S''}(M)$, where $S', S''$ are the polygons obtained by cutting $S$ along $c$. This is because the maps $f^*, R$ and  $\ren$ are all compatible with tensor products. Therefore writing $ R_c \Omega=\sum \omega'\otimes \omega''$ (Sweedler's notation) with $\omega' \in \A_{S'}(M)$, $\omega'' \in \A_{S''}(M)$, we deduce that $R_c \Omega$ equals 
$$\sum \left(\sum_{J' \subset \chi_{S'}} f^*_{J'} (R_{J'} \omega')^{\ren} \right)\otimes \left( \sum_{J''  \subset \chi_{S''}} f^*_{J''} (R_{J''} \omega'')^{\ren} \right)=  \sum_{c\notin J} f_J^* (R_J R_c \Omega)^{\ren}$$
We therefore deduce  that
$$\int_{X^{\delta}} f_c^* R_c \Omega_ =   \sum_{c\notin J} \int_{X^{\delta}} f_c^* f_J^* (R_J R_c \Omega)^{\ren}  = \sum_{c\in  J} \int_{X^{\delta}} f_J^* (R_J  \Omega)^{\ren}= \sum_{c\in J} \frac{1}{s_J} \int_{X_J} \Omega_J^{\ren}\    ,  $$
where the last equality follows from the same arguments as in the proof of theorem \ref{thmrenormOpen}. We have therefore shown  that both sides of (\ref{residueinMandel}) coincide for  all values of $s_c$ such that the integrals converge. Note that since the left-hand side admits a Laurent expansion, the same is true of the right-hand side. 
\end{proof}

\subsection{Laurent expansion of closed string amplitudes} 
The following lemma is the single-valued version of the formula $\frac{1}{s} = \int_0^1 x^{s-1} dx$. 

\begin{lemma} \label{lemClosed1/s} For all $0<\mathrm{Re}(s)<\frac{1}{2}$ the following Lebesgue integral equals
$$\frac{1}{2\pi i} \int_{\PP^1(\CC)} |z|^{2s}\left(- \frac{dz}{z(1-z)}\right)\wedge \frac{d\overline{z}}{\overline{z}}  = \frac{1}{s} \ .$$
\end{lemma}
\begin{proof} Since $ d |z|^{2s} = s |z|^{2s} \left( \frac{dz}{z} +  \frac{d\overline{z}}{\overline{z}}\right)$, the integrand equals
$$  \frac{1}{s} \,d \left( |z|^{2s} \frac{dz}{z(1-z)}  \right) \ .$$
For $\varepsilon>0$ small enough, let $U_{\varepsilon}$ be the open subset of $\PP^1(\CC)$ given by the complement of three open discs of radius $\varepsilon$ around  $0,1,\infty$ (in the local coordinates $z, 1-z, z^{-1}$). By Stokes' theorem, 
$$\frac{1}{2\pi i}\int_{U_{\varepsilon}}  |z|^{2s} \left(- \frac{dz}{z(1-z)}\right)\wedge \frac{d\overline{z}}{\overline{z}} = \frac{1}{2\pi i}\frac{1}{s}  \int_{\partial U_{\varepsilon}}  |z|^{2s}  \frac{dz}{z(1-z)}$$
where the boundary $\partial U_{\epsilon}$ is a union of three negatively oriented circles around $0,1,\infty$. By using Cauchy's theorem, we see that all integrals in the right-hand side are bounded as $\varepsilon\rightarrow 0$, and that the only one which is non-vanishing in the limit as $\varepsilon\rightarrow 0$ is around the point $1$,  giving 
$$\lim_{\varepsilon \rightarrow 0} \,\frac{1}{2\pi i}\frac{1}{s}  \int_{\partial U_{\varepsilon}}  |z|^{2s} \dfrac{dz}{z(1-z)}  = \frac{1}{s}\ .$$
\end{proof}

Let $\Omega$ be as in \eqref{eq:Omega for Laurent expansion}. We set 
$$\Omega^{\mathrm{closed}} = (2\pi i)^{-n} \left(\prod_{c\in\chi_S}|   u_c|^{2s_c}\right)\nu_{S}\wedge\overline{\omega}\ .$$
 
The closed string amplitudes can be canonically renormalised as follows.

\begin{theorem} \label{thmrenormclosed}  For   all  $s_c \in \CC$ satisfying the assumptions of proposition \ref{prop:integrability} (2):
\begin{equation} \label{renormintOmegaClosed} 
\int_{\overline{\mathcal{M}}_{0,S}(\CC)} \Omega^{\mathrm{closed}} \quad =  \quad \sum_{J \subset \chi_S}  \frac{1}{s_J}  \int_{\overline{\mathcal{M}}_{0,S/J}(\CC)}  (\Omega_J^{\mathrm{ren}})^{\mathrm{closed}} \ , 
\end{equation} 
where the sum in the right-hand side is over all subsets of non-crossing chords (including the empty set).  The integrals on the right-hand side 
converge if $$-\varepsilon < \mathrm{Re}\, s_c <\varepsilon\quad \hbox{ for some } \varepsilon>0\ .$$
  \end{theorem}

\begin{proof}
As in the proof of Theorem \ref{thmrenormOpen} we have
$$\int_{\overline{\mathcal{M}}_{0,S}(\CC)}  \Omega^{\closed} = \sum_J \int_{\overline{\mathcal{M}}_{0,S}(\CC)}  (f_J^* (R_{J} \Omega)^{\ren})^{\mathrm{closed}}$$
We note that
$$f_J : \overline{\mathcal{M}}_{0,S}(\CC) \To  (\overline{\mathcal{M}}_{0,4}(\CC))^J\times \overline{\mathcal{M}}_{0,S/J} (\CC)$$
is an isomorphism outside a set of Lebesgue measure zero.
Using the compatibility \eqref{eq:fstar nu} between $f_J^*$ and the forms $\nu_{S}$, and making a  change of variables via $f_J$, we can write each summand in the above expression as
$$
 \prod_{c\in J}  \left( \frac{1}{2\pi i} \int_{\overline{\mathcal{M}}_{0,4}(\CC)}           |u_c|^{2s_c} \nu_c\wedge \frac{d \overline{u_c} }{\overline{u_c} }\right)   \times \int_{\overline{\mathcal{M}}_{0,S/J}(\CC)}  (\Omega^{\ren}_J )^{\mathrm{closed}}\ .$$
 By applying Lemma \ref{lemClosed1/s} we deduce
 (\ref{renormintOmegaClosed}).
\end{proof}

\begin{corollary}The closed string amplitude has a canonical Laurent expansion 
$$ I^{\closed}(\omega,\underline{s})  \quad \in \quad \CC[\textstyle\frac{1}{s_c} \ :  \ \mathrm{ord}_{D_c} \omega =-1]  [[ s_c \ : \ c\in \chi_S]]\ .$$
\end{corollary}
By Proposition \ref{propIclosedconverges}, it only has simple poles in the $s_c$ corresponding to chords $c$ such that $\mathrm{Res}_{D_c} \, \omega \neq 0$. 
 More precisely, 

\begin{corollary} The residue at $s_c=0$ of $ I^{\closed}(\omega,\underline{s}) $ is
\begin{equation} \label{residueinMandelclosed} \mathrm{Res}_{s_c}\, \int_{\overline{\mathcal{M}}_{0,S}(\CC)}  \Omega^{\closed}  = \int_{ \overline{\mathcal{M}}_{0,S'}(\CC) \times \overline{\mathcal{M}}_{0,S''}(\CC)}   (\Omega_c)^{\closed} \ ,\end{equation}
where cutting along $c$ decomposes $S$ into $S',S''$. 
\end{corollary}
The proof is similar to the proof of (\ref{residueinMandel}).

\section{Motivic string perturbation amplitudes} \label{sect: MotAmp} 

Having performed a Laurent expansion of string amplitudes, we now turn to their interpretation as periods of mixed Tate motives. 

\subsection{Decomposition of convergent forms}
Let $V_c \subset F_S(M)$ denote the ideal generated by $u^m_{c'}-1$ for any  $m \in M$, where $c'$ is a chord which crosses $c$. More generally, for any set of chords $I\subset \chi_{S}$ set  $V_{\emptyset} = F_S(M)$ and 
$$V_I = \bigcap_{c\in I} V_c\ . $$ 
Let $\Omega_{S}^{|S|-3}(I)\subset \Omega_{S}^{|S|-3}$ denote the subspace of forms whose residue vanishes along $D_c$ for all $c\in I.$ We have the following refinement of Lemma \ref{lemdiff formal}.

\begin{lemma} \label{lemmaDecompForms} Let  $\chi\subset \chi_{S}$ be a subset of chords, and let $\Omega \in \mathcal{A}_S(M)$  such that $R_c \Omega =0$ for all chords $c \in \chi$.  Then $\Omega$ has a  canonical decomposition 
\begin{equation} \label{Omegadecompasconv} \Omega = \sum_{I \subset \chi}  \Omega_{\chi}^{(I)}
\end{equation} 
where $I$ ranges over all subsets of $\chi$, and 
$$\Omega_{\chi}^{(I)} \quad  \in \quad   V_I \otimes  \Omega^{|S|-3}_S(\chi \backslash I)\ .$$
 \end{lemma} 
 \begin{proof} 
 
First observe that the case where $\chi=\{c\}$
 is a single chord follows from Lemma \ref{lemdiff formal}, since $R_c \Omega =0$ implies that $\Omega=\Omega-f_c^* R_c \Omega$, and therefore 
 $$\Omega\quad \in\quad  V_c \otimes \Omega^{|S|-3}_{S}  \ + \ F_S(M) \otimes \Omega^{|S|-3}_S(\{c\}) \ . $$
 Although the sum is not direct, the decomposition into two parts can be made canonical. For this,  consider the natural  inclusion 
 $$i_c: F_c(M)\otimes F_{S'}(M) \otimes F_{S''}(M) \To F_S(M)$$
 corresponding to the inclusions $\chi_{S'}, \chi_{S''} \subset \chi_S. $
 This map satisfies $\ev_c i_c =\id$.  Observe that $V_c \subset F_S(M)$ is the kernel of the map $\ev_c$. 
 The map $i_c \ev_c$ simply sends $u_{c'}$ to $1$ for all $c'$ crossing $c$, and preserves $u_{c'}$ for all other chords. 
 Now write
$$\Omega = (1 - i_c \ev_c \otimes \id )\Omega   +  (i_c\ev_c \otimes \id) \Omega $$
The first term  is annihilated by $\ev_c\otimes \id$, and so lies in $V_c \otimes \Omega^{|S|-3}_{S} $.
The second term satisfies $(\id \otimes d\log(u_c) \mathrm{Res}_{D_c} )  (i_c\ev_c \otimes \id) \Omega  = (i_c \otimes \id) R_c \Omega =0$ by definition of $R_c$, and hence lies in $F_S(M) \otimes \Omega^{|S|-3}_S(\{c\})$.
This gives a canonical decomposition of the form (\ref{Omegadecompasconv}) when $\chi = \{c\}$.

In the general case, proceed by induction on the size of $\chi$ by setting:
 $$ \Omega_{\chi \cup c }^{(I)} =    (i_c \ev_c  \otimes \id) \Omega^{(I)}_{\chi}     
 \qquad \hbox{ and } \qquad 
 \Omega_{\chi \cup c }^{(I \cup c)} =  \Omega^{(I)}_{\chi} -  \Omega_{\chi \cup c }^{(I)} 
 \ . $$
  Since the maps $\ev_c$ commute for different $c$, the definition does not depend on the order in which the chords in $\chi \cup c$ are taken, and the decomposition is canonical.
 By an identical argument to the one  above, we check by induction that indeed
 $$ \Omega_{\chi \cup c }^{(I \cup c)}  \in    V_{I \cup c} \otimes  \Omega^{|S|-3}_S(\chi \backslash I ) \qquad \hbox{ and }   \qquad    \Omega_{\chi \cup c }^{(I)} \in       V_{I} \otimes  \Omega^{|S|-3}_S(\chi \cup c \backslash I )     $$
 since $(\ev_c \otimes \id)\Omega_{\chi \cup c }^{(I \cup c)}=0 $
and $(\id \otimes \mathrm{Res}_{D_c} ) \, \Omega_{\chi \cup c }^{(I)} =0. $ 
  \end{proof}

 Note that the sum of the spaces $V_I \otimes \Omega_S^{|S|-3}(\chi \backslash I)$ is not direct.

\subsection{Logarithmic expansions}
For each chord $c$, let  $\ell_c $ be a formal symbol which we think of as corresponding to  a logarithm of $u_c$.

There is a continuous homomorphism of algebras defined on generators by
\begin{eqnarray}\label{formallogexpansion} F_S(M) & \To &  \QQ[\{\ell_c\}_{c\in \chi_S}][[M]]   \\
 u_c^m & \mapsto & \sum_{n\geq 0 } \frac{m^n}{n!} \ell^n_c  \ .\nonumber
\end{eqnarray} 
This extends to a map 
\begin{equation}\label{eq:formal log expansion forms}
\A_S(M)\To \Omega_S^{|S|-3}[\{\ell_c\}_{c\in \chi_S}][[M]]\ .
\end{equation}
For an additive map $\alpha:M\rightarrow \CC$ we have the realisation maps $\rho_\alpha^{\mathrm{open}}$ and $\rho_{\alpha}^{\mathrm{closed}}$ from Definition \ref{defnalphreal}. A form $\rho_\alpha^{\mathrm{open}}(\Omega)$ (resp. $\rho_\alpha^{\mathrm{closed}}(\Omega)$) has a series expansion given by composing \eqref{eq:formal log expansion forms} with $\alpha$ and by interpreting the formal symbols $\ell_c$ as:  
$$\ell_c  \mapsto \log(u_c) \qquad  \hbox{(resp. } \ell_c \mapsto \log |u_c|^2 \hbox{)}\ . $$

\begin{definition}
 A \emph{convergent}  monomial is one of the form
\begin{equation}\label{eq:convergent monomial}
\Big(\prod_{c\in \chi_{S}} \ell_c^{k_c}\Big)  \, \omega  \qquad   \in \qquad  \Omega^{|S|-3}_S[\{\ell_c\}_{c \in \chi_S}] 
\end{equation}
   where for every $c\in \chi_{S}$ such that $\mathrm{Res}_{D_c} \, \omega \neq 0$, there exists another chord $c'\in \chi_{S}$ which crosses $c$ such that $k_{c'}\geq 1$.
\end{definition}

\begin{lemma}
For a convergent monomial \eqref{eq:convergent monomial}, the corresponding integrals
$$\int_{X^\delta}\Big(\prod_{c\in \chi_{S}} \log^{k_c}(u_c)\Big)  \, \omega \quad \mbox{and} \quad \int_{\overline{\mathcal{M}}_{0,S}(\CC)}\Big(\prod_{c\in \chi_{S}} \log^{k_c}|u_c|^2\Big)\,\nu_S\wedge\overline{\omega}$$
are convergent.
\end{lemma}

\begin{proof}
For any chord $c'$ which crosses $c\in\chi_{S}$, equation (\ref{ucprimecrossingc}) implies that
$$\log(u_{c'}) = \alpha  u_c \qquad \hbox{ and } \qquad \log |u_{c'}|^2 = \beta  u_c  +  \gamma \overline{u_c} $$
for some $\alpha, \beta, \gamma$. By applying this to every chord $c$ for which $\mathrm{Res}_{D_c} \omega \neq 0$, we can rewrite the above integrals as linear combinations of 
$$\int_{X^{\delta}} F(u_c) \, \omega' \quad \hbox{ respectively } \quad \int_{\overline{\mathcal{M}}_{0,S}(\CC)} G(u_c) \,\nu_S \wedge \left( \prod_{c\in \chi} \frac{u_c}{\overline{u_c}}\right)\,  \overline{\omega'}
 $$
where  $F,G$ have at most logarithmic singularties near boundary divisors, 
 $\chi \subset \chi_S$ is a subset of chords, and  $\omega'$ has no poles along $D_c$, for all $c\in \chi_{S}$. 
Convergence in both cases follows from a very small modification of propositions \ref{lemconv}, \ref{propIclosedconverges} to allow for possible logarithmic divergences. The latter do not affect the  convergence since 
$|\log(z)|^k z^s$ tends to zero as $z\rightarrow 0$ for any $\mathrm{Re} \, s>0.$
\end{proof}

\begin{proposition} Let $\Omega \in \mathcal{A}_{S}(M)$ such that $R_c \Omega = 0 $ for all chords $c$. Then $\Omega$ 
 admits a canonical expansion  which only involves convergent monomials (\ref{eq:convergent monomial}).  \end{proposition} 
 \begin{proof}
 Apply   (\ref{formallogexpansion}) to each term in $\Omega^{(I)}_{\chi}$ in the decomposition  (\ref{Omegadecompasconv}).
 \end{proof}
 
 \begin{corollary}  \label{corexpansionasconv}  Renormalised amplitudes, where they converge,  can be canonically written as infinite sums of integrals of convergent monomials in logarithms:
 \begin{eqnarray} \int_{X^{\delta}} \rho_{\alpha}^{\mathrm{open}}\Omega^{\ren} &  = & \sum_{K=(k_c)}  a_K    \int_{X^{\delta}} \Big(\prod_{c\in \chi_{S}} \log^{k_c}(u_c)\Big)  \, \omega_K  \label{logexpansionopen} \\ 
 \int_{\overline{\mathcal{M}}_{0,S}(\CC) } \rho^{\mathrm{closed}}_{\alpha}\Omega^{\ren}   &=  & (2\pi i)^{-n} \sum_{K=(k_c)}  a'_K   \int_{\overline{\mathcal{M}}_{0,S}(\CC) } \Big(\prod_{c\in \chi_{S}} \log^{k_c}(|u_c|^2)\Big)  \,\nu_S\wedge \overline{\omega'_K} \ , 
 \nonumber
 \end{eqnarray}
 where $a_K, a'_K$ lie in the $\QQ$-subalgebra of $\CC$ generated by $\alpha(M)$. Each integral on the right-hand side converges. 
 \end{corollary} 
 
 Equivalently, if we treat the elements of $M$ as formal variables then the open and closed string amplitudes admit expansions in $\CC[[M]]$ whose coefficients are canonically expressible as  $\QQ$-linear combinations of integrals of convergent monomials as above.

\begin{example} We apply the above recipe  to Example  \ref{exampleOmegaren4}. 
By abuse of notation we identify $s$, $t$ with their images under a realisation $\alpha:\mathbb{N}s\oplus\mathbb{N}t\rightarrow \CC$, and write $(\Omega^{\mathrm{ren}})^{\mathrm{open}}= \rho_\alpha^{\mathrm{open}}(\Omega^{\mathrm{ren}})$, $(\Omega^{\mathrm{ren}})^{\mathrm{closed}}= \rho_\alpha^{\mathrm{closed}}(\Omega^{\mathrm{ren}})$. In the open case we get 
\begin{multline}\label{eq:expansion beta function} \int_{X^{\delta}} (\Omega^{\ren})^{\mathrm{open}} = \sum_{m\geq0, n\geq 1} \frac{s^m}{m!}\frac{t^n}{n!} \int_0^1  \log^m(x) \log^n(1-x) \frac{dx}{x} \\
+ \sum_{m\geq1 , n\geq 0} \frac{s^m}{m!}\frac{t^n}{n!} \int_0^1  \log^m(x) \log^n(1-x) \frac{dx}{1-x} \ \cdot
\end{multline} 
Note that $\log(x)$ vanishes at $x=1$, and   $\log(1-x)$ at $x=0$, so the integrals on the right-hand side are convergent. In the closed case we get
\begin{multline} \int_{\PP^1(\CC)} (\Omega^{\ren})^{\mathrm{closed}} = \frac{1}{\pi} \sum_{m\geq0, n\geq 1} \frac{s^m}{m!}\frac{t^n}{n!} \int_{\PP^1(\CC)}  \log^m|z|^2 \log^n|1-z|^2 \frac{d^2z}{|z|^2(1-z)} \\
+ \frac{1}{\pi}\sum_{m\geq1 , n\geq 0} \frac{s^m}{m!}\frac{t^n}{n!} \int_{\PP^1(\CC)}  \log^m|z|^2 \log^n|1-z|^2 \frac{d^2z}{z|1-z|^2} \ \cdot
\end{multline} 
Again, the integrals on the right-hand side are convergent.
\end{example} 

\subsection{Removing a   logarithm}
We can now replace each logarithm with an integral one by one. It suffices to do this once and for all for $|S|=5$. 
\begin{example} \label{ExampleLogarithm}  (The logarithm) 
  Consider the forgetful map $x: \mathcal{M}_{0,5} \rightarrow \mathcal{M}_{0,4}$ of example \ref{Examples4,5}. It is a fibration whose fibers are isomorphic to 
the projective line minus 4 points. More precisely, 
the forgetful maps $u_{24}=x$ and $u_{25}=y$ embed $\mathcal{M}_{0,5}$ as the complement of $1-xy=0$ in the product $\mathcal{M}_{0,4}\times\mathcal{M}_{0,4}$. The projection onto the first factor gives a commutative diagram
$$
\begin{array}{ccc}
  \mathcal{M}_{0,5} &   \subset   &   \mathcal{M}_{0,4} \times \mathcal{M}_{0,4} \\
  \downarrow_{x}  &    & \downarrow    \\
  \mathcal{M}_{0,4} &   =  &    \mathcal{M}_{0,4}
\end{array}
$$
The projection restricts to the real domains $X_5\rightarrow X_4$ whose fibers are identified with $(0,1)$ with respect to the coordinate $y$. Then 
$$ \int_{0< y <1 } \frac{du_{13}}{u_{13}} = \int_{0< y <1 }  d\log(1-xy) =  \log(1-x) \  . $$
The function $1-x$ is the  dihedral coordinate $v_{13}$ on $\mathcal{M}_{0,4}$ and so 
\begin{equation} \label{logv13} \log v_{13} =  \int_{0< y <1 } \frac{du_{13}}{u_{13}} \ . 
\end{equation}
In this manner we shall  inductively replace  all logarithms of dihedral coordinates with algebraic integrals. 
Note that it is \emph{not} possible to express the logarithmic dihedral coordinate  $\log v_{24} = \log x$ as an integral  of another logarithmic dihedral coordinate over the fiber in $y$ with respect to the same dihedral structure. It is precisely this subtlety that complicates the following arguments. 
\end{example}

From now on, we fix  a dihedral structure $(S,\delta)$, and consider  a differential  form of degree $|S|-3$ of the following type:
\begin{equation}\label{omegaPinlogs} \Omega= \Big(\prod_{c\in \chi_{S,\delta}} \log^{n_c}(u_c)\Big)\, \omega_0
\end{equation}
where   $\omega_0 \in \Omega^{|S|-3}(\mathcal{M}_{0,S})$. Suppose that it is convergent, i.e., for every chord $c$ such that $\mathrm{Res}\,_{D_c} \omega_0\neq 0$, 
 there exists a $c' \in I$ which crosses $c$ with $n_{c'}\geq 1$.  Define 
 $$ \mathrm{weight} (\Omega) =   |S|-3 + \sum_{c\in \chi_{S,\delta}} n_c\ .$$ 
 We can remove one logarithm at a time as follows.

 \begin{lemma}   \label{lemremovealllog}  Pick any chord $c$ such that $n_c\geq 1$, and write 
 $$\Omega = \log(u_c) \, \Omega'\ .$$
 Then there exists an enlargement  $(S_c, \delta_c)$ of $(S,\delta)$, i.e.,  $S\subset S_c$ and $|S_c \backslash S|=1$, where the restriction of $\delta_c$ to $S$ induces $\delta$, and a differential  form $\Omega'' $  of  degree $|S_c|-3$  which is a sum of  convergent forms (\ref{omegaPinlogs})   such that 
\begin{equation} \label{removelog} \int_{X^{\delta}}  \log (u_c)  \,\Omega'  = \int_{X^{\delta_c}} \Omega'' \ . 
\end{equation}
Furthermore,  each monomial in $\Omega''$  has weight equal to that of $\Omega$.
\end{lemma}

\begin{proof} The chord $c$ on $S$ is determined by four edges $T= \{t_1,\ldots, t_4\} \subset S$, where 
$t_1,t_2$ and $t_3,t_4$ are consecutive with respect to $\delta$.  This identifies $\mathcal{M}_{0,T} \cong \mathcal{M}_{0,4}$ and  $u_c$ with the dihedral coordinate $v_{13}$ in $\mathcal{M}_{0,4}$. 
Consider the set $S_c = S \cup \{t_5\}$, equipped with the dihedral structure $\delta_c$ obtained  by inserting a new edge $t_5$ next to $t_1$ and in between $t_1$ and $t_4$.
Let $T'$ be the set of five edges $T \cup \{t_5\}$ with the  dihedral structure inherited from $\delta_c$. Then $\mathcal{M}_{0,T'} \cong \mathcal{M}_{0,5}$ (see example \ref{ExampleLogarithm}).

Consider the diagram
  $$
\begin{array}{ccc}
\mathcal{M}^{\delta_c}_{0,S_c}    & \overset{f_{T'}}{\To}  & \mathcal{M}^{{\delta_c}|_{T'}}_{0,5}    \\
\downarrow_{f_S}  &   &  \downarrow_{x}  \\
 \mathcal{M}^{\delta}_{0,S}  &  \overset{u_c}{\To}  &   \mathcal{M}^{\delta}_{0,4}
\end{array}\ .
$$
It commutes  since forgetful maps are functorial.  Let $\beta \in \Omega^1(\mathcal{M}_{0,5}) $ denote the form $d\log u_{13}$ of  
 example \ref{ExampleLogarithm}  whose integral in the fiber yields $\log (v_{13})$ and set 
$$\Omega'' = f_S^*(\Omega') \wedge f_{T'}^*(\beta)\ .$$
 The product $f= f_S \times f_{T'}$ induces  a morphism
$$ f: \mathcal{M}_{0,S_c} \To  \mathcal{M}_{0,S} \times_{\mathcal{M}_{0,4}} \mathcal{M}_{0,5}$$
and an isomorphism
$f: X^{\delta_c}  \cong X^{\delta} \times (0,1)$. 
Since $\Omega''= f^*(\Omega' \wedge \beta)$, we find by   changing  variables  along the map $f$  that 
$$\int_{X^{\delta_c}} \Omega'' =    \int_{X^{\delta} \times (0,1)} \Omega' \wedge \beta =    \int_{X^{\delta}} \log(u_c) \,\Omega' \ .$$
The second integral takes place on the fiber product  $\mathcal{M}_{0,S} \times_{\mathcal{M}_{0,4}} \mathcal{M}_{0,5}$ and is computed using  (\ref{logv13}).
This proves equation (\ref{removelog}).

We now check that $\Omega''$ is of the required shape (\ref{omegaPinlogs}) with respect to  $\mathcal{M}_{0,S_c}^{\delta_c}$ and convergent. 
First of all,  observe that  for any forgetful morphism $f:S' \rightarrow S$ and any chords  $a, b$ in $S$ which cross, we have by  (\ref{fTstar})
$$f^* \left(\log(u_{a}) \,  \frac{du_b}{u_b} \right) =  \sum_{a',b'}  \log( u_{a'}  )\frac{du_{b'}}{u_{b'}}   $$
where $a', b'$  range over  chords in $S'$ in the preimage of $a$ and $b$ respectively. Every pair  $a',b'$  crosses. 
It remains  to check the convergence condition along the poles of the $1$-form $\beta.$ For this, 
denote the two chords in $S_c$ lying above the chord $c$ by $c_1, c_2$. The chord $c_1$ corresponds to  edges $\{t_5,t_1; t_3,t_4\}$ and $c_2$ to $\{t_1,t_2; t_3,t_4\}$. By (\ref{fTstar}), we have
$f_S^* (u_c) =   u_{c_1} u_{c_2}$, and 
 by example \ref{ExampleLogarithm}  $f_{T'}^*\beta = \frac{d u_{c_2}}{u_{c_2}}$, ($\beta$  corresponds to $ d\log \, u_{13}$ in example \ref{Examples4,5}). We therefore check that 
 $$ f_S^*(d u_c) \wedge f_{T'}^*(\beta) =   d (u_{c_1} u_{c_2}) \wedge \frac{d u_{c_2}}{u_{c_2}} =   d u_{c_1} \wedge d u_{c_2} $$
 $$ f_S^*  \left(\log\,(  u_{c'}) \frac{d  u_c}{u_c}\right) \wedge f_{T'}^*(\beta) =  \left( \sum_{c''}  \log u_{c''} \right) \frac{du_{c_1} }{u_{c_1}} \wedge \frac{d u_{c_2}}{u_{c_2}} $$
where $c'$ crosses $c$. In the sum,  $c''$ ranges over the preimages of $c'$ under $f_S$, and necessarily crosses both $c_1$ and $c_2$. 
It follows that $\Omega''$ is a sum of convergent monomials in logarithms. The statement about the weights is clear. 
\end{proof}

\begin{remark}
Note that $S_c$ depends on the  choice of where to insert the new edge we called $t_5$.
 Similarly, the computation in example \ref{ExampleLogarithm} also involves a choice: we could instead have used 
$$- \log(1-x) = \int_{0<y < 1}  \frac{d u_{35}}{u_{35}}  = \int_{0<y<1} d \log \Big( \frac{1-x}{1-xy}\Big) \ .$$  
Thus there are two different ways in which we can remove each logarithm. One can presumably  make these choices in a canonical way.
\end{remark} 

\begin{corollary}\label{corbeforecormotivic}
Let $\Omega $ be of the form  \eqref{omegaPinlogs} and convergent.  Then 
the integral $I$ of $\Omega$ over $X^{\delta}$ is an absolutely convergent integral
\begin{equation} \label{removedalllogs} \int_{X^{\delta'}}   \omega 
\end{equation}
where  $S' \supset S$ is a set  with dihedral  structure $\delta'$ compatible with $\delta$,  and $\omega \in \Omega_{S'}$ a logarithmic algebraic  differential form with no poles along the boundary of $\mathcal{M}_{0,S'}^{\delta'}$. Furthermore, $|S'|= \mathrm{weight} (\Omega) +3$. 
\end{corollary} 
\begin{proof} Apply the previous lemma inductively to remove the logarithms $\log(u_c)$ one at a time. At each stage the total degree of the logarithms  decreases by one. One obtains a $\ZZ$-linear combination of convergent integrals of the form \eqref{removedalllogs}. Add the integrands together to obtain a single integral of the required form.
\end{proof} 

\subsection{Motivic versions of open string amplitude coefficients}
    Let $\MT(\ZZ)$ denote the tannakian category of mixed Tate motives over $\ZZ$ with rational coefficients \cite{delignegoncharov}. An object $H\in\MT(\ZZ)$ has two underlying $\QQ$-vector spaces $H_\dR$ (the de Rham realisation) and $H_\B$ (the Betti realisation) together with a comparison isomorphism $\mathrm{comp}:H_\dR\otimes_\QQ\CC\stackrel{\sim}{\to}H_\B\otimes_\QQ\CC$. For every integer $N\geq 3$ we have an object $H=H^{N-3}(\mathcal{M}^\delta_{0,N},\partial\mathcal{M}^\delta_{0,N})$ in $\MT(\ZZ)$ whose de Rham and Betti realisations are the usual relative de Rham and Betti cohomology groups of the pair $(\mathcal{M}^\delta_{0,N},\partial\mathcal{M}^\delta_{0,N})$ \cite{goncharovmanin, BrENS}.
    
    Let us recall \cite{brownICM, brownnotesmot} the algebra $\Pe^\mm_{\MT(\ZZ)}$ of motivic periods of the category $\MT(\ZZ)$. Its elements can be represented as equivalence classes of triples $[H,[\sigma],[\omega]]^\mm$ with $H\in \MT(\ZZ)$, $[\sigma]\in H_\B^\vee$ and $[\omega]\in H_\dR$. It is equipped with a period map 
    $$\mathrm{per}:\Pe^\mm_{\MT(\ZZ)}\to \CC$$ 
    defined by $\mathrm{per}\,[H,[\sigma],[\omega]]^\mm = \langle[\sigma],\mathrm{comp}\,[\omega]\rangle$. Let us also recall the subalgebra $\Pe^{\mm,+}_{\MT(\ZZ)}$ of effective motivic periods of $\MT(\ZZ)$. 

\begin{corollary} \label{cormotivic} 
Let $\Omega $ be of the form  \eqref{omegaPinlogs} and convergent. Then the integral 
$$I=\int_{X^{\delta}} \Omega $$
is a period of a universal moduli space motive
$H^{N-3}(\mathcal{M}_{0,S'}^{\delta'},  \partial \mathcal{M}_{0,S'}^{\delta'})\ ,$
where $S' \supset S$ is a set with dihedral structure $\delta'$ compatible with $\delta$, and $|S'|=N=3+\mathrm{weight}(\Omega)$.
More precisely, we can write $I = \mathrm{per}\, I^{\mm}$ where
$$I^{\mm} = [ H^{N-3}(\mathcal{M}_{0,S'}^{\delta'},  \partial \mathcal{M}_{0,S'}^{\delta'}), X^{\delta'}, [\omega]]^{\mm} \quad\in\quad\Pe^{\mm,+}_{\mathcal{MT}(\ZZ)}$$
is an effective motivic period  of weight $N-3$ and $\omega \in \Gamma(\overline{\mathcal{M}}_{0,N},\Omega^{N-3}_{\overline{\mathcal{M}}_{0,S'}}(\log\partial\overline{\mathcal{M}}_{0,S'}))$ is a logarithmic differential form.
\end{corollary}
\begin{proof} Let $(S',\delta')$ and $\omega$ be as in Corollary \ref{corbeforecormotivic}, set $H=H^{N-3}(\mathcal{M}_{0,S'}^{\delta'},\partial\mathcal{M}^{\delta'}_{0,S'})$ and define $I^\mm$ as in the statement. 
We have
$$I=\int_{X^{\delta'}}\omega = \mathrm{per}\,I^\mm$$
by definition of the comparison isomorphism for $H$. The statement about the weight follows from the fact that $\omega$ is logarithmic in the following way. As for every mixed Tate motive, the weight filtration on $H_\dR$ is canonically split by the Hodge filtration \cite[2.9]{delignegoncharov} and we have a weight grading on $H$;
this splitting implies  in particular that  $H=W_{2(N-4)}H \oplus F^{N-3}H$.
The statement about the weight says that the class of $\omega$ is a homogeneous element of weight $2(N-3)$ with respect to this grading, i.e., that $[\omega]\in F^{N-3}H$. This can be checked after extending the scalars to $\CC$ and thus follows from Corollary 4.13 in \cite{BD1}  (compare with \cite[Proposition 3.12]{dupontodd}).
\end{proof} 

\begin{remark}Note that the motivic lift $I^{\mm}$ of $I$ depends on some choices which go into Lemma \ref{lemremovealllog}.  One expects, from the period conjecture, that it is independent of these choices. One can possibly make the lift canonical by fixing choices in the application of Lemma \ref{lemremovealllog}.
\end{remark} 

We deduce a number of consequences:
\begin{theorem} \label{thm: opencoeffsareMZV}
The coefficients in the Laurent expansion of open string amplitudes with $N$ particles are multiple zeta values.  More precisely, 
$$I^{\open}(\omega, \underline{s})= \sum_{\underline{n}=(n_c)_{c \in \chi_S} }  \zeta_{\underline{n}}  \, \underline{s}^{\underline{n}} \qquad \hbox{ where } \qquad    \underline{s}^{\underline{n}}= \prod_{c\in \chi_{S}} s_{c}^{n_{c}} $$ 
and each $n_c\geq -1$. Here,   $\zeta_{\underline{n}}$  is a $\QQ$-linear combination of multiple zeta values  of weight    $N+|\underline{n}| -3$, where $|\underline{n}|= \sum_{c\in \chi_S} n_c. $
\end{theorem} 

\begin{proof} Use the fact that the periods of universal moduli space motives are multiple zeta values \cite{BrENS}. One can obtain the statement about the weights either by modifying the argument of \emph{loc. cit} or as a corollary of the next theorem (using the fact that a real motivic period of weight $n$ of an effective  mixed Tate motive over $\ZZ$ is a $\QQ$-linear combination of motivic multiple zeta values of weight $n$.)
\end{proof}

Theorem \ref{thm: opencoeffsareMZV} is well-known in this field using results scattered throughout the literature,  but until now lacked a completely rigorous proof from start to finish.
\begin{theorem} \label{thm: motivicliftIopen}
The above expansion admits a (non-canonical) motivic lift
\begin{equation} \label{Amotivic}  I^{\mm}(\omega, \underline{s})= \sum_{\underline{n}=(n_c)_{c \in \chi_S} }  \zeta^{\mm}_{\underline{n}}  \, \underline{s}^{\underline{n}}  
\end{equation}
where    $\zeta^{\mm}_{\underline{n}}$   is a $\QQ$-linear combination of motivic multiple zeta values of weight $N+|\underline{n}| -3$, whose period is 
$\zeta_{\underline{n}}. $
\end{theorem}

\begin{proof}  Apply the  Laurent expansions \eqref{logexpansionopen}  to each renormalised integrand \eqref{renormintOmegaopen} and  invoke  Corollary \ref{cormotivic}.  This expresses the terms in the Laurent expansion  as linear combinations of products of motivic periods of the required type. 
\end{proof}

\begin{remark} The existence of a motivic lift is a pre-requisite for the computations of Schlotterer and Stieberger \cite{schlottererstiebergermotivic}, in which the motivic periods are decomposed into an `$f$-alphabet' (rephrased in a different language, this paper and related literature studies the action of the motivic Galois group  on $I^{\mm}(\omega, \underline{s})$). In \cite{schlottererschnetz}, this is achieved by assuming the period conjecture. The computation of universal moduli space periods in terms of multiple zeta values can be carried out algorithmically \cite{BrENS}, \cite{Panzer2015}, \cite{Bogner2016}.  This type of analytic  argument (or \cite{terasoma}, \cite{BroedelSchlottererStiebergerTerasoma}) is used in the literature to deduce a  theorem of the form  \ref{thm: opencoeffsareMZV}, but it is not capable of proving the much stronger statement \ref{thm: motivicliftIopen}.
\end{remark}

\begin{example} (`Motivic' beta function).   
We now treat the case of the beta function (Example \ref{exampleOmegaren4}) by using the expansion \eqref{eq:expansion beta function} of the renormalised part. We can remove all logarithms at once and write, for $\omega\in\{\frac{dx}{x},\frac{dx}{1-x}\}$:
$$\int_0^1\frac{\log^m(x)}{m!}\frac{\log^n(1-x)}{n!}\,\omega = (-1)^{m+n}\int_{\Delta^{m+n+1}}\frac{du_1}{1-u_1}\cdots \frac{du_n}{1-u_n} \,\omega\, \frac{dv_1}{v_1}\cdots \frac{dv_m}{v_m} $$
where $\Delta^{m+n+1}=\{0<u_1<\cdots <u_n<x<v_1<\cdots <v_m<1\}$ is the standard simplex. We thus get the following expansion:
\begin{multline*}
\beta(s,t)  =\frac{1}{s}+\frac{1}{t} + \sum_{m\geq 0,n\geq 1} (-s)^m(-t)^n\,\zeta(\{1\}^{n-1},m+2) \\
+ \sum_{m\geq 1,n\geq 0}(-s)^m(-t)^n\, \zeta(\{1\}^{n},m+1)\ ,
\end{multline*}
which can be rewritten as
\begin{equation}\label{eq: beta expansion MZVs}
\beta(s,t)=\left(\frac{1}{s}+\frac{1}{t}\right)\left(1-\sum_{m\geq 1,n\geq 1}(-s)^m(-t)^n\,\zeta(\{1\}^{n-1},m+1)\right)\ .
\end{equation}
The above argument yields  a `motivic' beta function
\begin{equation} \label{motbeta} \beta^{\mm}(s,t)=  \left(\frac{1}{s}+\frac{1}{t}\right)\left(1-\sum_{m\geq 1,n\geq 1}(-s)^m(-t)^n\,\zeta^{\mm}(\{1\}^{n-1},m+1)\right)\ .
\end{equation}
whose period, applied termwise, gives back \eqref{eq: beta expansion MZVs}.

Note that Ohno and Zagier observed in \cite{ohnozagier} that \eqref{eq: beta expansion MZVs} agrees with the more classical expansion of the beta function (\ref{introbetastexpansion}). 
Likewise, one can verify using motivic-Galois theoretic techniques that \eqref{motbeta} indeed coincides with the definition \eqref{introbetamot}.

\end{example}

\subsection{Single valued projection and single-valued periods} We let $\Pe^{\mm,\dR}_{\MT(\ZZ)}$ denote the algebra of motivic de Rham periods of the category $\MT(\ZZ)$ \cite{brownnotesmot} (see also \cite[\S 2.3]{BD1}). It is equipped with a single-valued period map
$$\s: \Pe^{\mm,\dR}_{\MT(\ZZ)} \To \RR$$
defined in \cite{brownSVMZV, brownnotesmot} (see also \cite[\S 2]{BD1}). The de Rham projection 
$$\pi^{\mathfrak{m},\dR}: \Pe^{\mm, +}_{\MT(\ZZ)} \rightarrow \Pe^{\mathfrak{m},\dR}_{\MT(\ZZ)}$$ 
on effective mixed Tate motivic periods was defined in \cite{brownSVMZV} and \cite[4.3]{brownnotesmot} (see also \cite[Definition 4.3]{BD1}).

\begin{definition}
Given a choice of motivic lift (\ref{Amotivic}), define its de Rham projection to be its image after applying $\pi^{\mathfrak{m},\dR}$ term-by-term:
$$I^{\mathfrak{m},\dR}(\omega, \underline{s})= \sum_{\underline{n}=(n_c)_{c \in \chi_S} }          \zeta^{\mathfrak{m},\dR}_{\underline{n}}  \, \underline{s}^{\underline{n}}  \ ,  \quad \hbox{ where }  \qquad  \zeta^{\mathfrak{m},\dR}_{\underline{n}} = \pi^{\mathfrak{m},\dR} \,  \zeta^{\mm}_{\underline{n}} \ .$$
This makes sense since $\zeta^{\mm}_{\underline{n}}$ is effective. Likewise, define its single-valued version 
$$I^{\sv}(\omega, \underline{s})= \sum_{\underline{n}=(n_c)_{c \in \chi_S} }  \zeta^{\sv}_{\underline{n}}  \, \underline{s}^{\underline{n}}  \ , \quad \hbox{ where } 
  \zeta^{\sv} =  \s \,   \zeta^{\mathfrak{m},\dR}_{\underline{n}}$$
  It is a  Laurent series whose coefficients are $\QQ$-linear combinations of single-valued multiple zeta values. 
\end{definition} 
Since $\sv\circ \pi^{\mathfrak{m},\dR} = \s\circ\pi^{\mathfrak{m},\dR}$, we could equivalently  have applied the map $\sv$, which is specific to the mixed Tate situation (see \cite[\S2.6]{BD1}). We now  compute $I^{\sv}(\omega, \underline{s})$.

\begin{lemma}  \label{lemmasvlog} For any $x\in \CC \backslash \{1\}$, 
$$ \log |1-x|^2 = \frac{1}{2\pi i} \int_{\PP^1(\CC)}  \left(-\frac{d y}{y(1-y)}\right) \wedge  d\log(1-\overline{x}\overline{y}) $$
\end{lemma}
\begin{proof} This follows from the computations of \cite[\S6.3]{BD1} after a change of coordinates. 
\end{proof}

\begin{theorem}
 Consider an integral of the form 
 $$I = \int_{X^{\delta}}  \prod_{c\in \chi_{S}} (\log^{n_c} (u_c)) \, \omega_0$$
 where the integrand is convergent of the form (\ref{omegaPinlogs}). Let $I^{\mm}$ denote a choice of motivic lift (Corollary \ref{cormotivic}). Its  single-valued period 
 $I^{\sv} = \s \, \pi^{\mathfrak{m},\dR}  (I^{\mm})$ is 
 \begin{equation} \label{svImmformula} I^{\sv}   = (2\pi i)^{3-|S|} \int_{\overline{\mathcal{M}}_{0,S}(\CC)}  \prod_{c\in \chi_{S}} (\log^{n_c}  |u_c|^2) \, \nu_S\wedge\overline{\omega_0} \ ,
 \end{equation} 
and in particular does not depend on the choice of motivic lift. 
\end{theorem}

\begin{proof}   Repeated application of Lemma \ref{lemremovealllog}
(which may involve a choice at each stage), gives rise to a dihedral structure $(S', \delta')$, a morphism 
$$ f: \mathcal{M}_{0,S'} \To    \mathcal{M}_{0,S} \times_{\mathbb{A}^k}  \big(\mathcal{M}_{0,5}\big)^k \quad \subset \quad \mathcal{M}_{0,S} \times (\mathbb{P}^1\backslash \{0,1,\infty\})^k  \ ,$$
and  a differential form $\omega' \in \Omega_{S'}$ with no poles along $\partial \mathcal{M}_{0,S'}^{\delta'}$ such that 
$$I= \int_{X^{\delta'}} \omega'  \quad \hbox{ is the period of } \quad I^{\mm} = [H^{|S'|-3}(\mathcal{M}_{0,S'}^{\delta'}, \partial \mathcal{M}_{0,S'}^{\delta'}), X_{\delta'}, \omega']^{\mm}\ .$$
The form $\omega'$ satisfies
$ \omega'= f^*(\omega_0 \wedge \beta)$
where
$$  \beta =     d \log (1-x_1y_1) \wedge \ldots \wedge d \log (1-x_ky_k) $$
and $x_1,\ldots, x_k$ denote the coordinates on $\mathbb{A}^k$ and correspond to the $1-u_c$, with multiplicity $n_c$, taken in some order.  By   \cite[Theorem 3.16]{BD1}
and  Corollary \ref{cornuSdual}, 
\begin{equation}\label{inproofsvIm} \s \, \pi^{\mathfrak{m},\dR}  \, I^{\mm} = (2\pi i)^{3-|S'|} \int_{\overline{\mathcal{M}}_{0,S'}(\CC)} \nu_{S'} \wedge   \overline{\omega'}  \ .
\end{equation} 
 Since $\omega', \nu_{S'}$ are logarithmic with singularities along distinct divisors,  the integral converges. 
By repeated application of Lemma \ref{lemfcnus},  we obtain that $\nu_{S'}$ is, up to a sign, the pullback by $f$ of the form
$$\nu_S\wedge\left(-\frac{dy_1}{y_1(1-y_1)}\right)\wedge\cdots\wedge \left(-\frac{dy_k}{y_k(1-y_k)}\right)\ ,$$
the sign being such that after changing coordinates via $f$ we obtain:
$$  I^{\sv}  =  (2\pi i)^{3-|S'|}  \int_{\overline{\mathcal{M}}_{0,S}(\CC)} 
   \nu_{S}   \wedge
 \overline{\omega_0}   \times \prod_{j=1}^k \int_{\PP^1(\CC)}   
\left(-\frac{d y_j}{y_j(1- y_j)}\right) \wedge d \log (1-  \overline{ x_{j}}\overline{y_j}) 
 \ . $$
Formula (\ref{svImmformula}) follows on applying Lemma \ref{lemmasvlog}.  \end{proof}
 
 \begin{theorem} We have
 \begin{equation} I^{\sv}(\omega, \underline{s}) = I^{\closed} (\omega, \underline{s}) \ .
 \end{equation} 
 In other words, the coefficients in the canonical Laurent expansion of the closed string amplitudes (\ref{renormintOmegaClosed}) are the images of the single-valued projection of the coefficients in any motivic lift of the expansion coefficients of open string amplitudes. 
 \end{theorem}
 
 \begin{proof}
 By (\ref{renormintOmegaopen}), we write the open string amplitude as
 \begin{equation} \label{renormintOmega} 
\int_{X^{\delta}} \Omega    \quad =  \quad \sum_{J \subset \chi_S}  \frac{1}{s_J}  \int_{X_J}   \Omega_J^{\ren} \ , 
\end{equation} 
By Corollary  \ref{corexpansionasconv}, each integrand on the right-hand side admits a Taylor expansion, whose coefficients are products of integrals over moduli spaces, each of which can be lifted to motivic periods by Corollary \ref{cormotivic}. 
Thus 
$$\int_{X_J}   \Omega_J^{\ren}  = \mathrm{per} ( I_J^{\mm}(\underline{s})) $$
for some formal power series $I_J^{\mm}(\underline{s})$ in the $s_c$, $c\in \chi_S$ whose coefficients are effective motivic periods coming from tensor products of universal moduli space motives. 
Since $\s$ is an algebra homomorphism,  (\ref{svImmformula}) yields
$$\s\,  \pi^{\mathfrak{m},\dR} \,  I^{\mm}_J  =  \int_{\overline{\mathcal{M}}_{0,S/J}(\CC)}    (\Omega_J^{\ren})^{\mathrm{closed}}   \ .$$

On the other hand,  by (\ref{renormintOmegaClosed}) these integrals can be repackaged into
$$ \sum_{J \subset \chi_S}  \frac{1}{s_J}  \int_{\overline{\mathcal{M}}_{0,S/J}(\CC)}  (\Omega_J^{\ren})^{\mathrm{closed}} \quad = \quad \int_{\overline{\mathcal{M}}_{0,S}(\CC)} \Omega^{\mathrm{closed}}  \ .$$
 \end{proof}
 By abuse of notation, we may express the previous theorem as the formula
$$\s  \int_{X^{\delta}} \left(\prod_{c\in \chi_{S}} u_c^{s_c} \right) \omega  = (2\pi i)^{3-|S|} \int_{\overline{\mathcal{M}}_{0,S}(\CC)}   \left(\prod_{c\in \chi_{S}} |u_c|^{2 s_c} \right) \nu_S \wedge \overline{\omega} $$
which is equivalent to the form conjectured  in \cite{stiebergersvMZV}.

    \begin{remark}
Define a motivic version of closed string amplitudes by setting
$$I^{\sv, \mm}(\omega, \underline{s}) =  \s^{\mm} I^{\mathfrak{m},\dR}  (\omega, \underline{s})$$
where $\s^{\mm}$ was defined in  \cite[(4.3)]{brownnotesmot} (see also \cite[Remark 2.10]{BD1}). Its period is $\mathrm{per} \, (I^{\sv, \mm}(\omega, \underline{s})) = I^{\sv}(\omega, \underline{s}),  $
which coincides with the closed string amplitude, by the previous theorem. 
This object is of interest because it immediately implies a compatibility between the actions of the motivic Galois group on the open and closed string amplitudes.
    \end{remark}

\section{Background on (co)homology  of  \texorpdfstring{$\mathcal{M}_{0,S}$}{M0,S} with coefficients}  \label{sect: CoeffBackground}

Most, if not all, of the results reviewed below are taken from the literature. A proof that the Parke--Taylor forms are a basis for cohomology with coefficients can be found in the appendix. See
\cite{  kitayoshida, kitayoshida2, ChoMatsumoto,  KMatsumoto, mimachiyoshida}   and references therein for  more details.

\subsection{Koba--Nielsen connection and local system}
Let $S$ be a finite set with $|S|=N=n+3 \geq 3$. 
Let ${\underline{s}}=(s_{ij})$ be a solution to the momentum conservation equations (\ref{MC}).
Let 
\begin{equation} \QQ^{\B}_{\underline{s}} =  \QQ ( e^{2 \pi i s_{kl}}) \qquad \hbox{ and } \qquad \QQ^{\dR}_{\underline{s}} = \QQ(s_{kl})
\end{equation}
be the  subfields of $\CC$ generated by the $\exp( 2\pi i s_{kl})$ and $s_{kl}$ respectively.

\begin{definition} Let $\Or_S$ denote the structure sheaf on $\mathcal{M}_{0,S} \times_{\QQ} \QQ^{\dR}_{\underline{s}}$. 
The \emph{Koba--Nielsen connection} \cite{kobanielsen} is  the logarithmic connection on $\Or_S$ defined by 
$$\nabla_{\underline{s}}: \Or_S \To \Omega^1_S  \qquad \hbox{ where } \qquad \nabla_{\underline{s}} =    d +  \omega_{\underline{s}}$$
and $\omega_{\underline{s}}$ was defined in (\ref{omegasdef}). The \emph{Koba--Nielsen local system} is the $\QQ_s^\B$-local system of rank one on $\mathcal{M}_{0,S}(\CC)$ defined by
$$\mathcal{L}_{\underline{s}} = \QQ^\B_{\underline{s}} \prod_{1\leq i<j\leq N}   (p_j-p_i)^{-s_{ij}}\ .$$
\end{definition}

Since $\omega_{\underline{s}}$ is  a closed one-form,  the connection $\nabla_{\underline{s}}$ is integrable. The horizontal sections of the analytification $(\mathcal{O}_S^{\mathrm{an}},\nabla^{\mathrm{an}}_{\underline{s}})$ define a rank one local system over the complex numbers that is naturally isomorphic to the complexification of $\mathcal{L}_{\underline{s}}$:
\begin{equation}\label{eq: horizontal sections KN}
\left(\mathcal{O}_S^{\mathrm{an}}\right)^{\nabla_{\underline{s}}^{\mathrm{an}}} \simeq \mathcal{L}_{\underline{s}}\otimes_{\QQ_{\underline{s}}^\B}\CC \ .
\end{equation}
We will also consider the dual of the Koba--Nielsen local system
\begin{equation}\label{eq: dual KN local system}
\mathcal{L}_{\underline{s}}^\vee = \QQ_{\underline{s}}^\B\prod_{1\leq i<j\leq N}(p_j-p_i)^{s_{ij}} \simeq \mathcal{L}_{-\underline{s}}\ .
\end{equation}
Let $(S,\delta)$ be a dihedral structure. In dihedral  coordinates, 
$$\nabla_{\underline{s}} = d + \sum_{c \in \chi_{S,\delta}} s_{c}  \frac{d u_c}{u_c}
\qquad \hbox{ and } \qquad  \mathcal{L}_{\underline{s}} = \QQ^\B_{\underline{s}} \prod_{c \in \chi_{S,\delta}}   u_c^{-s_{c}}\ .$$

\begin{definition} A solution to the momentum conservation equations (\ref{MC}) is \emph{generic} if \begin{equation} \label{assumptions} \sum_{i, j \in I} s_{ij} \notin \ZZ  
\end{equation} 
for every subset $I\subset S$ with $|I|\geq 2$, and $|S\backslash I|\geq 2$.

\end{definition}

 \begin{remark}  Write $H= H^1_{\dR}(\mathcal{M}_{0,S}/\QQ)$. By formality (\ref{formality}) and (\ref{H1description}),  and Lemma \ref{lem: VNtoH1isom},   the form $\omega_{\underline{s}}$ is the specialisation of the universal abelian one-form
$$ \omega \quad \in  \quad   H^{\vee}  \otimes \Omega_{S}^1 \cong  H^{\vee} \otimes H  $$
which represents the identity in $H^{\vee} \otimes H \cong \mathrm{End}(H)$.
\end{remark} 

\begin{remark} \label{remark: KZvsKN} The formal one-form $\omega$
 defines a logarithmic connection on the universal enveloping algebra of the braid Lie algebra.  It is the universal connection on   the affine ring of the unipotent de Rham fundamental group $\pi_1^{\dR}(\mathcal{M}_{0,S})$:
$$\nabla_{\mathrm{KZ}} :  \Or(\pi_1^{\dR}(\mathcal{M}_{0,S})) \To       \Omega^1_S \otimes \Or(\pi_1^{\dR}(\mathcal{M}_{0,S}))$$
 The Koba--Nielsen connection (viewed as a connection over the field $\QQ_{\underline{s}}^{\dR}$, i.e., for the universal solution of the momentum-conservation equations) is its abelianisation.  
 Given any particular complex solution to the moment conservation equations, the latter specialises to a connection over $\CC$.
 \end{remark}

\subsection{Singular (co)homology} Denote the (singular)   homology, 
locally finite (Borel-Moore) homology,  cohomology, and cohomology with compact supports  of $\mathcal{M}_{0,S}$ with coefficients in $\mathcal{L}_{\underline{s}}$ by 
$$ \ H_k(\mathcal{M}_{0,S}, \mathcal{L}_{\underline{s}}) \ , H_k^{\lf}(\mathcal{M}_{0,S}, \mathcal{L}_{\underline{s}})  \ ,  \ H^k(\mathcal{M}_{0,S}, \mathcal{L}_{\underline{s}}) \ ,  \ H_\c^k(\mathcal{M}_{0,S}, \mathcal{L}_{\underline{s}})\ .$$   They   are finite-dimensional $\QQ^\B_{\underline{s}}$-vector spaces. The  second    is  the cohomology of the complex of formal infinite sums of cochains with coefficients in $\mathcal{L}_{\underline{s}}$ whose  restriction to any compact subset have only finitely many non-zero terms.

Because of \eqref{eq: dual KN local system}, duality between homology and cohomology gives rises to canonical isomorphisms of $\QQ^\B_{\underline{s}}$-vector spaces for all $k$:
$$  H_k( \mathcal{M}_{0,S}, \mathcal{L}_{-\underline{s}}) \simeq H^k(\mathcal{M}_{0,S}, \mathcal{L}_{\underline{s}})^\vee \quad , \quad 
 H^{\lf}_k( \mathcal{M}_{0,S}, \mathcal{L}_{-\underline{s}})\simeq H_\c^k(\mathcal{M}_{0,S}, \mathcal{L}_{\underline{s}})^\vee\ .$$

\begin{proposition}  \label{propReghomology}  If the $s_{ij}$ are generic in the sense of (\ref{assumptions}),  then the  natural maps  induce the following  isomorphisms:
\begin{eqnarray} H_\c^k(\mathcal{M}_{0,S}, \mathcal{L}_{\underline{s}})   & \overset{\sim}{\To} & H^k(\mathcal{M}_{0,S}, \mathcal{L}_{\underline{s}})  \label{compactsupptousual} \\
 \label{homologytolf}   H_k ( \mathcal{M}_{0,S}, \mathcal{L}_{\underline{s}})& \overset{\sim}{\To} &  H_k^{\lf} ( \mathcal{M}_{0,S}, \mathcal{L}_{\underline{s}}) \ . 
\end{eqnarray} 
\end{proposition}

\begin{proof} 
Let $j:\mathcal{M}_{0,S} \hookrightarrow \overline{\mathcal{M}}_{0,S}$ denote the open immersion. We claim that the natural map $j_!\mathcal{L}_{\underline{s}}\rightarrow Rj_*\mathcal{L}_{\underline{s}}$ is an isomorphism in the derived category of the category of sheaves on $\overline{\mathcal{M}}_{0,S}$, which amounts to the fact that $Rj_*\mathcal{L}_{\underline{s}}$ has zero stalk at any point of $\partial\overline{\mathcal{M}}_{0,S}$. Since $\partial\overline{\mathcal{M}}_{0,S}$ is a normal crossing divisor we are reduced, by K\"{u}nneth, to proving that $\mathcal{L}_{\underline{s}}$ has non-trivial monodromy around each boundary divisor. 
For a divisor defined by the vanishing of a dihedral coordinate $u_c=0$,  monodromy acts by  multiplication by 
$\exp(-2 \pi i s_c)$, since $\mathcal{L}_{\underline{s}}$ is generated by $\prod_{c\in \chi_{S,\delta}} u_c^{-s_c}$.  Every other boundary divisor on $\mathcal{M}_{0,S}$ is obtained from such a divisor by permuting the elements of $S$.  
From  formula (\ref{sijassc}) and the action of the symmetric group, it follows that the monodromy of $\mathcal{L}_{\underline{s}}$ around a divisor $D$ defined by a partition $S= S_1 \sqcup S_2$ is given by $\exp(-2 \pi i s_D)$ where $s_D=\sum_{i,j \in S_1} s_{ij}=\sum_{i,j\in S_2}s_{ij}$.  
It is non-trivial if and only if $s_D\notin \ZZ$, which is equation (\ref{assumptions}). The first statement follows by applying $R\Gamma_\c\simeq R\Gamma$ to the isomorphism $j_!\mathcal{L}_{\underline{s}} \overset{\sim}{\rightarrow} Rj_*\mathcal{L}_{\underline{s}}$.
 The second statement is dual to the first. 
\end{proof}

 \begin{remark}\label{rem: artin vanishing twisted}
 Under the assumptions \eqref{assumptions}, Artin vanishing and duality imply that all homology and cohomology groups  in Proposition \ref{propReghomology} vanish if $i\neq n$.
 \end{remark}

 The inverse to the isomorphism (\ref{homologytolf}) is sometimes called \emph{regularisation}. For any dihedral structure  $\delta$ on $S$, the function $f_{\underline{s}}=\prod_{c\in \chi_{S,\delta}}u_c^{s_c}$ is well-defined on the domain $X^\delta\subset \mathcal{M}_{0,S}(\RR)$ and defines a class $[X^\delta\otimes f_{\underline{s}}]$ in $H_n^{\lf}(\mathcal{M}_{0,S},\mathcal{L}_{-\underline{s}})$. (Strictly speaking, this class is represented by the infinite sum of the simplices of a fixed locally finite triangulation of $X^\delta$, and does not depend on the choice of triangulation.) Its image  under the regularisation map, assuming \eqref{assumptions}, defines a class we  abusively also  denote by  $[X^\delta\otimes f_{\underline{s}}]\in H_n(\mathcal{M}_{0,S},\mathcal{L}_{-\underline{s}})$. The following result is classical.
 
 \begin{proposition}\label{prop: basis betti coefficients}
 Assume that the $s_{ij}$ are generic in the sense of \eqref{assumptions}. Choose three distinct elements $a,b,c\in S$. A basis of $H_n(\mathcal{M}_{0,S},\mathcal{L}_{-\underline{s}})$ is provided by the classes $[X^\delta\otimes f_{\underline{s}}]$, where $\delta$  ranges over the set of  dihedral structures on $S$  with respect to which  $a,b,c$ appear  consecutively, and in that order (or the reverse order).
 \end{proposition}
 
 \begin{proof}
 We may assume that $S=\{1,\ldots,n+3\}$, $(a,b,c)=(n+1,n+2,n+3)$, and work in simplicial coordinates $(t_1,\ldots,t_n)$ by fixing $p_{n+1}=1$, $p_{n+2}=\infty$, $p_{n+3}=0$. These coordinates give an isomorphism of $\mathcal{M}_{0,S}$ with the complement of the hyperplane arrangement in $\mathbb{A}^{n}$ consisting of the hyperplanes $\{t_i=0\}$ and $\{t_i=1\}$ for $1\leq i\leq n$, and $\{t_i=t_j\}$ for $1\leq i<j\leq n$. This arrangement is defined over $\mathbb{R}$ and the real points of its complement is the disjoint union of the domains $X^\delta$ for $\delta$ a dihedral structure on $S$. Such a domain is bounded in $\mathbb{R}^n$ if and only if it is of the form $\{0< t_{\sigma(1)}< \cdots < t_{\sigma(n)}< 1\}$ for some permutation $\sigma\in\Sigma_n$, i.e., if and only if  no simplicial coordinate is adjacent to $\infty$ in the dihedral ordering. Equivalently 
 the points $n+1,n+2,n+3$ are consecutive and in that order (or its reverse) in the dihedral ordering $\delta$. The proposition is thus a special case of \cite[Proposition 3.1.4]{douaiterao}.
 \end{proof}

\subsection{Betti pairing}
Under the assumptions \eqref{assumptions}, Poincar\'{e}--Verdier duality combined with \eqref{compactsupptousual}  defines a perfect pairing of $\QQ_{\underline{s}}^\B$-vector spaces in cohomology
$$ \langle \,\, , \, \rangle^{\B}:   H^n(\mathcal{M}_{0,S}, \mathcal{L}_{\underline{s}}) \otimes_{\QQ_{\underline{s}}^\B}  H^n(\mathcal{M}_{0,S}, \mathcal{L}_{-\underline{s}}) \To  \QQ^\B_{\underline{s}} \ , $$ 
which is dual to a perfect pairing in homology
 $$  \langle \,\, , \, \rangle_{\B}:   H_n(\mathcal{M}_{0,S}, \mathcal{L}_{-\underline{s}}) \otimes_{\QQ_{\underline{s}}^\B} H_n(\mathcal{M}_{0,S}, \mathcal{L}_{\underline{s}})  \To  \QQ^\B_{\underline{s}}\ .$$

If $\sigma\otimes f_{-\underline{s}}$ and $\tau\otimes f_{\underline{s}}$ are locally finite representatives for homology classes in $H_n(\mathcal{M}_{0,S},\mathcal{L}_{\underline{s}})$ and $H_n(\mathcal{M}_{0,S},\mathcal{L}_{-\underline{s}})$ respectively, then the corresponding pairing 
$$\langle [\tau\otimes f_{\underline{s}}],[\sigma\otimes f_{-\underline{s}}]\rangle_\B$$
is the number of intersection points (with signs) of $\sigma$ and $\widetilde{\tau}$ where $[\widetilde{\tau}\otimes f_{\underline{s}}]$ is a regularisation of $[\tau\otimes f_{\underline{s}}]$ and $\widetilde{\tau}$ is in general position with respect to $\sigma$ \cite{kitayoshida, kitayoshida2}. The matrix of the cohomological Betti pairing is the inverse  transpose of the matrix of the  homological Betti pairing. 
 
\subsection{Algebraic de Rham cohomology}
Let $H^k(\mathcal{M}_{0,S}, \nabla_{\underline{s}})$ denote  the algebraic de Rham cohomology of $\mathcal{M}_{0,S}$ with coefficients in the algebraic vector bundle  $(\Or_S , \nabla_{\underline{s}})$
 with integrable connection over $\mathcal{M}_{0,S}\times_{\QQ} \QQ^{\dR}_{\underline{s}}$. It is a finite-dimensional $\QQ^\dR_{\underline{s}}$-vector space. Let $(\Or_S^{\mathrm{an}},\nabla_{\underline{s}}^{\mathrm{an}})$ denote  the analytic rank one vector bundle with connection on $\mathcal{M}_{0,S}(\CC)$ obtained from $(\Or_S,\nabla_{\underline{s}})$. We have an isomorphism
  \begin{equation}\label{eq: algebraic analytic de Rham KN}
  H^k(\mathcal{M}_{0,S},\nabla_{\underline{s}})\otimes_{\QQ_{\underline{s}}^\dR}\CC \simeq H^k(\mathcal{M}_{0,S}(\CC),\nabla_{\underline{s}}^{\mathrm{an}})\ ,
  \end{equation}
  where the right-hand side denotes the cohomology of the complex of global smooth differential forms on $\mathcal{M}_{0,S}(\CC)$ with differential $\nabla_{\underline{s}}$. Recall from \cite[\S 3]{BD1} the notation $\A^\bullet_{\overline{\mathcal{M}}_{0,S}}(\log\partial\overline{\mathcal{M}}_{0,S})$ for the complex of sheaves of smooth forms on $\overline{\mathcal{M}}_{0,S}$ with logarithmic singularities along $\partial\overline{\mathcal{M}}_{0,S}$.
  
  \begin{proposition}\label{prop:smooth log forms coefficients}
  Under the assumptions \eqref{assumptions} we have a natural isomorphism
  $$H^k(\mathcal{M}_{0,S}(\CC),\nabla_{\underline{s}}^{\mathrm{an}}) \simeq H^k(\Gamma(\overline{\mathcal{M}}_{0,S},\A^\bullet_{\overline{\mathcal{M}}_{0,S}}(\log\partial\overline{\mathcal{M}}_{0,S})),\nabla_{\underline{s}}^{\mathrm{an}})\ .$$
  \end{proposition}
  
  \begin{proof}
  This is a smooth version of \cite[Proposition 3.13]{deligneEqDiff}. The assumptions of [\emph{loc. cit.}] are implied by \eqref{assumptions} and one can check that its proof can be copied in the smooth setting.
  \end{proof}
  
  By a classical argument due to Esnault--Schechtman--Viehweg \cite{esnaultschechtmanviehweg}, we can replace global logarithmic smooth forms with global algebraic smooth forms and the cohomology group $H^k(\mathcal{M}_{0,S},\nabla_{\underline{s}})$ is given, under the assumptions \eqref{assumptions}, by the cohomology of the complex $(\Omega^\bullet_S\otimes\QQ_{\underline{s}}^\dR \; , \; \omega_{\underline{s}}\wedge -)$. In particular, $H^n(\mathcal{M}_{0,S},\nabla_{\underline{s}})$ is simply the quotient of $\Omega^n_S\otimes\QQ_{\underline{s}}^\dR$ by the subspace spanned by the elements $\omega_{\underline{s}}\wedge\varphi$ for $\varphi\in \Omega^{n-1}_S$. The following theorem gives a basis of that quotient.
  
  \begin{theorem} \label{thm: AomotoBasis}
  Assume that the $s_{ij}$ are generic in the sense of \eqref{assumptions}. A basis of $H^n(\mathcal{M}_{0,S},\nabla_{\underline{s}})$ is provided by the classes of the differential forms 
  $$\frac{dt_1\wedge\cdots \wedge dt_n}{\prod_{k=1}^n(t_k-t_{i_k})}$$
  for the tuples $(i_1,\ldots,i_n)$ with $0\leq i_k\leq k-1$ and where we set $t_0=0$.
  \end{theorem}
  
  \begin{proof}
  This is a special case of \cite[Theorem 1]{aomotogaussmanin}.
  \end{proof}

 The following more symmetric basis is more prevalent in the string theory literature. Its elements are called \emph{Parke--Taylor factors} \cite{parketaylor}. Therefore  we shall refer to it as  the \emph{Parke--Taylor basis}, as opposed to the \emph{Aomoto basis} of Theorem \ref{thm: AomotoBasis}.
 Although the following theorem is frequently referred to in the literature, we could not find a complete proof for it and therefore provide one in Appendix \ref{appendix}.
  
  \begin{theorem}\label{thm: Parke Taylor}
  Assume that the $s_{ij}$ are generic in the sense of \eqref{assumptions}. A basis of $H^n(\mathcal{M}_{0,S},\nabla_{\underline{s}})$ is provided by the classes of the differential forms 
  $$\frac{dt_1\wedge\cdots \wedge dt_n}{\prod_{k=1}^{n+1} (t_{\sigma(k)}-t_{\sigma(k-1)})}$$
  for permutations $\sigma\in\Sigma_n$, where we set $t_{\sigma(0)}=0$ and $t_{\sigma(n+1)}=1$.
  \end{theorem}
  
  Other bases can be found in the literature, e.g., the $\beta$nbc bases of Falk--Terao \cite{falkterao}.
 
\subsection{de Rham pairing} 

Under the assumptions \eqref{assumptions} there is an algebraic de Rham version of the intersection pairing, which is a perfect pairing
$$\langle \,\, , \, \rangle^{\dR} : H^n(\mathcal{M}_{0,S}, \nabla_{\underline{s}}) \otimes_{\QQ^{\dR}_{\underline{s}}}  H^n(\mathcal{M}_{0,S}, \nabla_{-\underline{s}}) \To \QQ^{\dR}_{\underline{s}}$$
of $\QQ_{\underline{s}}^\dR$-vector spaces. 
This is easily checked after extending the scalars to $\CC$ by working with smooth de Rham complexes. The only part to check is that this pairing is algebraic, i.e., defined over $\QQ^{\dR}_{\underline{s}}.$  Indeed, it can be defined algebraically (see  \cite{ChoMatsumoto} for the general case of curves), and computed explicitly 
 for hyperplane arrangements  \cite{KMatsumoto}, which contains the present situation as a special case.
 
\medskip

If  $\omega ,\nu \in \QQ^{\dR}_{\underline{s}} \otimes \Omega_S^n$ are logarithmic $n$-forms, let $\widetilde{\nu}$ be a smooth  $\nabla_{\underline{s}}$-closed  $n$-form on $\mathcal{M}_{0,S}(\CC)$ which represents $[\nu]$ and has compact support. Then 
$$\langle [\nu], [\omega] \rangle^{\dR} = (2\pi i)^{-n}\int_{\mathcal{M}_{0,S}(\CC)} \widetilde{\nu} \wedge  \omega\ .$$
Our normalisation differs from the one in the literature by the factor of $(2\pi i)^{-n}$. 

\subsection{Periods}
Since algebraic de Rham cohomology is defined over $\QQ^{\dR}_{\underline{s}}$, we can meaningfully speak of periods.  
Using \eqref{eq: horizontal sections KN} we see that integration induces a perfect pairing of complex vector spaces: 
\begin{eqnarray} 
H_n(\mathcal{M}_{0,S}, \mathcal{L}_{-\underline{s}}\otimes_{\QQ_{\underline{s}}^\B}\CC) \otimes_\CC H^n(\mathcal{M}_{0,S}(\CC), \nabla_{\underline{s}}^{\mathrm{an}}) &\To &  \CC \nonumber \\
{[}\gamma \otimes f_{\underline{s}}] \otimes [\omega]     & \To & \int_{\gamma} f_{\underline{s}}\, \omega \nonumber
\end{eqnarray} 
which is well-defined by Stokes' theorem. By \eqref{eq: algebraic analytic de Rham KN} it induces an
 isomorphism:
\begin{equation}  \label{comparisonCoeffs}
\mathrm{comp}_{\B, \dR} \ :  \   H^n(\mathcal{M}_{0,S}, \nabla_{\underline{s}}) \otimes_{\QQ^{\dR}_{\underline{s}}} \CC \overset{\sim}{\To}   H^n(\mathcal{M}_{0,S}, \mathcal{L}_{\underline{s}})\otimes_{\QQ^{\B}_{\underline{s}}} \CC\ .
\end{equation} 
We will use the notation $\mathrm{comp}_{\B,\dR}(\underline{s})$ when we want to make the dependence on $\underline{s}$ explicit.
If we choose a $\QQ^{\dR}_{\underline{s}}$-basis of the left-hand vector space, and a $\QQ^{\B}_{\underline{s}}$-basis of the right-hand vector space, the  isomorphism $\mathrm{comp}_{\B,\dR}(\underline{s})$ can be expressed as a matrix
$P_{\underline{s}}$, and we will sometimes abusively use the notation $P_{\underline{s}}$ instead of $\mathrm{comp}_{\B,\dR}(\underline{s})$.

\begin{theorem}\label{thm: twisted period relations} (Twisted period relations \cite{kitayoshida, ChoMatsumoto})  Assume that the $s_{ij}$ are generic in the sense of \eqref{assumptions}. Let $\omega,\nu\in\QQ_{\underline{s}}^\dR\otimes \Omega_S^n$ be logarithmic $n$-forms giving rise to classes in $H^n(\mathcal{M}_{0,S},\nabla_{-\underline{s}})$ and $ H^n(\mathcal{M}_{0,S},\nabla_{\underline{s}})$ respectively. We have the equality:
\begin{equation}\label{IdPIbP}
(2\pi i)^n \langle [\nu],[\omega]\rangle^\dR = \langle P_{\underline{s}}[\nu],P_{-\underline{s}}[\omega]\rangle^\B\ ,
\end{equation}
where the cohomological Betti pairing is naturally extended by $\CC$-linearity.
\end{theorem} 

\begin{proof}
This follows from the fact that the (iso)morphisms \eqref{compactsupptousual} and \eqref{homologytolf} and Poincar\'{e}--Verdier duality are compatible with the comparison isomorphisms.
\end{proof}

The reason for the factor $(2 \pi i)^n$ in the formula \eqref{IdPIbP} is because of our insistence that the de Rham intersection pairing $I_{\dR}$ be algebraic and have entries in $\QQ^{\dR}_{\underline{s}}.$ 

\begin{proposition}
Assume that the $s_{ij}$ are generic in the sense of \eqref{assumptions}. Let $\delta$ be a dihedral structure on $S$ and let $\omega\in \QQ_{\underline{s}}^\dR\otimes\Omega^n_S$ be a regular logarithmic form on $\mathcal{M}_{0,S}$ of top degree. If the inequalities of Proposition \ref{lemconv} hold then we have
$$\langle [X^\delta\otimes f_{\underline{s}}] , \mathrm{comp}_{\B, \dR}[\omega]\rangle = \int_{X^\delta}f_{\underline{s}}\,\omega\ .$$
\end{proposition}

\begin{proof}
\begin{enumerate}[1)]
\item We first prove that the formula holds for any algebraic $n$-form $\omega$ on $\mathcal{M}^\delta_{0,S}$ with logarithmic singularities along $\partial\mathcal{M}^\delta_{0,S}$, provided $\mathrm{Re}(s_c)>0$ for every chord $c\in\chi_{S,\delta}$. By definition we have for every $\underline{s}$ the formula:
$$\langle [X^\delta\otimes f_{\underline{s}}] , \mathrm{comp}_{\B, \dR}[\omega]\rangle = \int_{X^\delta}f_{\underline{s}}\,\widetilde{\omega}\ ,$$
where $\widetilde{\omega}$ is a global section of $\mathcal{A}^n_{\overline{\mathcal{M}}_{0,S}}(\log\partial\overline{\mathcal{M}}_{0,S})$ with compact support which is cohomologous to $\omega$, i.e., such that $\omega-\widetilde{\omega}=\nabla_{\underline{s}}\phi$, with $\phi$ a global section of $\mathcal{A}^{n-1}_{\overline{\mathcal{M}}_{0,S}}(\log\partial\overline{\mathcal{M}}_{0,S})$. Thus, we need to prove that the integral of $f_{\underline{s}}\nabla_{\underline{s}}\phi = d(f_{\underline{s}}\phi)$ on $X^\delta$ vanishes if $\mathrm{Re}(s_c)>0$ for all chords $c\in\chi_{S,\delta}$. We note that in general $f_{\underline{s}}\phi$ has singularities along the boundary of $X^\delta$ unless $\mathrm{Re}(s_c)>1$ for all chords $c\in \chi_{S,\delta}$, so that we cannot apply Stokes' theorem directly.
We can write
$$\phi = \sum_{J}  \left(\phi_J   \wedge\bigwedge_{c\in J}\frac{du_c}{u_c}\right)$$
where the sum is over subsets of chords $J \subset \chi_{S,\delta}$ and $\phi_J$ extends to a smooth form on $\mathcal{M}_{0,n}^\delta$ (i.e., has no poles along the boundary of $X^{\delta}$).
By properties of dihedral coordinates, we can furthermore assume that $\phi_J=0$ if $J$ contains two crossing chords. Indeed, a form $\bigwedge_{c\in J}\frac{du_c}{u_c}$ extends to a regular form on $\mathcal{M}_{0,n}^\delta$ if every chord in $J$ is crossed by another chord in $J$ by \eqref{ucprimecrossingc}. It is therefore sufficient to consider a single term given by a set $J\subset \chi_{S,\delta}$ consisting of chords that do not cross. We write $\phi'=\phi_J$ and set $f'_{\underline{s}}=\prod_{c\notin J}u_c^{s_c}$ so that we have 
$$d(f_{\underline{s}}\phi)=d(f'_{\underline{s}}\phi') \wedge \left(\bigwedge_{c\in J}u_c^{s_c}\frac{du_c}{u_c} \right)\ .$$
The forgetful maps \eqref{fDXdelta} give rise to a diffeomorphism $X^\delta\simeq (0,1)^k\times X_J$ with $k=|J|$ and $X_J=X^{\delta_0}\times\cdots\times X^{\delta_k}$. We can thus write
$$\int_{X^\delta}d(f_{\underline{s}}\phi) = \pm \int_{(0,1)^k}\left(\bigwedge_{i=1}^k x_i^{s_{c_i}}\frac{dx_{i}}{x_i} \int_{X_J}d(f_{\underline{s}}'\phi')\right)\ ,$$
where the $x_i$ are the coordinates on $(0,1)^k$, corresponding to the dihedral coordinates $u_{c_i}$ for $J=\{c_1,\ldots,c_k\}$. Now the boundary of $X_J$ has components $\{u_c=0\}$ for $c$ a chord that does not cross any chord in $J$. Thus, if $\mathrm{Re}(s_c)>0$ for every chord $c$ we have that $f_{\underline{s}}'$, and hence $f_{\underline{s}}' \phi'$, vanishes on the boundary of $X_J$, and the inner integral is zero by Stokes' theorem for manifolds with corners.
\item Now, for $\omega$ as in the statement of the proposition, let $D_{c_1},\ldots,D_{c_r}$ be the divisors along which $\omega$ does not have a pole. Applying the first step of the proof to the product $\left(\prod_{i=1}^ru_{c_i}^{-1}\right)\omega$ yields the result.
\end{enumerate}
\end{proof}

\begin{example} \label{example: M04withcoeffs} Let $|S|=4$. Then 
$H_1(\mathcal{M}_{0,S}, \mathcal{L}_{-s,-t}) \cong H^{\lf}_1(\mathcal{M}_{0,S}, \mathcal{L}_{-s,-t})$
is one-dimensional. The locally finite homology  is spanned by the class of 
 $\sigma \otimes x^s(1-x)^t $
where $\sigma $ is the open interval $(0,1)$. 
The algebraic  de Rham cohomology group 
$H^1(\mathcal{M}_{0,S}, \nabla_{\underline{s}}) \cong \QQ^{\dR}_{\underline{s}}[\nu]  $
is one-dimensional spanned by the class of $\nu = - \nu_S$ where
$$\nu =  \frac{dx}{x(1-x)} \quad \in  \quad \Omega^1_{S}\ . $$
The period matrix $P_{\underline{s}}$ is the $1\times 1$ matrix whose  entry is the beta function:
\begin{equation} P_{\underline{s}} \;  =  \; \left(\int_{0}^1 x^s (1-x)^t \frac{dx}{x(1-x)} \right) \; = \; \left(\,\beta(s,t)\,\right) \; = \;  \left(\frac{\Gamma (s )\Gamma(t)}{\Gamma(s+t)} \right) \ .
\end{equation}
For any small $\varepsilon>0$, a representative for the regularisation of $[\sigma \otimes x^s (1-x)^t]$ is
$$  \Big(\frac{S_0(\varepsilon)}{e^{2\pi i s}-1}  + [\varepsilon, 1- \varepsilon] -   \frac{S_1(\varepsilon)}{e^{2\pi i t}-1} \Big) \otimes x^s(1-x)^t$$ 
where $S_i(\varepsilon)$ denotes the small circle of radius $\varepsilon$ winding positively around $i$.
From this one easily deduces the intersection product with the class of  $\sigma \otimes x^{-s}(1-x)^{-t}$. It is 
\begin{equation} \label{rank1Bettipairing}
 \langle [\sigma \otimes x^{s} (1-x)^{t}], [\sigma \otimes x^{-s} (1-x)^{-t}] \rangle_{\B}     =      
\frac{ 1 - e^{ 2\pi i (s+t) }}{ (1- e^{2\pi i s} )(1-e^{2\pi i t})}   =  \frac{i}{2}  \frac{\sin (\pi (s+t))}{\sin (\pi s) \sin (\pi t)}\ \cdot
 \end{equation}
 See, e.g.,  \cite{ChoMatsumoto}, \cite[\S 2]{kitayoshida}, or \cite[\S 2]{mimachiyoshida}. Dually:
 $$\langle[\sigma\otimes x^{s}(1-x)^{t}]^\vee,[\sigma\otimes x^{-s}(1-x)^{-t}]^\vee\rangle^\B = \frac{2}{i} \, \frac{\sin (\pi s) \sin (\pi t)}{\sin \pi (s+t)}\ \cdot$$

 On the other hand, the de Rham intersection pairing \cite{KMatsumoto} is
\begin{equation} \label{dRpairingS=4} \langle [\nu], [\nu] \rangle^{\dR} = \frac{1}{s} + \frac{1}{t}\ \cdot
\end{equation}
In this case, equation (\ref{IdPIbP}) reads
\begin{equation}\label{eq: functional equation beta} 
 2 \pi i\, \Big(\frac{1}{s} + \frac{1}{t}\Big) =\beta(-s,-t)\, \beta(s,t) \,\frac{2}{i}\frac{\sin (\pi s) \sin (\pi t)}{\sin \pi (s+t)}
 \end{equation}
 using (\ref{rank1Bettipairing}) and (\ref{dRpairingS=4}), as observed in  \cite{ChoMatsumoto}, or in terms of the gamma function:
 \begin{equation}\label{eq: functional equation beta gamma} 
 \frac{\Gamma(s)\Gamma(t)}{\Gamma(s+t)} \frac{\Gamma(-s)\Gamma(-t)}{\Gamma(-s-t)}= -\pi\,\frac{(s+t)\sin(\pi(s+t))}{s\sin(\pi s)\, t\sin(\pi t)}\ \cdot
 \end{equation}
  This  can easily be deduced  from the well-known functional equation for the gamma function  $\Gamma(s) \Gamma(-s) = - \frac{\pi}{s \sin (\pi s)}$, and is in fact equivalent to it (set $t=-s/2$).

\end{example}

\subsection{Self-duality} It is convenient to reformulate the above relations as a statement about self-duality. Consider  the object
\begin{eqnarray} M_{\dR}  &=  & H^n( \mathcal{M}_{0,S} , \nabla_{\underline{s}}) \oplus H^n(\mathcal{M}_{0,S},\nabla_{- \underline{s}}) \nonumber \\ 
M_{\B} & = &  H^n( \mathcal{M}_{0,S} , \mathcal{L}_{\underline{s}}) \oplus H^n(\mathcal{M}_{0,S}, \mathcal{L}_{-\underline{s}}) \nonumber 
\end{eqnarray}
and 
denote the comparison
$$P= P_{\underline{s}} \oplus P_{-\underline{s}} :   M_{\dR} \otimes_{\QQ_{\underline{s}}^\dR} \CC \overset{\sim}{\To} M_{\B}\otimes_{\QQ_{\underline{s}}^\B} \CC$$
The results in the previous section  can be summarised by saying that the triple of  objects 
$(M_{\dR}, M_{\B}, P)$
is self-dual. In other words, the Betti and de Rham pairings induce isomorphisms
$$I_{\dR} : M_{\dR} \cong M_{\dR}^{\vee} \qquad \hbox{ and } \qquad I_{\B} : M_{\B} \cong M_{\B}^{\vee}$$
which are compatible with the comparison isomorphism $P$. With these notations, equation (\ref{IdPIbP}) can be written in the simpler form: 
 \begin{equation}(2\pi i)^n I_{\dR}  = P^{\vee} I_{\B} P \ .
 \end{equation}

\section{Single-valued periods for cohomology with coefficients} \label{sect: SectCohomcoeffs}

    We fix a solution $(s_{ij})$ of the momentum conservation equations over the complex numbers.

\subsection{Complex conjugation and the single-valued period map}
We can define and compute a period pairing on de Rham cohomology classes by transporting complex conjugation which is the anti-holomorphic diffeomorphism:
 \begin{equation}\label{eq: complex conjugation local systems}
 \mathrm{conj}:\mathcal{M}_{0,S}(\CC) \To \mathcal{M}_{0,S}(\CC)\ .
 \end{equation}
 
 Since it reverses the orientation of simple closed loops, and since a rank one local system on $\mathcal{M}_{0,S}(\CC)$ is determined by a representation of the abelian group $H_1(\mathcal{M}_{0,S}(\CC))$  we see that we have an isomorphism of local systems: 
 \begin{equation}\label{eq: iso local systems conj}
 \mathrm{conj}^*\mathcal{L}_{\underline{s}}\simeq \mathcal{L}_{-\underline{s}}\ .
 \end{equation} We thus get a morphism of local systems on $  \mathcal{M}_{0,S}(\CC)$: 
 $$\mathcal{L}_{\underline{s}} \To \mathrm{conj}_*\mathrm{conj}^*\mathcal{L}_{\underline{s}} \simeq \mathrm{conj}_*\mathcal{L}_{-\underline{s}}\ ,$$
 which at the level of cohomology induces a morphism of $\QQ_{\underline{s}}^{\mathrm{B}}$-vector spaces
 $$F_\infty: H^n(\mathcal{M}_{0,S}, \mathcal{L}_{\underline{s}}) \To H^n(\mathcal{M}_{0,S},\mathcal{L}_{-\underline{s}})\ .$$
 We call $F_\infty$ the \emph{real Frobenius} or \emph{Frobenius at the infinite prime}. We will use the notation $F_\infty(\underline{s})$ when we want to make dependence on $\underline{s}$ explicit. One checks that the Frobenius is involutive: $F_{\infty}(-\underline{s}) F_{\infty}(\underline{s}) =\id$.
 
 \begin{remark}\label{rem: real frobenius homology} 
 The isomorphism \eqref{eq: iso local systems conj} is induced by the trivialisation of the tensor product $\mathrm{conj}^*\mathcal{L}_{\underline{s}}\otimes \mathcal{L}_{\underline{s}}$ given by the section 
 $$g_{\underline{s}} = \prod_{i<j}|p_j-p_i|^{-2s_{ij}} = \prod_{i<j}(\overline{p_j}-\overline{p_i})^{-s_{ij}}\cdot \prod_{i<j}(p_j-p_i)^{-s_{ij}}  \ .$$
 Thus, the action of real Frobenius on homology
 $$F_\infty:H_n(\mathcal{M}_{0,S},\mathcal{L}_{\underline{s}})\To H_n(\mathcal{M}_{0,S},\mathcal{L}_{-\underline{s}})$$
 is given by the formula
  $$\sigma\otimes \prod_{i<j}(p_j-p_i)^{-s_{ij}} \;\mapsto\;\overline{\sigma}\otimes \prod_{i<j}(\overline{p_j}-\overline{p_i})^{-s_{ij}}\, g_{\underline{s}}^{-1} = \overline{\sigma}\otimes \prod_{i<j}(p_j-p_i)^{s_{ij}}\ .$$
 \end{remark}
 
 \begin{remark}
 A morphism similar to $F_\infty$ was considered in \cite{hanamurayoshida} and leads to similar formulae but has a different definition. Our definition only uses the action of complex conjugation  on the complex points of  $\mathcal{M}_{0,S}$, whereas the definition in [\emph{loc. cit.}] conjugates the field of coefficients of the local systems. Note that our definition does not require the $s_{ij}$ to be real.
 \end{remark}

\begin{definition}\label{definitionsvforcoefficients}
The \emph{single-valued period map} is the $\CC$-linear isomorphism 
$$\s : H^n(\mathcal{M}_{0,S},\nabla_{\underline{s}})\otimes_{\QQ^{\mathrm{dR}}_{\underline{s}}}\CC\To H^n(\mathcal{M}_{0,S},\nabla_{-\underline{s}})\otimes_{\QQ^{\mathrm{dR}}_{\underline{s}}}\CC$$
defined as the composite 
$$\s = \mathrm{comp}_{\mathrm{B},\mathrm{dR}}^{-1}(-\underline{s}) \circ (F_\infty\otimes \mathrm{id}) \circ \mathrm{comp}_{\mathrm{B},\mathrm{dR}}(\underline{s}) \ .$$
In other words, it is defined by the following commutative diagram:
$$\xymatrixcolsep{5pc}\xymatrix{
 H^n(\mathcal{M}_{0,S},\nabla_{\underline{s}})\otimes_{\QQ^{\mathrm{dR}}_{\underline{s}}}\CC \ar[d]_{\s} \ar[r]^-{\mathrm{comp}_{\mathrm{B},\mathrm{dR}}(\underline{s})} & H^n(\mathcal{M}_{0,S},\mathcal{L}_{\underline{s}})\otimes_{\QQ_{\underline{s}}^{\mathrm{B}}}\CC \ar[d]^{F_\infty\otimes \mathrm{id}}\\
H^n(\mathcal{M}_{0,S},\nabla_{-\underline{s}})\otimes_{\QQ^{\mathrm{dR}}_{\underline{s}}}\CC \ar[r]_{\mathrm{comp}_{\mathrm{B},\mathrm{dR}}(-\underline{s})}& H^n( \mathcal{M}_{0,S},\mathcal{L}_{-\underline{s}})\otimes_{\QQ_{\underline{s}}^{\mathrm{B}}}\CC
}$$
\end{definition}

The single-valued period map can be computed explicitly by choosing
a  $\QQ^{\dR}_{\underline{s}}$-bases $\{[\omega]\}$ and $\{[\nu]\}$ for $H^n( \mathcal{M}_{0,S}, \nabla_
{-\underline{s}})$ and $H^n(\mathcal{M}_{0,S},\nabla_{\underline{s}})$ respectively
 and a $\QQ^{\B}_{\underline{s}}$-basis  $\{[\sigma \otimes f_{\underline{s}}]\}$  for     $H_n( \mathcal{M}_{0,S}, \mathcal{L}_{-\underline{s}})$.   
In these bases, the isomorphism (\ref{comparisonCoeffs})  is represented by a  matrix $P_{\underline{s}}$ with  entries 
$$P_{\underline{s}}([\sigma \otimes f_{\underline{s}}],[\nu]) = \int_{\sigma}   f_{\underline{s}}\, \nu \ .$$
By Remark \ref{rem: real frobenius homology} the entries of $F_\infty P_{-\underline{s}}$ are 
$$(F_{\infty} P_{-\underline{s}})([\sigma\otimes f_{\underline{s}}],[\omega]) = \int_{\overline{\sigma}}  f_{-\underline{s}}\,\omega \  .$$
The single-valued period matrix (the matrix of $\s$) is then the product 
$$P_{-\underline{s}}^{-1}(F_\infty P_{\underline{s}})=(F_\infty P_{-\underline{s}})^{-1}P_{\underline{s}}\ .$$
This formula is often impractical because one needs to compute \emph{all} the entries of the period matrix in order to compute any single entry of the single-valued period matrix.

\begin{example}\label{ex: sv beta definition}
With the notation of Example \ref{example: M04withcoeffs} we have $\overline{\sigma}=\sigma$ since $(0,1)$ is real, and the single-valued period matrix is 
$$(F_{\infty} P_{-\underline{s}})^{-1} P_{\underline{s}} = \left(\beta(-s,-t)^{-1}\beta(s,t)\right) = \left(\frac{\Gamma(s)\Gamma(t)\Gamma(-s-t)}{\Gamma(s+t)\Gamma(-s)\Gamma(-t)}\right)\ .$$
\end{example}

\subsection{The single-valued period pairing via the Betti pairing} 
If the $s_{ij}$ are generic in the sense of (\ref{assumptions}), the single-valued period map and the de Rham pairing induce a \emph{single-valued period pairing}
$$H^n_{\dR}(\mathcal{M}_{0,S}, \nabla_{\underline{s}}) \otimes_{\QQ^{\dR}_{\underline{s}}} H^n_{\dR}(\mathcal{M}_{0,S}, \nabla_{\underline{s}}) \To \CC $$
given for $\omega,\nu\in \QQ_{\underline{s}}^\dR\otimes\Omega^n_S$ by the formula
$$[\nu]\otimes[\omega] \,\mapsto\, \langle [\nu],\s[\omega]\rangle^{\dR} = \langle [\nu],P_{-\underline{s}}^{-1}F_\infty P_{\underline{s}}[\omega] \rangle^\dR\ .$$
One can use the compatibility between the de Rham and the Betti pairings to express the single-valued pairing in terms of the latter.

\begin{proposition}\label{lem2ndDC} Assume that the $s_{ij}$ are generic in the sense of \eqref{assumptions}. Let $\omega,\nu\in \QQ_{\underline{s}}^\dR\otimes\Omega_S^n$ and denote by $[\omega]$, $[\nu]$ their classes in $H^n(\mathcal{M}_{0,S},\nabla_{\underline{s}})$. The corresponding single-valued period is given by the formula
$$\langle[\nu],\s[\omega]\rangle^\dR =(2\pi i)^{-n}  \langle  P_{\underline{s}} [\nu] ,  F_{\infty} P_{\underline{s}} [\omega] \rangle^{\B}$$
and can be computed explicitly by a sum
$$(2\pi i)^{-n}\sum_{\substack{[\sigma\otimes f_{-\underline{s}}] \\ [\tau\otimes f_{\underline{s}}]}}   \langle [\tau\otimes f_{\underline{s}}]^{\vee}, [\sigma\otimes f_{-\underline{s}}]^{\vee}\rangle^{\B}  \int_{\tau} f_{\underline{s}}\,\nu \,\int_{\overline{\sigma}} f_{\underline{s}}\,\omega     \ ,$$
where $[\sigma\otimes f_{-\underline{s}}]$ and $[\tau\otimes f_{\underline{s}}]$ range over a basis of $H_n(\mathcal{M}_{0,S}, \mathcal{L}_{\underline{s}})$ and $H_n(\mathcal{M}_{0,S}, \mathcal{L}_{-\underline{s}})$ respectively, and $[\sigma\otimes f_{-\underline{s}}]^{\vee}, [\tau\otimes f_{\underline{s}}]^{\vee}$ are the dual bases.
\end{proposition}

\begin{proof}
We have
$$\langle[\nu],\s[\omega]\rangle^\dR = \langle [\nu],P_{-\underline{s}}^{-1}F_\infty P_{\underline{s}}[\omega]\rangle^\dR = (2\pi i)^{-n} \langle P_{\underline{s}}[\nu],F_\infty P_{\underline{s}}[\omega]\rangle^\B\ ,$$
where the first equality is the definition of the single-valued period map, and the second equality follows from Theorem \ref{thm: twisted period relations}. The second formula follows from the definition of $P_{\underline{s}}$ and Remark \ref{rem: real frobenius homology}.
\end{proof}

\begin{example} \label{ex: sv beta pairings} Following up on Example \ref{ex: sv beta definition} and using \eqref{dRpairingS=4} we see that we have 
\begin{equation}\label{eq: sv beta pairing dR}
\langle[\nu],\s[\nu]\rangle^\dR = \beta(-s,-t)^{-1}\beta(s,t)\left(\frac{1}{s}+\frac{1}{t}\right) = - \frac{\Gamma(s)\Gamma(t)\Gamma(1-s-t)}{\Gamma(s+t)\Gamma(1-s)\Gamma(1-t)}\ , 
\end{equation}
where we have used $\Gamma(-x)=-x\,\Gamma(1-x)$. Now using \eqref{rank1Bettipairing}, Proposition \ref{lem2ndDC} reads
\begin{equation}\label{eq: sv beta pairing B}
\langle[\nu],\s[\nu]\rangle^\dR = \frac{1}{2\pi i}\,\frac{2}{i}\frac{\sin(\pi s)\sin(\pi t)}{\sin(\pi(s+t))}\,\beta(s,t)^2 =-\frac{1}{\pi} \frac{\sin(\pi s)\sin(\pi t)}{\sin(\pi (s+t))}\left(\frac{\Gamma(s)\Gamma(t)}{\Gamma(s+t)}\right)^2.
\end{equation}
One deduces \eqref{eq: sv beta pairing B} from \eqref{eq: sv beta pairing dR}, and vice versa, by applying the functional equation \eqref{eq: functional equation beta}.
\end{example} 

\subsection{An integral formula for single-valued periods}
The single-valued period map is a transcendental comparison isomorphism that is naturally interpreted at the level of analytic de Rham cohomology via the isomorphism \eqref{eq: algebraic analytic de Rham KN}.

\begin{lemma}\label{lem:singlevalued analytic}
In analytic de Rham cohomology, the single-valued period map is induced by the morphism of smooth de Rham complexes
\begin{eqnarray*}
    \s^{\mathrm{an}} : (\mathcal{A}^\bullet_{\mathcal{M}_{0,S}(\CC)},\nabla_{\underline{s}}^{\mathrm{an}}) &\To &
    \mathrm{conj}_*(\mathcal{A}^\bullet_{\mathcal{M}_{0,S}(\CC) }, \nabla_{-\underline{s}}^{\mathrm{an}})
\end{eqnarray*}
given on the level of sections by 
$$\mathcal{A}^\bullet_{\mathcal{M}_{0,S}(\CC)}(U)\,\ni\, \omega \quad  \mapsto \quad  \prod_{i<j}|p_j-p_i|^{2s_{ij}}\; \mathrm{conj}^*(\omega)\,\in\, \mathcal{A}^\bullet_{\mathcal{M}_{0,S}(\CC)}(\overline{U})\ .$$
\end{lemma}

\begin{proof}
Recall the notation $g_{\underline{s}}=\prod_{i<j}|p_j-p_i|^{2s_{ij}}$. We first check that $\s^{\mathrm{an}}$ is a morphism of complexes:
\begin{eqnarray*}
\nabla^{\mathrm{an}}_{-\underline{s}}(\s^{\mathrm{an}}(\omega)) & = & \nabla^{\mathrm{an}}_{-\underline{s}}(g_{\underline{s}}\,\mathrm{conj}^*(\omega)) \\
& = & g_{\underline{s}}\left(\left(\sum_{i<j}s_{ij}\,d\log(p_j-p_i)  + \sum_{i<j}s_{ij}\,d\log(\overline{p_j}-\overline{p_i})\right)\wedge\mathrm{conj}^*(\omega) + d(\mathrm{conj}^*(\omega))\right) \\
& & \qquad -\sum_{i<j}s_{ij}\, d\log(p_j-p_i)\wedge (g_{\underline{s}} \,\mathrm{conj}^*(\omega)) \\
& = & g_{\underline{s}}\left(\sum_{i<j} s_{ij}\,d\log(\overline{p_j}-\overline{p_i})\wedge\mathrm{conj}^*(\omega)+d(\mathrm{conj}^*(\omega))\right) \\
& = & g_{\underline{s}}\,\mathrm{conj}^*\left(\sum_{i<j} s_{ij}\, d\log(p_j-p_i)\wedge\omega+ d\omega \right)  \\
& = & \s^{\mathrm{an}}(\nabla^{\mathrm{an}}_{\underline{s}}(\omega))\ .
\end{eqnarray*}
On the level of horizontal sections, we compute:
$$\s^{\mathrm{an}}\left(\prod_{i<j}(p_j-p_i)^{-s_{ij}}\right)  =  g_{\underline{s}}\,\prod_{i<j}(\overline{p_j}-\overline{p_i})^{-s_{ij}} 
= \prod_{i<j}(p_j-p_i)^{s_{ij}}\ .$$
Thus, $\s^{\mathrm{an}}$ induces the isomorphism $\mathcal{L}_{\underline{s}}\rightarrow \mathrm{conj}_*\mathcal{L}_{-\underline{s}}$ and the result follows.
\end{proof}

We can now give an explicit formula for single-valued periods in the case of forms with logarithmic singularities.

\begin{theorem} \label{theoremsvforlogwithcoeffs}
Assume the $s_{ij}$ are generic in the sense of \eqref{assumptions}. Let $\omega \in \QQ_{\underline{s}}^\dR\otimes \Omega^n_S$ and let $\nu_S$ be as in Definition \ref{definitionnuS}. Then $\nu_S$, $\omega$ define de Rham cohomology classes
$$[\nu_S], [\omega]  \quad \in \quad H_{\dR}^n(\mathcal{M}_{0,S}, \nabla_{\underline{s}})\ .$$
If the inequalities stated in Proposition \ref{propIclosedconverges} hold, then the single-valued period of $[\nu_S]\otimes[\omega]$ is given by the absolutely convergent integral
$$\langle[\nu_S],\s[\omega]\rangle^{\dR} = (2\pi i)^{-n} \int_{\overline{\mathcal{M}}_{0,S}(\CC)} \left( \prod_{i<j} |p_j-p_i|^{2s_{ij}} \right)  \nu_S \wedge   \overline{\omega} \, \ .$$
\end{theorem}

\begin{proof}
Recall the notation $g_{\underline{s}} = \prod_{i<j}|p_j-p_i|^{2s_{ij}}$. By definition and Lemma \ref{lem:singlevalued analytic} we have the formula, valid for every generic $(s_{ij})$:
$$\langle[\nu_S],\s[\omega]\rangle^{\dR} = (2\pi i)^{-n} \int_{\overline{\mathcal{M}}_{0,S}(\CC)} g_{\underline{s}}\, \nu_S \wedge   \overline{\widetilde{\omega}}\ ,$$
where $\widetilde{\omega}$ is a global section of $\mathcal{A}^n_{\overline{\mathcal{M}}_{0,S}}(\log\partial\overline{\mathcal{M}}_{0,S})$ with compact support which is cohomologous to $\omega$, i.e., such that $\omega-\widetilde{\omega}=\nabla_{\underline{s}}\phi$, with $\phi$ a global section of $\mathcal{A}^{n-1}_{\overline{\mathcal{M}}_{0,S}}(\log\partial\overline{\mathcal{M}}_{0,S})$. We need to prove that the integral of $g_{\underline{s}}\nu_S\wedge\overline{\nabla_{\underline{s}}\phi} = \pm d(g_{\underline{s}}\nu_S\wedge\overline{\phi})$ on $\overline{\mathcal{M}}_{0,S}(\CC)$ vanishes under the stated assumptions on $\underline{s}$. By a partition of unity argument, 
we can assume that $\phi$ has support in a local  chart on $\overline{\mathcal{M}}_{0,S}$.
For this, let  $(z_1,\ldots,z_n)$ denote local coordinates
$\overline{\mathcal{M}}_{0,S}$ taking values in a polydisk $\Delta^n=\{|z_i|<1\}$, with respect to which $\partial\overline{\mathcal{M}}_{0,S}$ is a union of coordinate hyperplanes $\{z_i=0\}$.  We can assume the support of $\phi$ is contained within this polydisk.  By renumbering the coordinates if necessary, we let $z_1,\ldots,z_r$ denote the equations of  components of $\partial\overline{\mathcal{M}}_{0,S}$ at finite distance and $z_{r+1},\ldots,z_{r+s}$ denote the coordinates corresponding to components at infinite distance (relative to the fixed dihedral structure on $S$). In these coordinates we have
$$g_{\underline{s}} = a \prod_{i=1}^r |z_i|^{2s_i} \prod_{i=r+1}^{r+s}|z_i|^{2t_i}$$
where $a$ is a smooth function on $\Delta^n$, $s_i$ is one of the Mandelstam variables $s_c$, and $t_i$ is a linear combination of Mandelstam variables $s_c$ with coefficients in $\{0,1,-1\}$. 
Since $\nu_S$ only has simple poles located along the divisors at finite distance, we can  write
$$\nu_S = b \, \frac{dz_1}{z_1}\wedge \cdots \wedge \frac{dz_r}{z_r}\wedge dz_{r+1}\wedge \cdots \wedge dz_n$$
where $b$ is a smooth function on $\Delta^n$. Since $\phi$ has degree $n-1$, by linearity in $\phi$,  we can assume that there is a single coordinate $z_p$ such that $dz_p$ does not appear in $\phi$. There are three cases to consider, depending on whether this coordinate is away from $\partial\overline{\mathcal{M}}_{0,S}$, or corresponds to a component at finite or infinite distance. In each case we use the Leibniz rule to compute $d(g_{\underline{s}}\nu_S\wedge\overline{\phi})$.
\begin{enumerate}[1)]
\item  We have  $p>r+s$. Without loss of generality, assume that $p=n$. Then $\phi$ has the form
$$\phi = c\, \frac{dz_1}{z_1}\wedge \cdots \wedge \frac{dz_{r+s}}{z_{r+s}}\wedge dz_{r+s+1}\wedge \cdots \wedge dz_{n-1}$$
where $c$ is a smooth function on $\CC^n$. Since $c$ has  support in $\Delta^n$, it is enough to show that the following integral vanishes:
$$\int_{\Delta^{n-1}} \bigwedge_{i=1}^r|z_i|^{2s_i}\frac{dz_i\wedge d\overline{z}_i}{z_i\,\overline{z_i}}\wedge \bigwedge_{i=r+1}^{r+s} |z_i|^{2t_i}\frac{dz_i\wedge d\overline{z}_i}{\overline{z_i}} \wedge \bigwedge_{i=r+s+1}^{n-1} dz_i\wedge d\overline{z_i} \left(\int_{|z_n|\leq 1} d(f\, dz_n)\right)\ ,$$
where $f=ab\overline{c}$ is a smooth function on $\CC^n$ with support in $\Delta^n$. The inner integral vanishes by Stokes' theorem  and we are done.
\item We have $p \leq r$. Without loss of generality, let $p=r$. Then $\phi$ has the form
$$\phi = c\, \frac{dz_1}{z_1}\wedge \cdots \wedge \frac{dz_{r-1}}{z_{r-1}}\wedge\frac{dz_{r+1}}{z_{r+1}}\wedge\cdots\wedge  \frac{dz_{r+s}}{z_{r+s}}\wedge dz_{r+s+1}\wedge \cdots \wedge dz_{n}$$
where $c$ is a smooth function on $\CC^n$ with support in $\Delta^n$. It is enough to show that  the following integral vanishes:
$$\int_{\Delta^{n-1}} \bigwedge_{i=1}^{r-1}|z_i|^{2s_i}\frac{dz_i\wedge d\overline{z}_i}{z_i\,\overline{z_i}}\wedge \bigwedge_{i=r+1}^{r+s} |z_i|^{2t_i}\frac{dz_i\wedge d\overline{z}_i}{\overline{z_i}} \wedge \bigwedge_{i=r+s+1}^{n} dz_i\wedge d\overline{z_i} \left(\int_{|z_r|\leq 1} d\left(f\,|z_r|^{2s_r} \frac{dz_r}{z_r}\right)\right)\ ,$$
where $f=ab\overline{c}$ is a smooth function on $\CC^n$ with support in $\Delta^n$. The inner integral is the limit as $\varepsilon$ goes to zero of the same integral over $\{\varepsilon\leq |z_r|\leq 1\}$, which by Stokes and changing  to polar coordinates evaluates to
$$\int_{|z_r|=\varepsilon}f\, |z_r|^{2s_r}\frac{dz_r}{z_r} =  \varepsilon^{2s_r}\int_{0}^{2\pi} i f\,d\theta\ .$$
This tends to zero when $\varepsilon$ goes to zero if $\mathrm{Re}(s_r)>0$. 
\item We have $r<p\leq r+s$. We can assume that $p=r+s$ and so $\phi$ has the form
$$\phi = c\, \frac{dz_1}{z_1}\wedge \cdots \wedge  \frac{dz_{r+s-1}}{z_{r+s-1}}\wedge dz_{r+s+1}\wedge \cdots \wedge dz_{n}$$
where $c$ is a smooth function on $\CC^n$ with support in $\Delta^n$. We want to prove that the following integral vanishes:
$$\int_{\Delta^{n-1}} \bigwedge_{i=1}^{r}|z_i|^{2s_i}\frac{dz_i\wedge d\overline{z}_i}{z_i\,\overline{z_i}}\wedge \bigwedge_{i=r+1}^{r+s-1} |z_i|^{2t_i}\frac{dz_i\wedge d\overline{z}_i}{\overline{z_i}} \wedge \bigwedge_{i=r+s+1}^{n} dz_i\wedge d\overline{z_i} \left(\int_{|z_{r+s}|\leq 1} d\left(f\,|z_{r+s}|^{2t_{r+s}} dz_{r+s}\right)\right)\ ,$$
where $f=ab\overline{c}$ is a smooth function on $\CC^n$ with support in $\Delta^n$. The inner integral is the limit as $\varepsilon$ goes to zero of the same integral over $\{\varepsilon\leq |z_{r+s}|\leq 1\}$, which by Stokes and changing to polar coordinates evaluates to
$$\int_{|z_{r+s}|=\varepsilon}f\, |z_{r+s}|^{2t_{r+s}}dz_{r+s} = \varepsilon^{2t_{r+s}+1}\int_0^{2\pi}i\, f\,e^{i\theta}d\theta \ .$$
This goes to zero when $\varepsilon$ goes to zero if $2\,\mathrm{Re}(t_{r+s})>-1$, which is a consequence of the inequalities stated in Proposition \ref{propIclosedconverges} as in the proof of that Proposition.
\end{enumerate}
Therefore, the integral of  $g_{\underline{s}}\nu_S\wedge\overline{\nabla_{\underline{s}}\phi}$ vanishes if the inequalities stated in Proposition \ref{propIclosedconverges} hold: namely $-\frac{1}{2}<\mathrm{Re}(s_c)<\frac{1}{N^2}$ for all $c$ and $\mathrm{Re}(s_c)>0$ for every divisor $D_c$ along which $\omega$ has a pole.  
\end{proof}

\begin{example} \label{ex: beta double integral} 
In the case of the beta function, with $\nu=-\nu_S=-\frac{dx}{x(1-x)}$, Theorem \ref{theoremsvforlogwithcoeffs} reads:
$$\langle[\nu],\s[\nu]\rangle^\dR= \frac{1}{2\pi i} \int_{\PP^1(\CC)} |z|^{2s} |1-z|^{2t}  \nu \wedge \overline{\nu}  \,   = \frac{1}{2\pi i}\int_{\PP^1(\CC)}  |z|^{2s-2} |1-z|^{2t-2} \, dz \wedge d \overline{z}\ ,$$ which equals    $- \beta_{\CC}(s,t) $.
\end{example}

\subsection{Double copy formula} 
By equating the two expressions for the single-valued period given in Proposition \ref{lem2ndDC} and Theorem \ref{theoremsvforlogwithcoeffs} we obtain an equality that expresses a volume integral as a quadratic expression in ordinary period integrals. 

\begin{corollary}  \label{theorem2ndDC}  Under the assumptions of Theorem \ref{theoremsvforlogwithcoeffs} we have the equality:
$$
\int_{\overline{\mathcal{M}}_{0,S}(\CC)}  \left(\prod_{i<j} |p_j-p_i|^{2s_{ij}}\right) \nu_S \wedge \overline{\omega}     \quad  = \sum_{\substack{[\sigma\otimes f_{-\underline{s}}] \\ [\tau\otimes f_{\underline{s}}]}}   \langle [\tau\otimes f_{\underline{s}}]^{\vee}, [\sigma\otimes f_{-\underline{s}}]^{\vee}\rangle^{\B}  \int_{\tau} f_{\underline{s}}\,\nu_S \,\int_{\overline{\sigma}} f_{\underline{s}}\,\omega     \ ,$$
where $[\sigma\otimes f_{-\underline{s}}]$ and $[\tau\otimes f_{\underline{s}}]$ range over a basis of $H_n(\mathcal{M}_{0,S}, \mathcal{L}_{s})$ and $H_n(\mathcal{M}_{0,S}, \mathcal{L}_{-\underline{s}})$ respectively, and $[\sigma\otimes f_{-\underline{s}}]^{\vee}, [\tau\otimes f_{\underline{s}}]^{\vee}$ are the dual bases.
\end{corollary}

This formula bears very close similarity to the  KLT formula \cite{klt}, and makes it apparent that the  `KLT kernel' 
should coincide with the Betti intersection pairing on  twisted 
\emph{cohomology}, which is the inverse transpose of the intersection pairing on  twisted \emph{homology}. 
Indeed, Mizera has shown in \cite{mizera} that the KLT kernel  indeed coincides with the inverse transpose matrix of the intersection pairing.

\begin{example} \label{ex: beta double copy}
Examples \ref{ex: sv beta pairings} and \ref{ex: beta double integral} give rise to the equality $$\beta_\CC(s,t) = -\frac{1}{2\pi i}\frac{2}{i}\frac{\sin(\pi s)\sin(\pi t)}{\sin(\pi (s+t))}\beta(s,t)^2\ ,$$
which is an instance of Corollary \ref{theorem2ndDC}.
\end{example}

\appendix

\section{The Parke--Taylor basis}\label{appendix}

    In this appendix we prove the following theorem, which is Theorem \ref{thm: Parke Taylor} from the main body of the paper.

    \begin{theorem}\label{thm: Parke Taylor appendix}
    Assume that the $s_{i,j}$ are generic in the sense of \eqref{assumptions}. A basis of $H^n(\mathcal{M}_{0,n+3},\nabla_{\underline{s}})$ is provided by the classes of the differential forms 
    \begin{equation}\label{eq: Parke Taylor form appendix}
    \frac{dt_1\wedge\cdots \wedge dt_n}{\prod_{k=1}^{n+1} (t_{\sigma(k)}-t_{\sigma(k-1)})}
    \end{equation}
    for permutations $\sigma\in\Sigma_n$, where we set $t_{\sigma(0)}=0$ and $t_{\sigma(n+1)}=1$.
    \end{theorem}

  \subsection{Working with the configuration space of points in \texorpdfstring{$\CC$}{C}}\label{par: appendix configuration space}
  
  We consider the configuration space 
  $$\mathrm{Conf}(n+2,\CC)=\{(z_0,z_1,\ldots,z_{n+1})\in \CC^{n+2}\; ,\; z_i\neq z_j\}$$ 
  and its (rational algebraic de Rham) cohomology algebra 
  $$A^\bullet=H^\bullet_\dR(\mathrm{Conf}(n+2,\CC))\ .$$ 
  By a classical result of Arnol'd \cite{arnold}, it is generated by the (classes of the) forms 
  $$\omega_{i,j}=d\log(z_j-z_i)\ .$$ 
  The natural (diagonal) $\CC^*$-action on $\mathrm{Conf}(n+2,\CC)$ induces a linear map in cohomology:
  $$\partial:A^\bullet\rightarrow A^{\bullet-1}\otimes H^1_\dR(\CC^*) \simeq A^{\bullet-1}\ .$$
  Compatibility with the cup-product and the Koszul sign rule implies that it is a graded derivation, i.e., that it satisfies the Leibniz rule
  $$\partial(ab)=\partial(a)b+(-1)^{|b|}a\,\partial(b)$$
  for $b$ homogeneous of degree $|b|$. 
  It is uniquely determined by $\partial(\omega_{i,j})=1$ for all $i,j$ and is given more generally by the formula:
  $$\partial(\omega_{i_1,j_1}\wedge\cdots \wedge\omega_{i_r,j_r}) = \sum_{k=1}^r(-1)^{k-1}\omega_{i_1,j_1}\wedge\cdots\wedge\widehat{\omega_{i_k,j_k}}\wedge\cdots \wedge\omega_{i_r,j_r}\ .$$
  We will make use of the following relations. 
  
  \begin{lemma}\label{lem: OS relations appendix}
  We have
  $$\omega_{i_1,j_1}\wedge\cdots \wedge\omega_{i_r,j_r} = 0 \;\;\;\;\mbox{ and }\;\;\;\; \partial(\omega_{i_1,j_1}\wedge\cdots \wedge\omega_{i_r,j_r})=0$$
  if there is a cycle in the graph with vertices $\{0,\ldots,n+1\}$ and edges $\{i_1,j_1\},\ldots,\{i_r,j_r\}$.
  \end{lemma}
  
  \begin{proof}
  The second relation follows from the first, which is already satisfied at the level of differential forms.
  \end{proof}
  
  A special case is the classical Arnol'd relations $\partial(\omega_{i,j}\wedge\omega_{i,k}\wedge\omega_{j,k})=0$, which generate all the relations among the generators $\omega_{i,j}$ in the algebra $A^\bullet$ \cite{arnold}. 
  
  The (homological) complex $(A^\bullet,\partial)$ is contractible. A contracting homotopy is given, for instance, by multiplication by the generator $\omega_{0,1}$.\medskip
  
  There is a natural quotient morphism (by the affine group $\CC\rtimes \CC^*$):
  $$\mathrm{Conf}(n+2,\CC)\twoheadrightarrow \mathcal{M}_{0,n+3} \;\; , \;\; (z_0,\ldots,z_{n+1})\mapsto (z_1,\ldots,z_{n+1},\infty,z_0)\ ,$$
  defined so that the simplicial coordinates $t_0=0, t_1,\ldots, t_n, t_{n+1}=1$ on $\mathcal{M}_{0,n+3}$ are related to the coordinates $z_i$ by the formula
  $$t_i=\frac{z_i-z_0}{z_{n+1}-z_0}\ \cdot$$
  The pullback by this quotient identifies $\Omega^\bullet = H^\bullet_\dR(\mathcal{M}_{0,n+3})$ with the subalgebra $\ker(\partial)\subset A^\bullet$. Since $(A^\bullet,\partial)$ is contractible we have $\ker(\partial)=\mathrm{Im}(\partial)$ and we get a short exact sequence:
  \begin{equation}\label{eq: short exact sequence appendix}
  0\longrightarrow \Omega^\bullet\longrightarrow A^\bullet \stackrel{\partial}{\longrightarrow} \Omega^{\bullet-1}\longrightarrow 0 \ .
  \end{equation}
  
   We now move to cohomology with coefficients. A solution $\underline{s}$ to the momentum conservation equations is completely determined by the complex numbers $s_{i,j}$ for $0\leq i<j\leq n+1$, where we set $s_{0,j}:=s_{j,n+3}$ as in \S\ref{sect: StringAmpSimplicial},  subject to the single relation:
   \begin{equation}\label{eq: momentum conservation appendix}
   \sum_{0\leq i<j\leq n+1}s_{ij}=0\ .
   \end{equation}
   Denote the pullback of the Koba--Nielsen form to $\mathrm{Conf}(n+2,\CC)$ by the same symbol:
   $$\omega_{\underline{s}}=\sum_{0\leq i<j\leq n+1}s_{i,j}\,\omega_{i,j}\ .$$
   In the remainder of this appendix we extend the scalars from $\QQ$ to the field $\QQ_{\underline{s}}^\dR=\QQ(s_{i,j})$ generated by the $s_{i,j}$ inside $\CC$. In order to keep the notations simple, we continue to write $A^\bullet$ and $\Omega^\bullet$ for $A^\bullet\otimes_\QQ\QQ_{\underline{s}}^\dR$ and $\Omega^\bullet\otimes_\QQ\QQ_{\underline{s}}^\dR$, respectively.
   
   The short exact sequence \eqref{eq: short exact sequence appendix} induces a short exact sequence of complexes :
   $$0\longrightarrow (\Omega^\bullet,\wedge\omega_{\underline{s}})\longrightarrow (A^\bullet,\wedge\omega_{\underline{s}}) \stackrel{\partial}{\longrightarrow} (\Omega^{\bullet-1},\wedge\omega_{\underline{s}})\longrightarrow 0 \ .$$
   This is because $\partial(\omega_{\underline{s}})=0$, which follows from   equation \eqref{eq: momentum conservation appendix}. Assume that $\underline{s}$ is generic \eqref{assumptions}. Then $H^k(\Omega^\bullet,\wedge\omega_{\underline{s}})=0$ for $k\neq n$ (Remark \ref{rem: artin vanishing twisted}) and the long exact sequence in cohomology shows that the morphism induced by $\partial$ in top degree cohomology
   $$H^{n+1}(A^\bullet,\wedge\omega_{\underline{s}}) \stackrel{\partial}{\longrightarrow} H^n(\Omega^\bullet,\wedge\omega_{\underline{s}})$$
   is an isomorphism. An easy computation (see, e.g., \cite[Claim 3.1]{mizera}) shows that the Parke--Taylor form \eqref{eq: Parke Taylor form appendix} from Theorem \ref{thm: Parke Taylor appendix} is, up to the sign $\mathrm{sgn}(\sigma)$, the image by $\partial$ of the form
   $$\omega_\sigma = \omega_{0,\sigma(1)}\wedge\omega_{\sigma(1),\sigma(2)}\wedge\cdots \wedge\omega_{\sigma(n-1),\sigma(n)}\wedge\omega_{\sigma(n),n+1} \;\;\;\in A^{n+1}\ ,$$
   for $\sigma\in\Sigma_n$. Thus, a restatement of Theorem \ref{thm: Parke Taylor appendix} is as follows:
   
   \begin{theorem}\label{thm: Parke Taylor appendix bis}
   Assume that the $s_{i,j}$ are generic in the sense of \eqref{assumptions}. Then a basis of 
   $$H^{n+1}(A^\bullet,\wedge\omega_{\underline{s}}) = A^{n+1} / (A^n\wedge\omega_{\underline{s}})$$
   is provided by the classes of the forms $\omega_\sigma$, for permutations $\sigma\in\Sigma_n$.
   \end{theorem}
   
   The rest of this appendix is devoted to the proof of this theorem.
   
   \subsection{A basis of \texorpdfstring{$A^{n+1}$}{An+1}} 
   
    We slightly extend the notation $\omega_\sigma$ to permutations $\sigma\in\Sigma_{n+1}$:
    $$\omega_\sigma := \omega_{0,\sigma(1)}\wedge\omega_{\sigma(1),\sigma(2)}\wedge\cdots \wedge\omega_{\sigma(n-1),\sigma(n)}\wedge\omega_{\sigma(n),\sigma(n+1)} \;\;\;\in A^{n+1}\ .$$
    
    \begin{proposition}
    The cohomology classes $\omega_\sigma$, for $\sigma\in\Sigma_{n+1}$, are a basis of $A^{n+1}$.
    \end{proposition}
   
    \begin{proof}
    The dimension of $A^{n+1}$ is $(n+1)!$ by \cite{arnold}, so it is enough to prove that the $\omega_\sigma$ are linearly independent in $A^{n+1}$. 
    For $1\leq i\leq n+1$ we have a residue morphism along $\{z_0=z_i\}$: 
    $$\mathrm{Res}_i:H^{n+1}_\dR(\mathrm{Conf}(n+2,\CC))\longrightarrow H^n_\dR(\mathrm{Conf}(n+1,\CC))\ ,$$
    where in the target $\mathrm{Conf}(n+1,\CC)$ consists of tuples $(z_0,\ldots,\widehat{z_i},\ldots,z_{n+1})$. It satisfies:
    $$\mathrm{Res}_i(\omega_\sigma) = \begin{cases} 0 & \mbox{if }\; \sigma(1)\neq i\ ; \\ \omega_\sigma & \mbox{if }\; \sigma(1)=i \ ,\end{cases}$$
    where in the second case we implicitly use the natural bijection 
    $$\Sigma_n\simeq \mathrm{Bij}\,(\{2,\ldots,n+1\},\{1,\ldots,\widehat{i},\ldots,n+1\}) \ .$$
     The result follows by induction on $n$, since the case $n=0$ is trivial.
    \end{proof}
    
    We now let $F^kA^{n+1}$ denote the subspace of $A^{n+1}$ spanned by the basis elements $\omega_\sigma$, for $\sigma\in\Sigma_{n+1}$ such that $\sigma^{-1}(n+1)\geq k$. It forms a decreasing filtration, where $F^1A^{n+1}=A^{n+1}$, $F^{n+2}A^{n+1}=0$, and $F^{n+1}A^{n+1}$ is spanned by the $\omega_\sigma$ for $\sigma\in\Sigma_n$. 
    
    \begin{lemma}\label{lem: technical appendix}
    Any element of $A^{n+1}$ which is the exterior product of a form with an  element of the following type
    $$\omega_{0,i_1}\wedge\omega_{i_1,i_2}\wedge \cdots \wedge\omega_{i_{k-2},i_{k-1}}\wedge\omega_{i_{k-1},n+1}$$
    lies in $F^kA^{n+1}$.
    \end{lemma}
    
    \begin{proof}
    To a graph $\Gamma$ with set of vertices $\{0,\ldots,n+1\}$ and $(n+1)$ edges we associate a monomial $\omega_\Gamma$ (well-defined up to a sign) obtained by multiplying the generators $\omega_{i,j}$ together for $\{i,j\}$ an edge of $\Gamma$. What we need to prove is that $\omega_\Gamma\in F^kA^{n+1}$ if $\Gamma$ contains a path of length $k$ between the vertices $0$ and $n+1$. If $\Gamma$ contains a cycle then $\omega_\Gamma=0$ by Lemma \ref{lem: OS relations appendix} and there is nothing to prove. Since $\Gamma$ has $n+2$ vertices and $n+1$ edges we can thus assume that it is a tree. If we choose $0$ to be the root of $\Gamma$ then all edges inherit a preferred orientation (from the root to the leaves). We let $\delta_\Gamma(i)$ denote the distance between the vertices $0$ and $i$ in $\Gamma$ and set
    $$\delta_\Gamma=\sum_{i=1}^{n+1}\delta_\Gamma(i)\ .$$
    We have $\delta_\Gamma\leq 1+2+\cdots +(n+1)$ and the equality holds exactly when $\Gamma$ is linear, i.e., has only one leaf. We prove by decreasing induction on $\delta_\Gamma$ the statement: $\omega_\Gamma\in F^kA^{n+1}$ if $\delta_\Gamma(n+1)\geq k$. If $\Gamma$ is linear then $\omega_\Gamma$ is one of the basis elements $\omega_\sigma$ for some $\sigma\in \Sigma_{n+1}$ such that $\sigma^{-1}(n+1)\geq k$ and thus $\omega_\Gamma\in F^kA^{n+1}$ by definition. Now assume that $\Gamma$ is not linear and let $a$ be a vertex of $\Gamma$ with two children $b_1$ and $b_2$ (which means that there is an edge from $a$ to $b_1$ and an edge from $a$ to $b_2$ in $\Gamma$). Then we can use the Arnol'd relation
    $$\omega_{a,b_1}\wedge\omega_{a,b_2} = \omega_{a,b_1}\wedge \omega_{b_1,b_2} - \omega_{a,b_2}\wedge\omega_{b_1,b_2}$$
    and deduce a relation
    $$\omega_\Gamma=\pm\, \omega_{\Gamma'_2}\,\pm\omega_{\Gamma'_1}$$
    where $\Gamma'_s$ is obtained from $\Gamma$ by deleting the edge $\{a,b_s\}$ and adding the edge $\{b_1,b_2\}$, for $s\in\{1,2\}$. One easily sees that we have, for all $i\in\{1,\ldots,n+1\}$, $\delta_{\Gamma'_s}(i)\geq \delta_\Gamma(i)$, and $\delta_{\Gamma'_s}>\delta_\Gamma$, for $s\in\{1,2\}$. One can thus apply the induction hypothesis to $\Gamma'_1$ and $\Gamma'_2$, which completes the induction step and the proof.
    \end{proof}
   
   \subsection{The proof} 
   
   The filtration $F^k$ on $A^{n+1}$ induces a decreasing filtration denoted by the same symbol on the quotient $H^{n+1}(A^\bullet,\wedge\omega_{\underline{s}}) = A^{n+1}/(A^n\wedge\omega_{\underline{s}})$.
   
   \begin{proposition}\label{prop: filtration appendix}
   For every $k\in\{1,\ldots,n\}$ we have $$F^kH^{n+1}(A^\bullet,\wedge\omega_{\underline{s}})\subset F^{k+1}H^{n+1}(A^\bullet,\wedge\omega_{\underline{s}})\ .$$
   \end{proposition}
   
   \begin{proof}
   Let us fix a permutation $\sigma\in\Sigma_n$ such that $\sigma^{-1}(n+1)=k\in\{1,\ldots,n\}$. We set $i_s=\sigma(s)$ for every $s\in\{1,\ldots,n+1\}$ and $i_0=0$, so that we have
   $$\omega_\sigma=X\wedge Y$$
   where we set:
   $$X=\omega_{i_0,i_1}\wedge\omega_{i_1,i_2}\wedge\cdots\wedge \omega_{i_{k-1},n+1} \;\;\mbox{ and }\;\; Y = \omega_{n+1,i_{k+1}}\wedge\cdots\wedge\omega_{i_{n-1},i_n}\wedge\omega_{i_n,i_{n+1}}\ .$$
   We have the relation, in $H^{n+1}(A^\bullet,\omega_{\underline{s}})$:
   $$X\wedge \omega_{\underline{s}}\wedge\partial(Y)=0\ .$$
   We note that if $\{i,j\}=\{i_a,i_b\}$ for $0\leq a<b\leq k$ then we have $X\wedge \omega_{i,j}=0$ by Lemma \ref{lem: OS relations appendix}. Let $P$ denote the set of remaining pairs of indices. We then get the relation:
   $$\sum_{\{i,j\}\in P}s_{i,j}\,X\wedge\omega_{i,j}\wedge \partial(Y)=0\ .$$
   The Leibniz rule gives    
   $$\omega_{i,j}\wedge\partial(Y) = Y-\partial(\omega_{i,j}\wedge Y)$$
   and we can rewrite the relation as:
   \begin{equation}\label{eq: somewhere proof appendix}
   \left(\sum_{\{i,j\}\in P}s_{i,j}\right)\omega_\sigma = \sum_{\{i,j\}\in P} s_{i,j}\, X\wedge\partial(\omega_{i,j}\wedge Y)\ .
   \end{equation}
   We now claim that the term $X\wedge\partial(\omega_{i,j}\wedge Y)$ is in $F^{k+1}$ for all $\{i,j\}\in P$. There is one easy case: if $\{i,j\}=\{i_a,i_b\}$ for $k\leq a<b\leq n+1$ then we have $\partial(\omega_{ij}\wedge Y)=0$ by Lemma \ref{lem: OS relations appendix}. Thus, we only have to treat the case where $\{i,j\}=\{i_a,i_b\}$ for $0\leq a\leq k$ and $k+1\leq b\leq n+1$. We proceed by decreasing induction on $a$. If $a=k$ then the first case applies and we are done. If $a<k$ then the Leibniz rule implies:
   $$X\wedge\partial(\omega_{i_a,i_b}\wedge Y)-X\wedge \partial(\omega_{i_{a+1},i_b}\wedge Y) = X\wedge \partial(\omega_{i_a,i_b}\wedge \omega_{i_{a+1},i_b})\wedge \partial(Y)\ .$$

   We now set 
   $$X'=\omega_{i_0,i_1}\wedge\cdots\wedge\omega_{i_{a-1},i_a}\wedge\omega_{i_{a+1},i_{a+2}}\wedge\cdots\wedge \omega_{i_{k-1},n+1}$$
   so that we have $X=\pm\omega_{i_a,i_{a+1}}\wedge X'$. We thus get
   $$X\wedge\partial(\omega_{i_a,i_b}\wedge Y)-X\wedge \partial(\omega_{i_{a+1},i_b}\wedge Y) = \pm X'\wedge \omega_{i_a,i_b}\wedge \omega_{i_{a+1},i_b}\wedge\partial(Y)\ ,$$

   where we have used the Arnol'd relation $\partial(\omega_{i_a,i_{a+1}}\wedge\omega_{i_a,i_b}\wedge \omega_{i_{a+1},i_b})=0$ in the form:  
   $ \omega_{i_a,i_{a+1}}\wedge \partial(\omega_{i_a,i_b}\wedge \omega_{i_{a+1},i_b})= \pm \omega_{i_a,i_b}\wedge \omega_{i_{a+1},i_b}$. 
   Now  Lemma \ref{lem: technical appendix} applied to $\left(X'\wedge \omega_{i_a,i_b}\wedge \omega_{i_{a+1},i_b}\right)\wedge\partial(Y)$ and the induction hypothesis respectively imply that
   $$X'\wedge \omega_{i_a,i_b}\wedge \omega_{i_{a+1},i_b}\wedge\partial(Y) \;\;\mbox{ and }\;\; X\wedge\partial(\omega_{i_{a+1},i_b}\wedge Y)$$
   are in $F^{k+1}H^{n+1}(A^\bullet,\omega_{\underline{s}})$, which implies that it is also the case for $X\wedge\partial(\omega_{i_a,i_b}\wedge Y)$. This concludes the induction step and the proof by induction. Returning to \eqref{eq: somewhere proof appendix} we see that we have $S\,\omega_\sigma\in F^{k+1}$ for 
   $$S=\sum_{\{i,j\}\in P}s_{i,j} = -\sum_{0\leq a<b\leq k}s_{i_a,i_b}\ ,$$
   where we have used equation \eqref{eq: momentum conservation appendix}. Thus $S\neq 0$ by the genericity assumption, and $\omega_\sigma\in F^{k+1}$, which finishes the proof of the proposition.
   \end{proof}
   
   We can now conclude with the proof of Theorem \ref{thm: Parke Taylor appendix bis}. By Proposition \ref{prop: filtration appendix} we have 
   $$H^{n+1}(A^\bullet,\omega_{\underline{s}})=F^1H^{n+1}(A^\bullet,\omega_{\underline{s}}) = F^{n+1}H^{n+1}(A^\bullet,\omega_{\underline{s}})$$
   which is spanned by the $\omega_\sigma$ for $\sigma\in\Sigma_n$. Since the dimension of $H^{n+1}(A^\bullet,\omega_{\underline{s}})$ is $n!$ if the $s_{i,j}$ are generic, this implies that the $\omega_\sigma$ are a basis, and Theorem \ref{thm: Parke Taylor appendix bis} is proved. Theorem \ref{thm: Parke Taylor appendix} follows as explained in \S\ref{par: appendix configuration space}.

\bibliographystyle{alpha}

\bibliography{biblio}

\end{document}